%% file: link_group.tex
\documentclass[reqno]{amsart}
\usepackage{amsfonts}
\usepackage{amsmath}
\usepackage{amsthm, amscd}
\usepackage{mathtools}
\usepackage{amssymb}
\usepackage{mathrsfs}
\usepackage{graphicx}
\usepackage{caption}
\usepackage{subcaption}
\usepackage{float}
\usepackage{xypic}
\usepackage[abs]{overpic}
\usepackage[alphabetic]{amsrefs}
\usepackage{etex}
\usepackage[colorlinks=true,linkcolor=black,anchorcolor=black,citecolor=black,filecolor=black,menucolor=black,runcolor=black,urlcolor=black]{hyperref}
\usepackage{listings}
\allowdisplaybreaks

\theoremstyle{plain}\newtheorem{Theorem}{Theorem}[section]
\theoremstyle{plain}\newtheorem{Corollary}[Theorem]{Corollary}
\theoremstyle{plain}\newtheorem{Lemma}[Theorem]{Lemma}
\theoremstyle{plain}\newtheorem{Definition}[Theorem]{Definition}
\theoremstyle{plain}\newtheorem{Proposition}[Theorem]{Proposition}
\theoremstyle{plain}
\theoremstyle{plain}\newtheorem{Conjecture}[Theorem]{Conjecture}
\theoremstyle{plain}\newtheorem*{Theorem*}{Theorem}
\theoremstyle{plain}\newtheorem{Question}[Theorem]{Question}

\theoremstyle{remark}\newtheorem{remark}[Theorem]{Remark}
\theoremstyle{remark}

\theoremstyle{plain}
\makeatletter
\newtheorem*{rep@theorem}{\rep@title}
\newcommand{\newreptheorem}[2]{
\newenvironment{rep#1}[1]{
 \def\rep@title{#2 \ref{##1}}
 \begin{rep@theorem}}
 {\end{rep@theorem}}}
\makeatother
\newreptheorem{Theorem}{Theorem}
\newreptheorem{Proposition}{Proposition}
\newreptheorem{Corollary}{Corollary}

\numberwithin{equation}{section}

\usepackage{lipsum}

\newcommand\blfootnote[1]{%
  \begingroup
  \renewcommand\thefootnote{}\footnote{#1}%
  \addtocounter{footnote}{-1}%
  \endgroup
}

\DeclareMathOperator{\rank}{rank}
\DeclareMathOperator{\II}{I}
\DeclareMathOperator{\ad}{ad}

\DeclareMathOperator{\AHI}{AHI}
\DeclareMathOperator{\AKh}{AKh}

\DeclareMathOperator{\muu}{\mu^{orb}}

\DeclareMathOperator{\SU}{SU}
\DeclareMathOperator{\SL}{SL}
\DeclareMathOperator{\SO}{SO}

\DeclareMathOperator{\Hom}{Hom}
\DeclareMathOperator{\pt}{pt}
\DeclareMathOperator{\id}{id}
\DeclareMathOperator{\lk}{lk}
\DeclareMathOperator{\HFK}{\widehat{HFK}}

\DeclareMathOperator{\KHI}{KHI}
\DeclareMathOperator{\trace}{trace}
\DeclareMathOperator{\len}{length}

\newcommand{\bC}{\mathbb{C}}

\newcommand{\bR}{\mathbb{R}}

\newcommand{\bZ}{\mathbb{Z}}
\newcommand{\bfi}{\mathbf{i}}
\newcommand{\bfj}{\mathbf{j}}
\newcommand{\bfk}{\mathbf{k}}

\DeclarePairedDelimiter{\ceil}{\lceil}{\rceil}

\author{Yi Xie}
\address{Beijing International Center for Mathematical Research, Peking University, Beijing 100871, China}
\email{yixie@pku.edu.cn}
\author{Boyu Zhang}
\address{Department of Mathematics, Princeton University, New Jersey 08544, USA}
\email{bz@math.princeton.edu}
\title{On meridian-traceless $\SU(2)$--representations of link groups }

\begin{document}

\begin{abstract}
Suppose $L$ is a link in $S^3$. We show that $\pi_1(S^3-L)$ admits an irreducible meridian-traceless representation in $\SU(2)$ if and only if $L$ is not the unknot, the Hopf link, or a connected sum of Hopf links. As a corollary, $\pi_1(S^3-L)$ admits an irreducible representation in $\SU(2)$ if and only if $L$ is neither the unknot nor the Hopf link. This result generalizes a theorem of Kronheimer and Mrowka to the case of links. 
\end{abstract}

\maketitle

\section{Introduction}
\input{introduction}

\section{The main result}
\label{sec_main_thm}
\input{main_theorem}

\section{Singular Instanton Floer homology}\label{sec_singular_instanton}
\input{singular_instanton}

\section{Topological properties from instantons}
\label{sec_top_properties}
This section obtains several topological properties for links with minimal $\II^\natural$ (see Definition \ref{def_minimal_I_natural}).  From now on, unless otherwise specified, all instanton homology groups are defined with $\bC$--coefficients, and the dimensions of instanton homology groups are computed over $\bC$. We will use $N(L)$ to denote the tubular neighborhood of a link $L$.

\subsection{Seifert surfaces of sublinks}
\input{Seifert_surface}

\subsection{Two and three-component links with minimal $\II^\natural$}
\input{exceptional_link}

\subsection{Seifert disks in minimal position}
\input{disk_minimal_position}

\section{Eliminate cycles of lengths $\neq$ 4}
\label{sec_cycles}
\input{cycles}

\section{Eliminate $G_0$}
\label{sec_eliminate_G0}
\input{eliminate_G0}

\section{Combinatorics}
\label{sec_combinatorics}
\input{combinatorics}


\appendix
\section{Lower bounds of $\II^\natural$ from Alexander polynomials}
\input{appendix}

\bibliographystyle{amsalpha}
\bibliography{references}

\end{document}

%% file: introduction.tex

\blfootnote{
The first author was supported by 
National Key R\&D Program of China SQ2020YFA0712801 and NSFC 12071005.}

The geometric and topological properties of a $3$--manifold can be strongly reflected by its fundamental group.
 One approach to understand the fundamental group is to study its representations into 
 linear groups such as
$\SU(2)$ and $\SL(2,\bC)$. Many topological invariants are defined by considering such representations, such as the Casson invariant  \cite{AM_Casson}, the Casson-Lin invariant  \cite{lin1992knot}, and the A-polynomial \cite{CCGLS_A-polynomial}.

The following is a natural question about the fundamental groups of  3-manifolds and is a question on Kirby's problem list \cite{Kirby}.
\begin{Question}[{\cite[3.105(A)]{Kirby}}]\label{question_main}
Suppose $Y$ is an integer homology 3-sphere with $\pi_1(Y)\neq 0$, does $\pi_1(Y)$ always admit an irreducible representation in $\SU(2)$?
\end{Question}

\begin{remark}
There exist closed 3-manifolds $Y$ which are not integer homology spheres such that $\pi_1(Y)$ admits no irreducible  representation in $\SU(2)$, see \cite[Theorem 1.6]{SZ_SU2_abelian} or the remark below \cite[3.105(A)]{Kirby}.
\end{remark}

Positive answers to Question \ref{question_main} have been obtained for
many different families of integer homology 3-spheres: the Dehn surgeries of non-trivial knots  with certain conditions on 
the surgery slopes \cite{KM_property_p,KM_Dehn_surgery,Lin_cyclic,BS_L_space_surgery}, the branched double covers of $S^3$ over knots with determinant $1$
\cite{Zentner_SU2_simple}, the splicing
of two knot complements \cite{Zentner_sl2}, 
Stein fillable $3$--manifolds \cite{BS_Stein_filling}, and 
toroidal integer homology
$3$--spheres \cite{LPZ_toroidal,BS_splicing}. A complete answer to 
the existence problem of irreducible $\SU(2)$--representations is obtained for   
non-hyperbolic geometric closed
 $3$--manifolds \cite{SZ_SU2_abelian}. There are also several related results for 4-manifolds (see \cite{daemi_ribbon,Taniguchi_seifert,Daemi_homology_cobordism,daemi2020chern}). 

It is also natural to ask about non-closed $3$--manifolds. The following question is on Kirby's problem list \cite[Problem 1.86]{Kirby} and is attributed to Cooper.
\begin{Question}[Cooper]\label{question_knot}
Suppose $K$ is a non-trivial knot in $S^3$, is it true that $\pi_1(S^3-K)$ always has an irreducible representation in $\SL(2,\bC)$?
\end{Question}

Question \ref{question_knot} was positively answered by Kronheimer and Mrowka \cite{KM_property_p} as a corollary of their proof for the Property P conjecture. 
Moreover, in \cite{KM:suture},  Kronheimer and Mrowka proved the following stronger result.

\begin{Theorem}[{\cite[Corollary 7.17]{KM:suture}}]
\label{thm_KM_R}
If $K$ is a non-trivial knot in $S^3$, then $\pi_1(S^3-K)$ admits an  
irreducible representation in $\SU(2)$ such that the meridian of $K$ is mapped to a traceless element in $\SU(2)$. 
\end{Theorem}

From now on, we will call the fundamental groups of knot complements as \emph{knot groups}, and call the fundamental groups of link complements as \emph{link groups}.\footnote{In the study of link homotopies, the meaning of the terminology \emph{link group} is different from ours (see \cite{Milnor_link_group}). } We say that an $\SU(2)$--representation of a knot group or a link group is \emph{meridian-traceless},
if the meridian of every component is mapped to a traceless element in $\SU(2)$ (see Definition \ref{def_meridian_traceless}). 

This paper generalizes Theorem \ref{thm_KM_R} to the case of links. The main result is the following theorem.

\begin{Theorem}\label{thm_main}
Suppose $L\subset S^3$ is a link. Then
$\pi_1(S^3-L)$ admits an irreducible meridian-traceless representation in $\SU(2)$ 
if and only if $L$
is not the unknot, the Hopf link,  or a connected sum of Hopf links.
\end{Theorem}

As a corollary, we obtain the following result about general $\SU(2)$--representations (without the meridian-traceless condition).

\begin{Corollary}\label{main_cor}
Suppose $L\subset S^3$ is a link. Then 
$\pi_1(S^3-L)$ admits an irreducible representation  in $\SU(2)$ 
if and only if $L$
is neither the unknot nor the Hopf link.
\end{Corollary}
\begin{remark}
It is already known that the only links with abelian link groups are the unknot and the Hopf link (see \cite[Section 6.3]{Kawauchi_survey}).  Corollary \ref{main_cor} recovers this result, and also shows that 
every non-abelian link group admits an irreducible $\SU(2)$--representation.
\end{remark}

Assuming the generalized Riemann hypothesis (GRH), Kuperberg proved that
the unknot recognition problem lies in the complexity class {\bf coNP} 
\cite[Theorem 1.1]{Kuperberg_unknot_coNP}. The same result was later proved by Lackenby \cite{Lackenby_unknot_coNP} without assuming the GRH. 
Kuperberg's proof is based on the
 positive answer to Question \ref{question_knot}. 
The argument was later used by Zentner to prove that 
the 3-sphere recognition problem is 
in  {\bf coNP} (assuming the GRH) \cite[Theorem 11.2]{Zentner_sl2}.
Following the same argument, we have 
the following corollary of Theorem \ref{thm_main}.
\begin{Corollary}\label{cor_Hopf_link_NP}
Suppose $L\subset S^3$ is 
given by a link diagram, and assume the generalized Riemann hypothesis. Then the
assertion that $L$ is the unknot, the Hopf link or a connected sum of Hopf links 
is in {\bf coNP}.
\end{Corollary}

\vspace{\baselineskip}

The proof of the main theorem uses  
the singular instanton Floer homology theory introduced by Kronheimer and Mrowka 
\cite{KM:Kh-unknot,KM:YAFT}. 
Let $L$ be a link, if the link group of $L$ does not admit any irreducible meridian-traceless $\SU(2)$--representations,  we will show that the singular instanton Floer homology group of $L$ has the minimal possible rank (Propositions \ref{prop_removing_component_rank_change} and \ref{prop_abelian_rep_imply_minimal_I_Z}). We will then prove several topological properties for links with minimal instanton Floer homology using the annular instanton Floer homology \cite{AHI} and the excision properties of  singular instanton Floer homology \cite{XZ:excision}. From these properties, we will show that either (1) $L$ is the unknot, the Hopf link, or a connected sum of Hopf links, or (2) the link group of $L$ has an irreducible meridian-traceless $\SU(2)$--representation. Since Case (2) contradicts the assumption, we conclude that $L$ is the unknot,  the Hopf link, or a connected sum of Hopf links.

This paper is organized as follows.
In Section \ref{sec_main_thm}, we give a proof of Theorem \ref{thm_main}  assuming several results that will be established in the later sections. The rest of the paper is devoted to establishing these results.  Section \ref{sec_singular_instanton} proves several general properties of instanton Floer homology. Section \ref{sec_top_properties} proves several topological properties for links with minimal $\II^\natural$ (see Definition \ref{def_minimal_I_natural}). Sections \ref{sec_cycles} and \ref{sec_eliminate_G0} then apply the topological properties from Section \ref{sec_top_properties} to show that the linking graphs (see Definition \ref{def_linking_graph}) of links with minimal $\II^\natural$ satisfy certain combinatorial properties.  Section \ref{sec_combinatorics} proves a purely combinatorial result that was used to finish the proof in Section \ref{sec_main_thm}. The appendix contains some technical computations that are used in the earlier sections.

Some of the results in Sections \ref{sec_top_properties} and \ref{sec_cycles} are similar to the ones in \cite{XZ:forest}.
However, 
unlike Khovanov homology (which is the object of study in \cite{XZ:forest}), the instanton Floer homology groups are usually difficult to compute. This  created additional difficulties that need to be resolved by Sections \ref{sec_eliminate_G0} and \ref{sec_combinatorics}. More discussions about this are given in Remark \ref{rmk_after_conj}.

\section*{Acknowledgements}
We would like to thank Fan Wei for helpful conversations about graph theory, especially for constructing the example in Remark \ref{rmk_graph}.

%% file: main_theorem.tex
Suppose $G$ is a group; a representation $\rho:G\to \SU(2)$  is called \emph{irreducible} if the image of $\rho$ is non-abelian. Otherwise, $\rho$ is called \emph{reducible} or \emph{abelian}.

Notice that the set of traceless matrices is a conjugacy class in $\SU(2)$.
Let $M$ be a path-connected space, $\rho:\pi_1(M)\to \SU(2)$ be a representation, and $\gamma$ be an oriented closed loop in $M$ \emph{without} a base point. Then $\gamma$ defines an element $[\gamma]\in \pi_1(M)$ up to conjugacy. We say that $\rho$ maps $\gamma$ to a traceless matrix if $\rho([\gamma])$ is traceless.

We introduce the following definition. 
\begin{Definition}
\label{def_meridian_traceless}
	Let $L\subset S^3$ be a link.  A representation $\rho:\pi_1(S^3-L)\to \SU(2)$ is called \emph{meridian-traceless} if it maps all meridians of $L$ to traceless matrices.
\end{Definition} 

The main result of this article is the following theorem.

\begin{repTheorem}{thm_main}
Suppose $L\subset S^3$ is a link. Then
$\pi_1(S^3-L)$ admits an irreducible meridian-traceless $\SU(2)$--representation 
if and only if $L$
is not the unknot, the Hopf link,  or a connected sum of Hopf links.
\end{repTheorem}

Theorem \ref{thm_main} has the following immediate corollary.
\begin{repCorollary}{main_cor}
Suppose $L\subset S^3$ is a link. Then 
$\pi_1(S^3-L)$ admits an irreducible $\SU(2)$--representation 
if and only if $L$
is neither the unknot nor the Hopf link.
\end{repCorollary}
\begin{proof}[Proof of Corollary \ref{main_cor} assuming Theorem \ref{thm_main}]
	If $L$ is the unknot or the Hopf link, then $\pi_1(S^3-L)$ is abelian, so it does not admit any irreducible $\SU(2)$--representation.
	
	If $L$ is a connected sum of at least two Hopf links, then there exist two components $K_1,K_2$ of $L$ such that $K_1\cup K_2$ is an unlink. Therefore, $\pi_1(S^3-K_1-K_2)$ is the free group generated by two elements and hence has an irreducible $\SU(2)$--representation. Since the embedding of $S^3-L$ into $S^3-K_1-K_2$ induces a surjection on $\pi_1$, the group $\pi_1(S^3-L)$ also has an irreducible $\SU(2)$--representation.
	
	If $L$ is not the unknot, the Hopf link, or a connected sum of Hopf links, then by Theorem \ref{thm_main}, the fundamental group $\pi_1(S^3-L)$ has a meridian-traceless irreducible $\SU(2)$--representation.
\end{proof}

The next definition constructs a graph from a link using the linking numbers. We will use $\lk(K,K')$ to denote the linking number of two knots $K$ and $K'$ in $S^3$. Notice that the parity of the linking number is independent of the orientations, so it makes sense to  refer to the parity of $\lk(K,K')$ without specifying the orientations. 
\begin{Definition}\label{def_linking_graph}
	Let $L=K_1\cup K_2\cup \cdots \cup K_n$ be an $n$--component link. Define the
	\emph{linking graph} of $L$ to be the finite simple graph with $n$ vertices $v_1,\cdots, v_n$, such that for all $i\neq j$, the vertices $v_i$ and $v_j$ are connected by an edge if and only if $\lk(K_i,K_j)$ is odd.
\end{Definition}

We will establish a relation between linking graphs and meridian-traceless $\SU(2)$--representations in Lemma \ref{lem_existence_ijk_rep} below. 
In the following, we will view $\SU(2)$ as the unit sphere in the space of quaternions, and use $\bfi,\bfj,\bfk\in \SU(2)$ to denote the matrices given by the respective quaternions. In other words, we have
$$
\bfi:=\begin{pmatrix} i & 0 \\ 0 & -i \end{pmatrix},\quad
\bfj:=\begin{pmatrix} 0 & 1 \\ -1 & 0 \end{pmatrix},\quad
\bfk:=\begin{pmatrix} 0 & i \\ i & 0 \end{pmatrix}.
$$

\begin{Lemma}\label{lem_existence_ijk_rep}
Let $L$ be a link in $S^3$ and let $G$ be its linking graph. Let $V$ be the vertex set of $G$.  Suppose there exists a non-constant 
map $\varphi: V\to \{\bfi,\bfj,\bfk\}$, such that 
for every $v\in V$, the image $\varphi(v)$ is commutative to 
$$
\prod_{\{w\in V|w \text{ is adjacent to }v\}} \varphi(w). 
$$
Then $\pi_1(S^3-L)$ admits an irreducible meridian-traceless $\SU(2)$--representation.
\end{Lemma}

\begin{proof}
	Fix a planar diagram for $L$ and fix an orientation.  By the Wirtinger presentation, $\pi_1(S^3-L)$ is generated by the positively-oriented meridians that goes around each arc of the diagram, and every crossing of the diagram gives a relation on the generators. Therefore, to define an $\SU(2)$--representation for  $\pi_1(S^3-L)$, one only needs to assign an element of $\SU(2)$ to each arc such that they are compatible at the crossings. 
	
	Write $L$ as $K_1\cup\cdots\cup K_n$, and let $v_1,\cdots,v_n$ be the corresponding vertices of $G$. 
	Assign each component of $L$ with one of three colors, such that $K_i$ and $K_j$ have the same color if and only if $\varphi(v_i)=\varphi(v_j)$. Then the assumption on the map $\varphi$ can be translated to the following statement about the diagram of $L$: if one travels along a component of $L$, then the number of times this component goes under an arc with a different color is even.
	
	We construct an irreducible meridian-traceless $\SU(2)$--representation for  $\pi_1(S^3-L)$ as follows.  To each arc that belongs to the component $K_i$, we assign $\varphi(v_i)$ or $-\varphi(v_i)$. The compatibility condition then requires that if a component goes under an arc with a different color, then the image of the arc changes sign; if it goes under an arc with the same color, then the sign remains unchanged. The assumption on $\varphi$ then implies that one can choose the signs compatibly and obtain an $\SU(2)$--representation.  By definition, this representation is meridian-traceless. Since $\varphi$ is not constant, it is also irreducible.
\end{proof}

Suppose $L$ is a link with $n$ components and let $p\in L$ be a base point. Let $\II^\natural(L,p)$ be the instanton homology invariant of $L$ introduced by Kronheimer-Mrowka in \cite[Section 4.3]{KM:Kh-unknot}.  We will prove in Proposition \ref{prop_removing_component_rank_change} that $\dim \II^\natural(L,p;\bC)\ge 2^{n-1}$, which motivates the following definition:
\begin{Definition}
\label{def_minimal_I_natural}
Suppose $L$ is a link with $n$ components. We say that $L$ has \emph{minimal $\II^\natural$}, if  $\dim \II^\natural(L,p;\bC)= 2^{n-1}$ 	
	for all base points $p\in L$.
\end{Definition}

The following result will be implied by Proposition \ref{prop_abelian_rep_imply_minimal_I_Z} (see Remark \ref{rmk_no_rep_imply_minimal_I_natural_Z}).

\begin{Proposition}\label{prop_only_abelian_rep_implies_minimal_I}
If $L$ is a link in $S^3$ such that 
$\pi_1(S^3-L)$ does not have any  
irreducible meridian-traceless $\SU(2)$--representations, then $L$
has minimal $\II^\natural$.  
\end{Proposition}

We say that a link $L'$ is a \emph{sublink} of $L$, if every component of $L'$ is also a component of $L$.  The following result will be implied by Proposition \ref{prop_removing_component_rank_change} (see Remark \ref{rmk_sublink_min_I}).

\begin{Proposition}
\label{prop_sublink_minimal_I}
If $L$ has minimal $\II^\natural$, then all the sub-links of $L$ have minimal $\II^\natural$. 
\end{Proposition}

By definition, a finite simple graph is called a \emph{forest} if it does not contain any cycle.  A graph is a forest if and only if is isomorphic to a disjoint union of trees. We recall the following definition from \cite{Kh-unlink}.
\begin{Definition}
	A link $L$ is called a \emph{forest of unknots}, if it is given by the disjoint unions and connected sums of Hopf links and unknots.
\end{Definition}

If $L$ is a forest of knots, then its linking graph is a forest. For every forest $G$, there is a unique forest of unknots $L$ up to isotopy such that the linking graph of $L$ is isomorphic to $G$.
Proposition \ref{prop_tree_implies_Hopf_connected_sum} below will be proved in Section \ref{sec_top_properties}.

\begin{Proposition}
\label{prop_tree_implies_Hopf_connected_sum}
Suppose $L$ has minimal $\II^\natural$, then the linking number of every pair of components of $L$ is either $0$ or $\pm 1$. If we further assume that the linking graph of $L$ is a forest, then $L$ is a forest of unknots.
\end{Proposition}

Proposition \ref{prop_cycle_length_neq_4} below will be proved in Section \ref{sec_cycles}.

\begin{Proposition}\label{prop_cycle_length_neq_4}
Suppose $L$ is a link with $n$ components such that 
\begin{enumerate}
	\item $n\ge 3$ and $n\neq 4$,
	\item the linking graph of $L$ is a cycle.
\end{enumerate}  
Then $L$ does not have minimal $\II^\natural$.
\end{Proposition}

\begin{remark}
	It is not clear to the authors whether Proposition \ref{prop_cycle_length_neq_4} still holds when $n=4$. See Proposition \ref{prop_L8n8}, Conjecture \ref{conj_l8n8}, and the discussion afterwards for more details.
\end{remark}

We use $G_0$ to denote the graph given by Figure \ref{fig_G0_graph}. The graph $G_0$ has $6$ vertices and $7$ edges. The following result will be proved in Section \ref{sec_eliminate_G0}.

\begin{figure}
\includegraphics[width=0.4\textwidth]{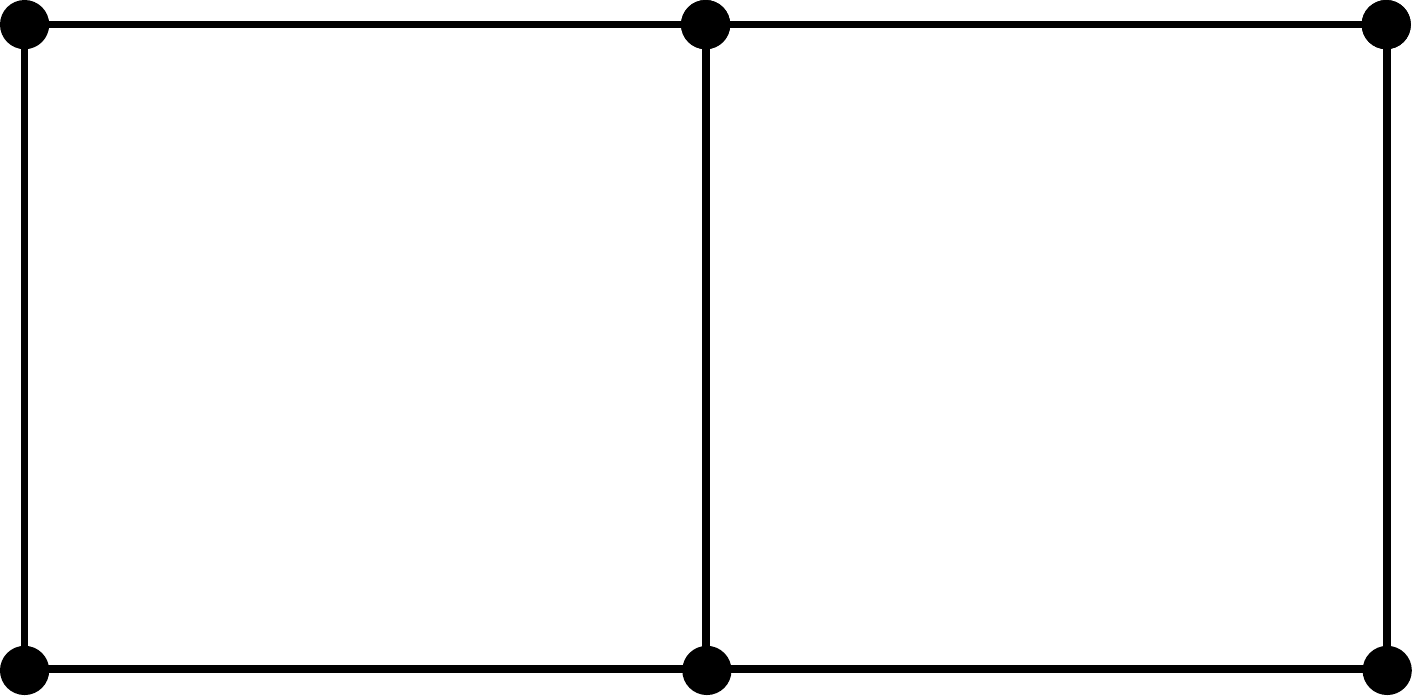}	
\caption{The graph $G_0$.}
\label{fig_G0_graph}
\end{figure}

\begin{Proposition}\label{prop_no_G0}
If $L$ is a link with 6 components such that its linking graph is isomorphic to $G_0$, then $L$ does not have minimal $\II^\natural$.
\end{Proposition}

We also introduce the following definition about graphs.

\begin{Definition}
\label{def_restriction_graph}
	Let $G$ be a finite simple graph with vertex set $V$. Suppose $V'$ is a subset of V. We define the \emph{restriction of $G$ to $V'$} to be the graph $G'$ such that 
	\begin{enumerate}
		\item the vertex set of $G'$ is $V'$,
		\item  the edge set of $G'$ consists of all the edges of $G$ whose vertices are both contained in $V'$.  
	\end{enumerate}	
	\end{Definition} 

Proposition \ref{prop_graph} below is a combinatorial result and will be proved in Section \ref{sec_combinatorics}.

\begin{Proposition}
\label{prop_graph}
 Suppose $G$ is a finite simple graph with vertex set $V$. Then at least one of the following conditions holds:
\begin{enumerate}
\item $G$ is a tree;
\item there exists $V'\subset  V$, such that the restriction of $G$ to $V'$ is a cycle of order $n$ with $n\ge 3$ and $n\neq 4$;
\item there exists $V'\subset  V$, such that the restriction of $G$ to $V'$ is isomorphic to $G_0$;
\item 
there exists a non-constant 
map $\varphi: V\to \{\bfi,\bfj,\bfk\}$, such that 
for every $v\in V$, the image $\varphi(v)$ is commutative to 
$$
\prod_{\{w\in V|w \text{ is adjacent to }v\}} \varphi(w). 
$$
\end{enumerate}
\end{Proposition}

\begin{remark} 
\label{rmk_graph}
	Let $G$ be the graph given by Figure \ref{fig_figures/K5-edge}.  Then $G$ is not a tree and does not satisfy Condition (3) above. It is also straightforward to check that $G$ does not satisfy Condition (4). Therefore, Condition (2) cannot be removed from the statement of Proposition \ref{prop_graph}. 
	It is not clear to the authors whether Condition (3) can be removed from the statement. 
\end{remark}

\begin{figure}
	\includegraphics[width=0.25\textwidth]{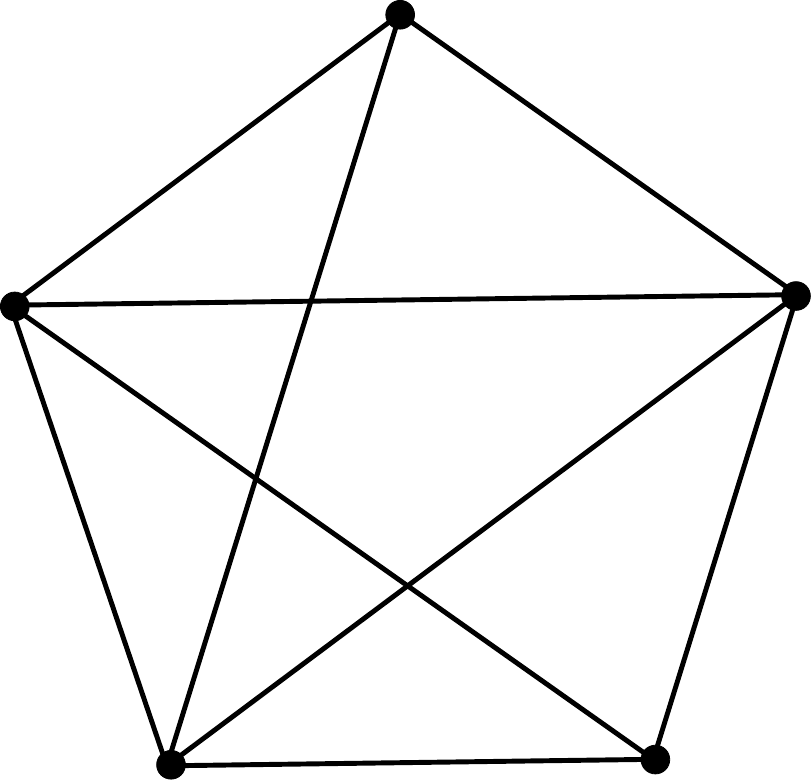}
	\caption{A graph with $5$ vertices.}
	\label{fig_figures/K5-edge}
\end{figure}

Assuming all the results above, we can now prove the main theorem.

\begin{proof}[Proof of Theorem \ref{thm_main}]
The proof consists of two parts.

(1) If $L$ is the Hopf link or a connected sum of Hopf links, we prove that $\pi_1(S^3-L)$ does not admit any irreducible traceless $\SU(2)$--representation. Suppose $L$ has $n$ components $K_1,\cdots,K_n$, then $\pi_1(S^3-L)$ has a presentation given as follows. There are $n$ generators $g_1,\cdots,g_n$, where $g_i$ is given by the meridian of $K_i$. For each pair $(i,j)$ such that $\lk(K_i,K_j)=\pm 1$, there is a relation $[g_i,g_j]=1$. As a consequence, if $\rho:\pi_1(S^3-L)\to\SU(2)$ is a meridian-traceless representation, then $\rho(g_i)=\pm \rho(g_j)$ whenever $\lk(K_i,K_j)=\pm 1$. Since the linking graph of $L$ is connected, $\rho$ is reducible. 

(2) Now suppose  $\pi_1(S^3-L)$ does not admit any  
irreducible meridian-traceless $\SU(2)$--representations, we prove that $L$ is the Hopf link or a connect sum of Hopf links.
By Propositions \ref{prop_only_abelian_rep_implies_minimal_I} and \ref{prop_sublink_minimal_I}, every sublink of $L$ has minimal $\II^\natural$. Let $G$ be the linking graph of $L$.  By Proposition \ref{prop_graph}, there are 4 cases: 
\begin{enumerate}
	\item [(i)] If $G$ is a tree, then Proposition \ref{prop_tree_implies_Hopf_connected_sum} implies that $L$ is the Hopf link or a connected sum of Hopf links.
	\item [(ii)] If $G$ satisfies Condition (2) of Proposition \ref{prop_graph}, then there is a contradiction by Proposition \ref{prop_cycle_length_neq_4}.
	\item [(iii)] If $G$ satisfies Condition (3) of Proposition \ref{prop_graph}, then there is a contradiction by Proposition \ref{prop_no_G0}.
	\item [(iv)] If $G$ satisfies Condition (4) of Proposition \ref{prop_graph}, then by Lemma \ref{lem_existence_ijk_rep}, there is an irreducible meridian-traceless $\SU(2)$--representation of $\pi_1(S^3-L)$, which contradicts the assumptions.
\end{enumerate}
In conclusion, the link $L$ is the Hopf link or a connect sum of Hopf links, and hence the theorem is proved. 
 \end{proof}
 
 Now we prove Corollary \ref{cor_Hopf_link_NP}.

\begin{repCorollary}{cor_Hopf_link_NP}
Suppose $L\subset S^3$ is 
given by a link diagram, and assume the generalized Riemann hypothesis (GRH). Then the
assertion that $L$ is the unknot, the Hopf link or a connected sum of Hopf links 
is in {\bf coNP}.
\end{repCorollary}
\begin{proof}
The proof is adapted from \cite{Kuperberg_unknot_coNP}  
with minor changes. 
 
Given a link diagram of $L$, we obtain the Wirtinger presentation
of $\pi_1(S^3-L)$ whose length $l$ depends polynomially on the number of crossings.
If $L$ is not the unknot, the Hopf link, or a  connected sum of Hopf links, then by Theorem \ref{thm_main}, there is a meridian-traceless homomorphism 
$$
\rho_\bC : \pi_1(S^3-L)\to SU(2)\subset SL(2,\bC)
$$
with non-abelian image. 

By the Wirtinger presentation,
meridian-traceless representations of $\pi_1(S^3-L)$ in $\SL(2,\bC)$  are given by 
solutions to a system of polynomial equations with integer coefficients.  
 Assuming the GRH, Kuperberg \cite[Theorem 3.3]{Kuperberg_unknot_coNP}  proved that if 
a system of polynomial equations with integer coefficients has a solution over $\bC$, then
it also has a solution over $\bZ/p$ for some prime $p$, such that $\log p$
is bounded from above by a polynomial on the number of  equations, the number of variables,
the number of digits of the maximum degree of the equations, and the number of digits of the maximum absolute value of the coefficients. All these data are bounded from above by a polynomial of $l$.
As a consequence, there is a meridian-traceless homomorphism
\begin{equation*}
\rho_p : \pi_1(S^3-L)\to  \SL(2,\bZ/p),
\end{equation*}
where $p$ is a prime number, such that $\log p$ is bounded from above by a polynomial of $l$. 
Moreover, one can further require that $\rho_p$ has a non-abelian image using the ``Rabinowitsch trick'' described in  \cite[Theorem 3.4]{Kuperberg_unknot_coNP}.

By considering the diagonalizations in the algebraic closure of $\bZ/p$, it is straightforward to verify that if $a,b\in \SL(2,\bZ/p)$ are two traceless elements that are commutative to each other, then $a=\pm b$. As a consequence, if $L'$ is an unknot, a Hopf link, or a connected sum of Hopf links, then $\pi_1(S^3-L')$ does not admit non-abelian traceless representations in  $\SL(2,\bZ/p)$. Therefore, the prime number $p$ and the representation $\rho_p$ above provide the desired \emph{certificate} for the problem.  
\end{proof}

The rest of the paper is devoted to the proofs of Propositions \ref{prop_only_abelian_rep_implies_minimal_I}, \ref{prop_sublink_minimal_I}, \ref{prop_tree_implies_Hopf_connected_sum}, \ref{prop_cycle_length_neq_4}, \ref{prop_no_G0}, and \ref{prop_graph}.

%% file: singular_instanton.tex
This section proves several general properties of singular instanton Floer homology. 
Suppose $L$ is a link in $S^3$, we use $|L|$ to denote the number of components of $L$.
Pick a base point  $p\in L$, the reduced singular instanton Floer
homology $\II^\natural(L,p)$ is defined by Kronheimer and Mrowka in 
\cite{KM:Kh-unknot}. 
 Roughly speaking, $\II^\natural(L,p)$ is defined as the Morse homology
of the (perturbed) Chern-Simons functional on a configuration space
of orbifold connections.
Let $m_p$ be a small meridian around $p$, let
$\omega$ be a small arc joining $m_p$ and $p$, let $\mu(\omega)$ be a meridian of $\omega$,
and let $\mu_i$ ($i=1,\cdots,|L|+1$) be the meridians of the components of $L^\natural:=L\cup m_p$. We require $\mu_{|L|+1}$ to be the meridian of $m_p$.
The critical set of the unperturbed Chern-Simons functional is given by
\begin{multline*}
	R^\natural(L,p):=\{
\rho: \pi_1(S^3-(L^\natural \cup \omega))\to \SU(2)|
\\
\trace\rho(\mu_i)=0~\text{for all}~i, 
\rho(\mu({\omega}))=-\id 
\}/\text{conjugations}.
\end{multline*}
In general, we have to perturb the Chern-Simons
functional and use the perturbed critical points to define $\II^\natural$.
However, if all the elements in 
$R^\natural(L,p)$ are non-degenerate critical points of the Chern-Simons
functional, then we can define $\II^\natural(L,p)$ using these points directly.

 Fix an orientation and a planar diagram for $L$, and fix a base point $p\in L$ on the diagram. Recall that by the Wirtinger presentation, the group $\pi_1(S^3-L)$ is generated by 
the meridians around the arcs of the diagram. Therefore, to specify a presentation of $\pi_1(S^3-L)$, one only needs to specify the image of the meridian around each arc.
 We define 
 $$
 R(L,p):=\{\rho: \pi_1(S^3-L)\to \SU(2)|
 \trace\rho(\mu_i)=0,~ 1\le i \le |L|, ~ \rho (m_p)=\mathbf{i}                 \}
 $$ 
Notice that in the definition of $R(L,p)$, one does not take the quotient by conjugations. The following  lemma implies that $R(L,p)$ does not depend on the choice of the diagram.
\begin{Lemma}
\label{lem_R=R^natural}
There is a canonical one-to-one correspondence between  $R(L,p)$ and 
$R^\natural(L,p)$. 
\end{Lemma}
\begin{proof} 
Fix an oriented planar diagram of $L$, and let $\rho\in R(L,p)$.  By the Wirtinger presentation, an $\SU(2)$--representation of $S^3-L$ is equivalent to a map from the set of the arcs to $\SU(2)$ that satisfies certain compatibility conditions.  Add $\omega$ and  $m_p$ to the diagram of $L$ as in  Figure \ref{fig_meridian_representation}. We define a representation $\rho^\natural$ of $\pi_1(S^3-(L^\natural-\omega))$ by taking the images of the meridians near $p$ as in Figure \ref{fig_meridian_representation}, and taking the images of the meridians of the other arcs to be the same as $\rho$. It is straightforward to verify that $\rho^\natural\in R^\natural(L,p)$, and that the map from $\rho$ to $\rho^\natural$ is a one-to-one correspondence.
\end{proof}

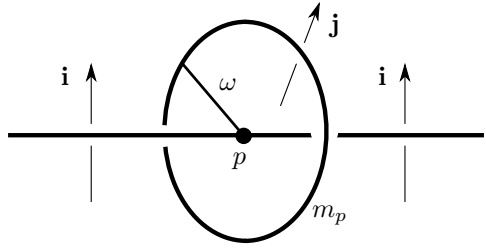
\begin{figure}
\begin{overpic}[width=0.5\textwidth]{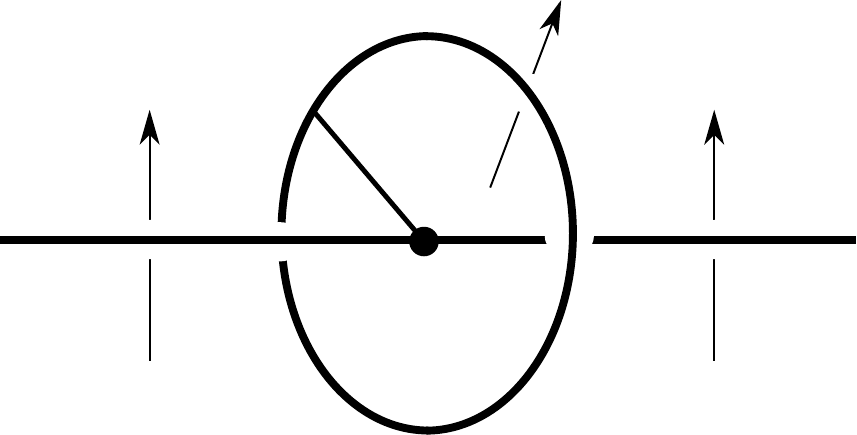}
\put(20,60){$\bfi$}	
\put(140,60){$\bfi$}	
\put(122,80){$\bfj$}	
\put(80,58){$\omega$}	
\put(85,30){$p$}	
\put(115,10){$m_p$}	
\end{overpic}
\caption{Define $\rho^\natural$ near $p$.}
\label{fig_meridian_representation}
\end{figure}

For $\rho\in R(L,p)$, we will use $\rho^\natural\in R^\natural(L,p)$ to denote the corresponding
$\SU(2)$--representation of 
$\pi_1(S^3-(L^\natural\cup\omega))$
given by Lemma \ref{lem_R=R^natural}. 
According to \cite[Lemma 3.13]{KM:YAFT},
the representation $\rho^\natural \in R^\natural(L,p)$ gives a non-degenerate critical point if and only if
the kernel of the map
\begin{equation}\label{eq_hessian_I_natural}
H^1(S^3-L^\natural; \ad \rho^\natural)\to 
H^1(\underline{m}^\natural;\ad \rho^\natural)
\end{equation}
is zero, where $\underline{m}^\natural$
is any collection of loops representing the meridians of all the components of $L^\natural$, and $\ad \rho^\natural$ is the local system on $S^3-L^\natural$ defined by the composition of $\rho^\natural$ with the adjoint action of $\SU(2)$ on 
$\mathfrak{su}(2)\cong \bR^3$. Since $-\id\in\SU(2)$ acts trivially on $\mathfrak{su}(2)$, the local system $\ad \rho^\natural$ extends from $S^3-(L^\natural\cup\omega)$ to $S^3-L^\natural$.

If we further assume that $\rho\in R(L,p)$ is abelian, then $\ad \rho^\natural$ maps
all the meridians of $L$ to the adjoint actions of $\pm\bfi$, which is given by $\text{diag}(1,-1,-1)\in\SO(3)$. 
\begin{Lemma}\label{lem_hessian_I}
 Suppose $\rho \in R(L,p)$ is an abelian representation.
Then the kernel of the map $\eqref{eq_hessian_I_natural}$ is isomorphic to
the kernel of 
\begin{equation}\label{eq_hessian_I}
H^1(S^3-L, \ad \rho)\to 
H^1(\underline{m};\ad \rho),
\end{equation}
where $\underline{m}$
is any collection of loops representing the meridians of all the components of $L$.
\end{Lemma}
\begin{proof} 
Let $H$ be the Hopf link.
The pair $(S^3, L^\natural)$ can be viewed as the connected
sum of $(S^3,L)$ and $(S^3, H)$ along a point on each of $L$ 
and $H$. To form the connected sum, we need to remove
small neighborhoods of the two points and glue the complements along 2-punctured spheres. We will use $B^3$ to denote the neighborhoods of the two points and use $(S^2-2\pt)$ to denote the 2-punctured sphere.  The Mayer-Vietoris sequence for the union of
$(S^3-L-B^3)$ and $(S^3-H-B^3)$ 
gives the follows long exact sequence
\begin{align}
\cdots&\to H^0(S^3-L-B^3;\ad \rho)\oplus H^0(S^3-H-B^3,\ad \rho_0) 
\to H^0(S^2-2\pt;\ad \rho)
\nonumber
\\
&\to H^1(S^3-L^\natural;\ad \rho^\natural)
 \to H^1(S^3-L-B^3;\ad \rho)\oplus H^1(S^3-H-B^3,\ad \rho_0) 
 \nonumber
 \\
&\to H^1(S^2-2\pt,\ad \rho)\to \cdots,
\label{eqn_long_exact_seq_1}
\end{align}
where  $\ad \rho_0$ is the local system on $S^3-H$ that maps the two generators of $\pi_1(S^3-H)\cong\bZ^2$ to
$\text{diag}(1,-1,-1)$ and $\text{diag}(-1,-1,1)$ respectively, and we use the same notation of a local system to denote its restrictions.
Notice that the inclusion maps
\begin{align*}
	S^3-H-B^3 &\hookrightarrow S^3-H,\\
	S^3-L-B^3 &\hookrightarrow S^3-L,
\end{align*}
are homotopy equivalences, so they induce isomorphisms on cohomology.
Since $T^2$ is a deformation retract of $S^3-H$, we have
$$
 H^*(S^3-H,\ad \rho_0)\cong H^*(T^2,\ad \rho_0)=0
$$
by the K\"unneth formula. Therefore, \eqref{eqn_long_exact_seq_1} yields the following long exact sequence:
\begin{align*}
\cdots &\to H^0(S^3-L^\natural;\ad \rho^\natural)\to  
    H^0(S^3-L;\ad \rho) 
\to H^0(S^2-2\pt;\ad \rho)\\
&\to H^1(S^3-L^\natural;\ad \rho^\natural) 
 \to H^1(S^3-L;\ad \rho) 
\to H^1(S^2-2\pt,\ad \rho)\to \cdots.
\end{align*}
Since all the elements of $R^\natural(L,p)$ are irreducible (this follows from, for example, the definition of $\rho^\natural$ in the proof of Lemma \ref{lem_R=R^natural}), we have
$H^0(S^3-L^\natural;\ad \rho^\natural)=0$. Since $\rho$ is reducible and non-trivial, 
$$
H^0(S^3-L;\ad \rho)\cong \bR.
$$
Since $(S^2-2\pt)$ retracts to a meridian of $L$,
it is straightforward to calculate that
$$
H^0(S^2-2\pt,\ad \rho)\cong \bR,~~
H^1(S^2-2\pt,\ad \rho)\cong \bR.
$$
Hence we obtain the following exact sequence
$$
0\to H^1(S^3-L^\natural;\ad \rho^\natural) 
 \to H^1(S^3-L;\ad \rho) 
\to H^1(S^2-2\pt,\ad \rho).
$$
The last map in the above sequence is a component of \eqref{eq_hessian_I}, therefore the kernels of \eqref{eq_hessian_I_natural} and
\eqref{eq_hessian_I} are the isomorphic.
\end{proof}

Let $\Delta_L(t)$ be the (single-variable) Alexander polynomial of $L$ under a given orientation. Then $|\Delta_L(-1)|$ is independent of the orientation of $L$. The value of $|\Delta_L(-1)|$ is called the \emph{determinant} of $L$, which will be denoted by $\det(L)$.

\begin{Proposition}\label{prop_degenerate_rep}
Let $L\subset S^3$ be a link, let $p\in L$, and suppose that $\rho\in R(L,p)$ is abelian. Then 
$\rho^\natural$ gives a non-degenerate critical point of the Chern-Simons functional if and only if
$\det(L)\neq 0$.
\end{Proposition}

\begin{proof}

We have 
$$
H^1(S^3-L, \ad \rho)\cong 
H^1(S^3-L, \bR)\oplus H^1(S^3-L, \widetilde{\bR}^{\oplus 2}) 
\cong \bR^{|L|}\oplus H^1(S^3-L, \widetilde{\bR}^{\oplus 2})
$$
and
$$
H^1(\underline{m};\ad \rho)\cong 
H^1(\underline{m}, \bR)\oplus H^1(\underline{m}, \widetilde{\bR}^{\oplus 2}) 
\cong \bR^{|L|},
$$
where $\widetilde{\bR}$ denotes the local system whose holonomy around each meridian of $L$ is $-1$. Using the above decompositions, it is clear that the kernel of
the map \eqref{eq_hessian_I} is isomorphic to
$
H^1(S^3-L, \widetilde{\bR}) ^{\oplus 2}.
$
Therefore by Lemma \ref{lem_hessian_I},
$\rho^\natural$ is  degenerate if and only if
$
H^1(S^3-L, \widetilde{\bR})\neq 0.
$

Since
$$
H^1(S^3-L, \widetilde{\bR})\cong \Hom(H_1(S^3-L, \widetilde{\bR}),\bR),
$$
it suffices to show that 
$$
\dim H_1(S^3-L, \widetilde{\bR}) >0
$$
if and only if $\det(L)=0$.

Let $\pi:\Sigma(L)\to S^3$ be the branched double cover of $S^3$ along $L$. We  use $\widetilde{S^3-L}$ to denote $\pi^{-1}(S^3-L)$, use $\widetilde{L}$ to denote $\pi^{-1}(L)$, and use $N(\widetilde{L})$ to denote a tubular neighborhood of $\widetilde{L}$. 

From the long exact sequence for $(\Sigma(L),\Sigma(L)-N(\widetilde{L}))$ and excision, we have the
following commutative diagram:
$$
\xymatrix{
H_2(N(\widetilde{L}), \partial N(\widetilde{L});\bR)  \ar[r]^{\quad\iota} \ar[dr]^{\cong}    
&  H_1(\widetilde{S^3-L};\bR )    \ar[d]  \ar[r]   
&  H_1(\Sigma(L);\bR)    \ar[r]
&  H_1(N(\widetilde{L}), \partial N(\widetilde{L});\bR) \ar[d]^{\cong}\\    
 &    H_1({S^3-L};\bR )         & & 0
}
$$
where the first row is exact.
The triangle in the above diagram implies that the map $\iota$ is injective, therefore 
$$
\dim H_1(\widetilde{S^3-L};\bR )= |L|+b_1(\Sigma(L)).
$$
According to \cite[Example 3H.3]{Hatcher}, we have the following 
long exact sequence 
\begin{align*}
H_2(\widetilde{S^3-L};\bR) &\twoheadrightarrow H_2({S^3-L};\bR)
\to H_1(S^3-L;\widetilde{\bR})  \\
&\to
H_1(\widetilde{S^3-L};\bR) \twoheadrightarrow H_1({S^3-L};\bR),
\end{align*}
where the surjectivity of the first arrow follows from 
the fact that  $H_2({S^3-L};\bR)$ is generated by 
the components of $\partial N(L)$, and the surjectivity of the last arrow follows from the previous
commutative diagram. Hence we obtain
$$
\dim H_1(S^3-L;\widetilde{\bR})=b_1(\Sigma(L)).
$$

By \cite[Theorem 9.1]{Lickorish}, if $A$ is the Seifert matrix of $L$, then $H_1(\Sigma(L);\bZ)$ is a finitely generated abelian group presented by the matrix $A+A^t$. Therefore
$b_1(\Sigma(L))>0$ if and only if $\det(A+A^t)=0$. Since $|\det(A+A^t)|=|\Delta_L(-1)|=\det(L)$, the proposition is proved.
\end{proof}

The following result is proved by the same argument as  \cite[Theorem 10]{Klassen}. 
\begin{Proposition}\label{prop_binary_dihedral_rep}
Suppose $L$ is a link in $S^3$ and $p\in L$ is a base point. If $\det(L)= 0$, then 
$R(L,p)$ contains 
infinitely many irreducible representations.
\end{Proposition}

\begin{proof}
	Fix an oriented planar diagram for $L$. If $L$ is the disjoint union of a sublink and an unknot, then the desired property is obvious.  From now on, we assume that $L$ does not split as the disjoint union of a sublink and an unknot, so the number of arcs is equal to the number of crossings in the diagram.
	
	Recall that by the Wirtinger presentation, an $\SU(2)$--representation of $\pi_1(S^3-L)$ can be specified by the images of meridians around all the arcs.
Let $D$ be the subset of $\SU(2)$ defined by 
$$
D := \{\cos\theta \, \bfi + \sin\theta \, \bfj \in\SU(2) |\theta\in \bR/2\pi\bZ\}.
$$
We study $\SU(2)$--representations of $\pi_1(S^3-L)$ such that the images of meridians around the arcs are all contained in $D$.  Such a representation is given by a map from the set of arcs to $D\cong \bR/2\pi\bZ$. 

Label the arcs in the diagram as $a_1,\cdots,a_N$, and the let $\theta_i\in \bR/2\pi\bZ$ be the image of $a_i$ under the map. For a crossing where the arc above is $a_i$ and the arcs below are $a_j$ and $a_k$, the compatibility condition is given by 
$$
2\theta_i - \theta_j - \theta_k = 0 \in \bR/2\pi\bZ.
$$

Since there are $N$ crossings, we have $N$ equations. Recall that any one of the relations in the Wirtinger presentation can be omitted from the presentation and the group remains the same, so we remove one of the equations and obtain $N-1$ equations. Suppose $a_N$ is the arc containing $p$, then the definition of $R(L,p)$ requires $\theta_N=0$. Hence the system is given by an $(N-1)\times (N-1)$ matrix. This matrix is the same as the Alexander matrix obtained from the Wirtinger presentation evaluated at $t=-1$. Therefore, if $\det(L)=0$, the system of equations is given by a matrix with determinant zero, and hence it has infinitely many solutions. 
Since there are only finitely many abelian elements  in $R(L,p)$, we conclude that $R(L,p)$ contains 
infinitely many irreducible elements.
\end{proof}

By Propositions \ref{prop_degenerate_rep} and \ref{prop_binary_dihedral_rep}, if all the elements in $R(L,p)$ are abelian, then $\det(L)\neq 0$, and hence every element of $R^\natural(L,p)$ represents a non-degenerate critical point of the Chern-Simons functional.

\begin{Proposition}\label{prop_removing_component_rank_change}
Suppose $L=K_1\cup\cdots\cup K_n$ is an $n$--component link in $S^3$ and
$p\in L$ is a base point which does not lie on $K_1$. Then we have
$$
\dim_\bC\II^\natural(L,p;\bC)\ge 2\dim_\bC \II^\natural(L-K_1,p;\bC),
$$
and
$$
\dim_\bC\II^\natural(L,p;\bC)\ge 2^{n-1}.
$$
\end{Proposition}
\begin{remark}
\label{rmk_sublink_min_I}
	As a corollary, if $L$ has minimal $\II^\natural$ (see Definition \ref{def_minimal_I_natural}), then every sublink of $L$ also have minimal $\II^\natural$.
\end{remark}

\begin{proof}
We use $\mathcal{R}$ to denote
the ring $\bC[t,t^{-1}]$.
Let $\Gamma_{1}$ be the 
local system associated with $K_1$ (see \cite[Section 3.9]{KM:YAFT} or \cite[Section 3]{XZ:forest} for the definition in detail). 
Let $U$ be an unknot that is contained in a solid 3-ball disjoint from $L$, Let $\Gamma_{U}$ be the 
local system associated with $U$. Then we have
\begin{align}
\rank_\mathcal{R} \II^\natural(L,p;\Gamma_1)&=\rank_\mathcal{R} 
\II^\natural((L-K_1)\sqcup U,p;\Gamma_U) 
\label{eqn_local_sys_compute_1}
\\
&=  \rank_{\mathcal{R}} \II^\sharp(U;\Gamma_U) \cdot \rank_\mathcal{R} 
\II^\natural(L-K_1 ,p;\mathcal{R}) 
\label{eqn_local_sys_compute_2}\\
&= 2\rank_\mathcal{R} 
\II^\natural(L-K_1 ,p;\mathcal{R}) 
\label{eqn_local_sys_compute_3}
\\
&=2 \dim_\bC
\II^\natural(L-K_1 ,p;\bC)
\label{eqn_local_sys_compute_4}
\end{align}
where \eqref{eqn_local_sys_compute_1} follows from \cite[Corollary 3.3]{XZ:forest}, and  \eqref{eqn_local_sys_compute_2}
 follows from \cite[Proposition 5.8]{KM:Kh-unknot}. 
(Actually, \cite[Proposition 5.8]{KM:Kh-unknot} proved a K\"unneth formula for $\II^\sharp$, but here we are taking
a product of $\II^\natural(\cdot)$ and $\II^\sharp(\cdot)$. So \eqref{eqn_local_sys_compute_2} does not follow directly from the statement of \cite[Proposition 5.8]{KM:Kh-unknot}, but the same excision argument works here without difficulty.) The local system $\mathcal{R}$ in \eqref{eqn_local_sys_compute_3} is the trivial local system with coefficient ring $\mathcal{R}$, and Equation \eqref{eqn_local_sys_compute_4} follows from the universal coefficient system by viewing $\mathcal{R}$ as a free module over $\bC$.

The trivial local system with $\bC$--coefficients can be recovered from $\Gamma_1$ by taking a tensor product with  $\mathcal{R}/(t-1)$ over $\mathcal{R}$ on the chain level.
By the universal coefficient theorem,
$$
\rank_\mathcal{R} \II^\natural(L,p;\Gamma_1)\le 
\dim_\bC 
\II^\natural(L,p;\Gamma_1\otimes_{\mathcal{R}}\mathcal{R}/(t-1))=
 \dim_\bC 
\II^\natural(L,p;\bC).
$$
Therefore, we have
$$
2\dim_\bC 
\II^\natural(L-K_1 ,p;\bC)\le  \dim_\bC 
\II^\natural(L,p;\bC).
$$

To prove the second part of the statement, assume with out loss of generality 
that $p\in K_n$. By induction, we have
$$
 \dim_\bC 
\II^\natural(L,p;\bC)\ge 2^{n-1}\dim_\bC  
\II^\natural(K_n,p;\bC).
$$
On the other hand, by \cite[Proposition 1.4]{KM:Kh-unknot}, we have
$$
\dim \II^\natural(K_n,p;\bC)= \dim \KHI(K_n),
$$
where $\KHI$ is the instanton knot homology introduced in \cite{KM:suture}.
By \cite[Proposition 7.16]{KM:suture}, we know that
$\KHI(K_n)$ is non-zero, so the desired result is proved.
\end{proof}

\begin{Proposition}\label{prop_abelian_rep_imply_minimal_I_Z}
If $\pi_1(S^3-L)$ does not have any irreducible meridian-traceless $\SU(2)$--representations, then
\begin{equation}
\label{eqn_no_rep_imply_minimal_I_natural_Z}
	\II^\natural(L,p;\bZ)\cong \bZ^{2^{|L|-1}}
\end{equation}
for all $p\in L$.
\end{Proposition}

\begin{remark}
\label{rmk_no_rep_imply_minimal_I_natural_Z}
	By the universal coefficient theorem, \eqref{eqn_no_rep_imply_minimal_I_natural_Z} implies
	$$
	\II^\natural(L,p;\bC)\cong \bC^{2^{|L|-1}}.
	$$
\end{remark}
\begin{proof}
Suppose $L$ has $n$ components, then the set $R(L,p)$ contains $2^{n-1}$ abelian elements.
By the assumption, these are all the elements of $R(L,p)$.
 By Lemma \ref{lem_R=R^natural}, Proposition \ref{prop_degenerate_rep}, and Proposition 
\ref{prop_binary_dihedral_rep}, every element of $R(L,p)$ represents a non-degenerate critical point of the Chern-Simons functional. Therefore, the Floer chain complex $C\II^\natural (L,p;\bZ)$
is generated by $2^{n-1}$ elements. On the other hand, we have
$$
\dim_\bC\II^\natural (L,p;\bC)\ge 2^{n-1}
$$
by Proposition \ref{prop_removing_component_rank_change}.
Therefore, the Floer differentials on $C\II^\natural (L,p;\bZ)$ is zero, and we have
$$
\II^\natural(L,p;\bZ)\cong \bZ^{2^{n-1}}.
\phantom\qedhere\makeatletter\displaymath@qed
$$
\end{proof}

The next proposition establishes a relation between $\II^\natural(L,p)$ and the Alexander polynomials of $L$. Suppose $f$ is a (possibly multi-variable) polynomial, we use $\|f\|$ to denote the sum of the absolute values of the coefficients of $f$.

\begin{Proposition}\label{prop_I>Alexander}
Suppose $L\subset S^3$ is a link with $n\ge 2$ components, and fix an orientation of $L$. Let $\Delta_L$ denote the single-variable Alexander polynomial of $L$, let $\tilde{\Delta}_L$ denote the multi-variable Alexander polynomial of $L$. Then we have
\begin{equation}
\label{eqn_I>Alexander}
	2^{n-1}\dim_\bC \II^\natural (L,p;\bC)\ge 
\|(1-x_1)(1-x_2)\cdots(1-x_n)\tilde{\Delta}_L(x_1,x_2,\cdots,x_n)  \|
\end{equation}
and
\begin{equation}
\label{eqn_I>single_v_Alexander}
2^{n-1}\dim_\bC\II^\natural (L,p;\bC)\ge 
\|(1-x)^{n-1}\Delta_L(x)  \| \ge 2^{n-1} \det(L).
\end{equation}
\end{Proposition}
\begin{proof}
Suppose $L=K_1\cup\cdots \cup K_n$.
Without loss of generality, we assume $p\in K_n$. 
Pick a point $p_i \in K_i$
for each $1\le i \le n-1$. For all $i=1,\cdots,n$, let $m_i$ be a small meridian of $K_i$
around $p_i$, and let $u_i$ be an arc joining $m_i$  and $p_i$. The pair $(m_i,p_i)$
is called an earring.

We use an ``earring-removing'' argument from 
\cite[Section 3]{Xie-earring}.
Pick a crossing between $m_k$ and $K_k$ and apply 
the unoriented skein exact triangle \cite[Section 6]{KM:Kh-unknot}, we have
a 3-cyclic exact sequence
\begin{align*}
\cdots &\to\II(L\cup \bigcup_{k\le i \le n} m_i,\sum_{k\le i \le n} u_i;\bC)\to 
\II(L\cup \bigcup_{k+1\le i \le n} m_i,\sum_{k+1\le i \le n} u_i;\bC)\\
&\to \II(L\cup \bigcup_{k+1\le i \le n} m_i,\sum_{k+1\le i \le n} u_i;\bC) \to \cdots
\end{align*}
for every $k=1,\cdots,n-1$.
 The above exact sequence implies that
$$
\dim_\bC\II(L\cup \bigcup_{k\le i \le n} m_i,\sum_{k\le i \le n} u_i;\bC)
\le  
2\dim_\bC \II(L\cup \bigcup_{k+1\le i \le n} m_i,\sum_{k+1\le i \le n} u_i;\bC)
$$
for every $k=1,\cdots,n-1$. Therefore
$$
\dim_\bC \II(L\cup \bigcup_{1\le i \le n} m_i,\sum_{1\le i \le n} u_i;\bC)
\le  
2^{n-1}\dim_\bC \II(L\cup m_n,u_n;\bC)=2^{n-1}\dim_\bC \II^\natural(L,p;\bC) .
$$

On the other hand, we have
$$
\II(L\cup \bigcup_{1\le i \le n} m_i,\sum_{1\le i \le n} u_i;\bC)
\cong \KHI(L;\bC)
$$
by \cite[Proposition 5.1]{Xie-earring}, where $\KHI(L)$ is the instanton knot homology for $L$ defined in 
\cite{KM:Alexander}. Therefore,
\begin{equation}
\label{eqn_bound_I_from_KHI}
	2^{n-1}\dim_\bC \II^\natural(L;\bC)\ge \dim_\bC \KHI(L;\bC).
\end{equation}

According to \cite[Theorem 1.4]{LY_multi_Alexander}, $\KHI(L)$ can be
equipped with multi-gradings whose graded Euler characteristics recover the absolute values of the coefficients of
$$
(1-x_1)(1-x_2)\cdots(1-x_n)\tilde{\Delta}_L(x_1,x_2,\cdots,x_n).
$$
 Therefore, we have
\begin{align*}
2^{n-1}\dim_\bC \II^\natural(L;\bC)
& \ge \dim_\bC \KHI(L;\bC)
\\
& \ge \|(1-x_1)(1-x_2)\cdots(1-x_n)\tilde{\Delta}_L(x_1,x_2,\cdots,x_n) \|,
\end{align*}
and hence \eqref{eqn_I>Alexander} is proved. Inequality \eqref{eqn_I>single_v_Alexander} follows from  \eqref{eqn_I>Alexander} because
$$
\Delta_L(x) = (1-x)\tilde{\Delta}_L(x,x,\cdots,x).
\phantom\qedhere\makeatletter\displaymath@qed
$$
\end{proof}

%% file: Seifert_surface.tex
\begin{Proposition}
\label{prop_disjoint_from_Seifert_surface}
	Suppose $L$ is a link with minimal $\II^\natural$, and suppose $L$ is divided into the union of two sublinks $L=L'\cup L''$ with $L'\cap L''=\emptyset$ such that the following properties are satisfied:
	\begin{enumerate}
		\item $L'$ is the Hopf link or a connected sum of Hopf links, 
		\item $L'$ has a Seifert surface $S$ with genus zero, such that the algebraic intersection number of $S$ with every component of $L''$ is zero.
	\end{enumerate}
	Then $L''$ can be isotoped in $S^3-L'$ to be disjoint from $S$.
\end{Proposition}
\begin{proof}[Sketch of proof]
If $L'$ is the link given by Figure \ref{fig_standard_chain} and $L''$ has only one component, then this proposition is the same as  
\cite[Proposition 7.9]{XZ:forest}: although \cite{XZ:forest} assumed that $L$ has minimal Khovanov homology, this assumption was only used in the proof of  \cite[Proposition 7.9]{XZ:forest} to deduce that $L$ has minimal $\II^\natural$. 
There is no essential change in the proof
for the general case. We give a sketch of the proof here
and refer the reader to \cite[Section 7.2]{XZ:forest} for  details.

We add an earring $(m_i,u_i)$ to each component of $L'$ (see the proof of Proposition \ref{prop_I>Alexander} for the definition of earrings). We claim that since $L$ has minimal $\II^\natural$, the instanton Floer
homology for the triple
\begin{equation}\label{triple_L+earrings}
(S^3, L\cup \bigcup_i m_i, \sum_i u_i)
\end{equation}
has dimension $2^{|L|+|L'|-2}=2^{|L''|+2|L'|-2}$. 
This is proved by showing that $2^{|L|+|L'|-2}$ is both an upper bound and a lower
bound for the dimension of the instanton Floer homology of \eqref{triple_L+earrings}.
The upper bound for the dimension is obtained by
the earring-removing argument as in the proof of Proposition \ref{prop_I>Alexander}.
To show that $2^{|L|+|L'|-2}$ is a lower bound, we first use the local coefficient
argument as in the proof of Proposition \ref{prop_removing_component_rank_change} to obtain 
$$
\dim\II(S^3, L\cup \bigcup_i m_i, \sum_i u_i)\ge 
2^{|L''|}\dim\II(S^3, L'\cup \bigcup_i m_i, \sum_i u_i).
$$
By \cite[Theorem 3.6 and (14)] {KM:Alexander}, we have  
\begin{align*}
\dim\II(S^3, L'\cup \bigcup_i m_i, \sum_i u_i)&=\dim \KHI(L') \\
&\ge ||(1-t)^{|L'|-1}\Delta_{L'}(t)||\\
&\ge | 2^{|L'|-1}\Delta_{L'}(-1)|   \\
&= 2^{2|L'|-2},
\end{align*}
which yields the desired lower bound estimate.

Since the Hopf link is a fibered link with genus $0$ for all possible orientations, by Assumption (1), the link $L'$
is a fibered link with genus $0$ for all possible orientations. 
Since the minimal-genus Seifert surface of an oriented fibered link is unique up to isotopy, the interior of $S$ is isotopic to a fiber of $S^3-L'$. 
Therefore, after a Dehn surgery on $S^3$ along $L'$ that identifies the meridians of $L'$ to the components of $S\cap N(L')$, one obtains the manifold  $S^1\times S^2$. 

We cut  the triple \eqref{triple_L+earrings} along the boundary tori
of the neighborhoods of the components of $L'$ and re-glue as above to obtain a new triple
$$
(S^1\times S^2, \tilde{L}''\cup \tilde{L}' \cup \bigcup_i {\tilde{m}_i},
\sum_i\tilde{u}_i),
$$ 
where $\tilde{L}'',\tilde{L}', \tilde{m}_i $ 
denote the images of $L'',L', m_i$ in the new manifold $S^1\times S^2$, and $\tilde{u}_i$ denotes an arc joining $\tilde{m}_i$ to the $i^{\text{th}}$ component of $\tilde{L}'$.
The excision theorem for instanton Floer homology gives
\begin{equation}
\label{eqn_L''+2L'-2_after_surgery}
	\dim \II(S^1\times S^2, \tilde{L}''\cup \tilde{L}' \cup \bigcup_i {\tilde{m}_i},
\sum_i\tilde{u}_i)
= 2^{|L''|+2|L'|-2}.
\end{equation}

In $S^1\times S^2$, the link $\tilde{L}' \cup \bigcup_i {\tilde{m}_i}$ can be isotoped to a
 ``vertical'' position in the sense
that each slice $\{\pt\}\times S^2$ intersects each component of
$\tilde{L}' \cup \bigcup_i {\tilde{m}_i}$ transversely at exactly one point. The Seifert
surface $S$ can be completed into an embedded $S^2$ that is isotopic to $\{\pt\}\times S^2$ in $S^1\times S^2$. On the other hand,
any connected closed surface in $S^1\times S^2$ that intersects
$\tilde{L}'$ transeversely at $|\tilde{L}'|$ points defines a Seifert surface of $L'$, since
the complement of $\tilde{L}'$ in $S^1\times S^2$ is the same 
as the complement of $L'$ in $S^3$ by our construction.

Let $V$ be an $|L''|$-component unlink in a small ball in 
$S^1\times S^2- (\tilde{L}' \cup \bigcup_i {\tilde{m}_i})$. 
Applying a similar argument as \eqref{eqn_L''+2L'-2_after_surgery}, we have
$$
\dim \II(S^1\times S^2, V\cup \tilde{L}' \cup \bigcup_i {\tilde{m}_i},
\sum_i\tilde{u}_i)
= 2^{|L''|+2|L'|-2}.
$$
By Assumption (2), all the components of $\tilde{L}''$ are 
null-homologous in $S^1\times S^2$, and hence they are null-homotopic in $S^1\times S^2$. Therefore
$\tilde{L}''$ is homotopic to $V$ in $S^1\times S^2$.
A local coefficient argument then shows that
\begin{equation}
\label{eqn_iso_L'_L''_V_S1*S2}
	 \II(S^1\times S^2, \tilde{L}''\cup \tilde{L}' \cup \bigcup_i {\tilde{m}_i},
\sum_i\tilde{u}_i)\cong 
 \II(S^1\times S^2, V\cup \tilde{L}' \cup \bigcup_i {\tilde{m}_i},
\sum_i\tilde{u}_i)
\end{equation}
via an isomorphism that preserves the action of $\muu(S^2)$. 
It is proved in \cite{XZ:excision} that the eigenvalues of $\muu(S^2)$ 
 determines the generalized Thurston norm.  To be more precise, suppose $\Sigma$
 is an embedded surface in $S^1\times S^2$ that is homologous to $\{\pt\}\times S^2$, suppose $\Sigma$
 has genus $g$ and intersects 
 $\tilde{L}''\cup \tilde{L}' \cup \bigcup_i {\tilde{m}_i}$ transversely at $n$ points, then
 the top eigenvalue of $\muu(S^2)$ on 
  $$\II(S^1\times S^2, \tilde{L}''\cup \tilde{L}' \cup \bigcup_i {\tilde{m}_i},
\sum_i\tilde{u}_i)$$
is equal to the smallest possible value of
  $2g+n-2$ among all such surfaces; and a similar result holds for $ V\cup \tilde{L}' \cup \bigcup_i {\tilde{m}_i}$. Therefore, the isomorphism \eqref{eqn_iso_L'_L''_V_S1*S2}
  implies that there exits a genus $0$ surface in $S^1\times S^2$ which
  is disjoint from $\tilde{L}''$ and 
  intersects every component of $\tilde{L}' \cup \bigcup_i {\tilde{m}_i}$ transversely at
  exactly one point. Such a surface corresponds to a Seifert surface $S'$ of $L'$
  that is disjoint from $L''$ in $S^3$ and has genus zero. 

  Since the oriented link $L'$ is fibered, the minimal-genus Seifert surface of $L'$ is unique up to isotopy 
  relative to $L'$. Therefore
  there is an isotopy which takes $S'$ to $S$. Using this isotopy, we can isotope
  $L''$ to a link disjoint from $S$ as desired.
\end{proof}	

%% file: exceptional_link.tex

Recall that a link is called a forest of unknots if it is given by the disjoint unions and connected sums of Hopf links and unknots. The main result of this subsection is the following proposition. 

\begin{Proposition}
\label{prop_2_3_components}
	If $L$ is a link with 2 or 3 components such that $L$ has minimal $\II^\natural$, then $L$ is a forest of unknots. 
\end{Proposition}

To start, we establish Proposition \ref{prop_2_3_components} when $L$ has two components. 

\begin{Lemma}[Special case of Proposition \ref{prop_2_3_components} when $L$ has two components]
\label{lem_2_component_min_I}
If $L$ is a link with 2 components such that $L$ has minimal $\II^\natural$, then $L$ is the Hopf link or the unlink.
\end{Lemma} 

\begin{proof}
	This lemma is essentially the same as \cite[Corollary 6.2]{XZ:forest}. Although the assumptions in \cite{XZ:forest} required that $L$ has minimal Khovanov homology, this condition was only used in \cite[Corollary 6.2]{XZ:forest} to show that $L$ has minimal $\II^\natural$ and that the multi-variable Alexander polynomial $\tilde\Delta_L $ of $L$ satisfies 
	\begin{equation}
	\label{eqn_alex_condition_hopf}
			\|(1-x)(1-y)\,\tilde\Delta_L(x,y)\| \le 4.
	\end{equation}
	In \cite{XZ:forest}, Inequality \eqref{eqn_alex_condition_hopf} was proved using Dowlin's spectral sequence \cite{Dowlin} from Khovanov homology to Heegaard Knot Floer homology. However, by Proposition \ref{prop_I>Alexander},  Inequality \eqref{eqn_alex_condition_hopf} also follows from the minimality of $\II^\natural$. 
\end{proof}

Lemma \ref{lem_Seifert_disks_of_Ki} below is a straightforward generalization of \cite[Proposition 7.1]{XZ:forest}.
\begin{Lemma}
\label{lem_Seifert_disks_of_Ki}
	Suppose $L=K_1\cup \cdots\cup K_n$ is an $n$--component link with minimal $\II^\natural$. Then for every component $K_i$ of $L$, there exists an embedded disk $D_i$, such that
	\begin{enumerate}
		\item $\partial D_i=K_i$ for all $i$,
		\item for all $j\neq i$, the disk $D_i$ intersects $K_j$ transversely at $|\lk(K_i,K_j)|$ points.
	\end{enumerate}
\end{Lemma}

\begin{remark}
	The value of $|\lk(K_i,K_j)|$ does not depend on the orientations of $K_i$ and $K_j$, so it makes sense to refer to $|\lk(K_i,K_j)|$ without specifying an orientation of $L$.
\end{remark}

\begin{proof}
	By Proposition \ref{prop_sublink_minimal_I} and Lemma \ref{lem_2_component_min_I}, for $i\neq j$, the sublink formed by $K_i$ and $K_j$ is either the unlink or the Hopf link, therefore $|\lk(K_i,K_j)|=0$ or $1$, and every component of $L$ is an unknot. For each $i$, take a base point $p_i\in K_i$, and let $\hat L_i\subset S^1\times D^2$ be the link $\cup_{j\neq i} K_j$ viewed as a link in the complement of $N(K_i)$. 
	By \cite[Section 4.3]{AHI} and the assumption that $L$ has minimal $\II^\natural$,
	$$
	\dim \AHI(\hat L_i) = \dim \II^\natural(L,p_i)= 2^{n-1}.
	$$
	
	We will abuse notation and use $K_j$ $(j\neq i)$ to denote the respective component of $\hat L_i$.
	Then by \cite[Corollary 4.4]{XZ:forest}, there is a meridional disk of $S^1\times D^2$ that is disjoint from $K_j$ when $|\lk(K_i,K_j)|=0$ and intersects $K_j$ transversely at one point when $|\lk(K_i,K_j)|=1$. This yields the desired disk $D_i$.
\end{proof}

We can now prove Proposition \ref{prop_tree_implies_Hopf_connected_sum}.
\begin{repProposition}{prop_tree_implies_Hopf_connected_sum}
Suppose $L$ has minimal $\II^\natural$, then the linking number of every pair of components of $L$ is either $0$ or $\pm 1$. If we further assume that the linking graph of $L$ is a forest, then $L$ is a forest of unknots.
\end{repProposition}

\begin{proof}[Proof of Proposition \ref{prop_tree_implies_Hopf_connected_sum}]
The first part of the statement follows from Proposition \ref{prop_sublink_minimal_I} and Lemma \ref{lem_2_component_min_I}. For the second part of the statement, we use induction on $n$.  
The case for $n=1$ follows from 
\cite[Proposition 1.4]{KM:Kh-unknot} and \cite[Proposition 7.16]{KM:suture}.
 When $n\ge 2,$ if the linking graph of $L$ is a forest, then there exists a vertex of the linking graph with degree no greater than 1. Let $K\subset L$ be the component corresponding to such a vertex, and let $L':=L-K$. There are two cases. 
	
	Case 1: there exists a unique element $K'$ of $L'$ such that $|\lk(K',K)|= 1$.  In this case, by Lemma \ref{lem_Seifert_disks_of_Ki}, the link $L$ is given by the union of $L'$ and a meridian of $K'$.  The linking diagram of $L'$ is also a forest.  Therefore, the desired result follows from Proposition  \ref{prop_sublink_minimal_I} and the induction hypothesis on $L'$.
	
	Case 2: the linking number of $K$ with every component of $L$ is zero. In this case, by Lemma \ref{lem_Seifert_disks_of_Ki}, the link $L$ is the disjoint union of $L'$ and an unknot. Therefore, the  desired result follows from Proposition  \ref{prop_sublink_minimal_I} and the induction hypothesis on $L'$.
\end{proof}

Now we prove Proposition \ref{prop_2_3_components}.

\begin{proof}[Proof of Proposition \ref{prop_2_3_components}]
	By Proposition \ref{prop_tree_implies_Hopf_connected_sum}, we only need to show that if $L$ is a link with three components and has minimal $\II^\natural$, then the linking graph of $L$ cannot be a triangle.
	
	Assume the linking graph of $L$ is a triangle, write $L=K_1\cup K_2\cup K_3$. Then $L' := L_1\cup L_2$ is the Hopf link. Consider the two Seifert surfaces $S_1,S_2$ of $L'$ as shown in Figure \ref{fig_hopf_seifert}. By the assumption on linking numbers, the algebraic intersection number of $K_3$ with either $S_1$ or $S_2$ is zero.  After taking the mirror image of $L$, we may assume without loss of generality that the algebraic intersection number of $K_3$ with $S_1$ is zero. By Proposition \ref{prop_disjoint_from_Seifert_surface}, after an isotopy, we may assume that $K_3$ is disjoint from $S_1$.
	
	By Lemma \ref{lem_Seifert_disks_of_Ki}, $K_3$ bounds an embedded disk $D$ that intersects each of $K_1$ and $K_2$ transversely at one point. After a generic perturbation, the intersection of $D$ and $S_1$ is the disjoint union of an arc and finitely many circles. Let $\theta$ be the arc component of $S_1\cap D$. Since $S_1$ is an annulus, $\theta$ is isotopic to the arc shown in Figure \ref{fig_hopf2_arc}.  By shrinking $K_3$ to a neighborhood of $\theta$ through $D$, we conclude that $L$ must be isotopic to the link shown in Figure \ref{fig_l6n1}, which is the link L6n1.
	
	\begin{figure}
	\begin{overpic}[width=0.8\textwidth]{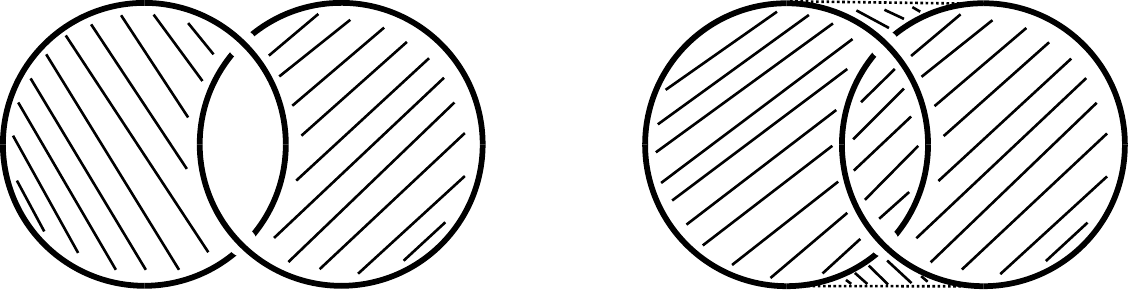}	
	\put(55,-12){$S_1$}
	\put(225,-12){$S_2$}
	\end{overpic}
	\vspace{\baselineskip}
	\caption{Two Seifert surfaces of $L'$}
	\label{fig_hopf_seifert}
	\end{figure}
	
	\begin{figure}[!tbp]
	\begin{minipage}[b]{.49\textwidth}
	\centering
	\includegraphics[width=.7\textwidth]{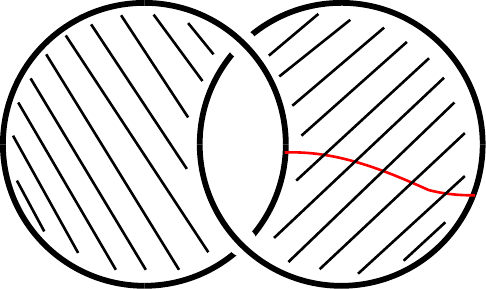}
	\caption{An arc on $S_1$.}
	\label{fig_hopf2_arc}
	\end{minipage}
	\hfill
	\begin{minipage}[b]{.49\textwidth}
	\centering
	\includegraphics[width=0.8\textwidth]{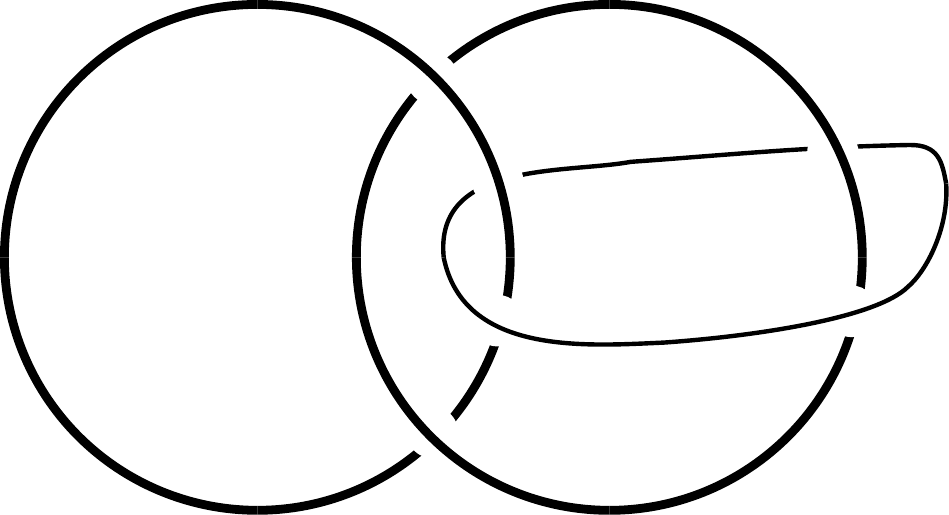}
	\caption{The link L6n1.}
	\label{fig_l6n1}	
	\end{minipage}
	\end{figure}
	
	Although this is not be needed in the proof, we remark that the link L6n1 is the torus link $T(3,3)$. 
	The rest of the proof shows that the link L6n1 does not have minimal $\II^\natural$. 

	We use an argument from \cite{KM_filtration}.
    Notice that the link L6n1 can also be described as the union of the closure of the 2-braid
	$\sigma_1^2$ and the axis unknot. Let $p$ be a point on the axis unknot. Resolving the two crossings of the braid, we obtain the following chain complex whose homology
	is isomorphic to $\II^\natural(L,p)$ by \cite[Theorem 6.8]{KM:Kh-unknot}:
\begin{equation}
\label{eqn_square_chain}
	\xymatrix{
  C\II^\natural(U_3)  \ar[r]\ar[d]\ar[rd]    &    C\II^\natural(U_2)   \ar[d] \\
   C\II^\natural(U_2)  \ar[r]      &          C\II^\natural(K_2\cup U)
}
\end{equation}
where $U_n$ denotes the unlink with $n$ components, and $K_2\cup U$ denotes the link given by the union of the closure of the trivial 2-braid and the axis unknot.
The arrows above are given by link cobordisms and are not necessarily commutative. The differential map of \eqref{eqn_square_chain} the sum of the arrows and the Floer differentials within the four vertices.

The unperturbed representation variety of an unlink with $n$--components is a product of $n-1$ copies of $S^2$. After a suitable perturbation, 
the perturbed representation variety consists of points with the same $\bZ/2$--grading.
The unperturbed representation variety 
for $K_2\cup U$ consists of 4 non-degenerate points which have the same $\bZ/2$--grading. Therefore, the four vertices in the square \eqref{eqn_square_chain} have zero Floer differentials. So \eqref{eqn_square_chain} is isomorphic to the following square
\begin{equation}
\label{eqn_square_Floer}
\xymatrix{
  \II^\natural(U_3,p)  \ar[r]\ar[d]\ar[rd]    &    \II^\natural(U_2,p)   \ar[d] \\
   \II^\natural(U_2,p)  \ar[r]      &          \II^\natural(K_2\cup U,p)
}
\end{equation}
where the non-diagonal arrows are induced by link cobordisms and
the diagonal arrow is induced by a 1-parameter family of link cobordisms. 
Notice that all these cobordisms are trivial products  on the axis unknot.

Every link in \eqref{eqn_square_Floer} is the union of  an annular link 
(i.e. a link in the solid torus) and
the axis unknot. By \cite[Equation (4.5)]{AHI}, the $\II^\natural$'s
in the square are isomorphic to the annular instanton Floer homology groups $\AHI$
of the corresponding annular link. 
Moreover, the isomorphism is natural 
with respect to link cobordisms that are trivial  on the axis unknot. Therefore,
the non-diagonal maps of \eqref{eqn_square_Floer} are conjugate to the following maps on $\AHI$:
\begin{equation}
\label{eqn_square_AHI}
\xymatrix{
  \AHI(U_2)  \ar[r]\ar[d]   &    \AHI(U_1)   \ar[d] \\
   \AHI(U_1)  \ar[r]      &          \AHI(K_2)
}
\end{equation}
where $U_n$ is the trivial annular link with $n$ components, and $K_n$ is the closure of the trivial $n$--braid. All the edge maps in \eqref{eqn_square_AHI}
are induced by pair-of-pants cobordisms between links in the solid torus.
 The edge maps of \eqref{eqn_square_AHI} are calculated 
in \cite[Proposition 5.14]{AHI} and all four maps are non-zero. In fact,
these maps can be identified with the pair-of-pants maps for \emph{annular Khovanov homology} introduced in \cite{APS}.
(See also \cite{Rob} for an exposition on annular Khovanov homology.)
Therefore by the previous discussion, the four edge maps in \eqref{eqn_square_Floer} are also all non-zero.
	
The groups in \eqref{eqn_square_Floer} are all $\bZ/2$--graded and every group is supported in a single $\bZ/2$--grading. Since the diagonal map is defined using a 1-parameter family of metrics, its degree differs from the degree of the composition of the two edge maps by $1$. Since the edge maps are all non-zero, the $\bZ/2$--grading implies that the diagonal map in  \eqref{eqn_square_Floer} must be zero 
(cf. \cite[Lemma 9.2]{KM_filtration}).

As a consequence, we have
\begin{equation*}
\dim\II^\natural(L_1,p)=\dim \AKh(\sigma_1^2;\bC)=6>4,
\end{equation*}
and the desired result is proved.
\end{proof}

%% file: disk_minimal_position.tex

Proposition \ref{prop_disks_in_minimal_position} below is a straightforward generalization of \cite[Proposition 7.6]{XZ:forest}. It is the consequence of Lemma \ref{lem_Seifert_disks_of_Ki} and a surgery argument. 

\begin{Proposition}
\label{prop_disks_in_minimal_position}
	Suppose $L=K_1\cup \cdots\cup K_n$ has minimal $\II^\natural$. Then there exists a sequence of embedded disks $D_1,\cdots,D_n$, such that
	\begin{enumerate}
		\item $\partial D_i = K_i$ for all $i$,
		\item $D_i$ is transverse to $K_j$ for all $j\neq i$,
		\item $D_i\cap D_j$ is empty if $|\lk(K_i,K_j)|=0$,
		\item $D_i$ and $D_j$ intersect transversely at one arc if $|\lk(K_i,K_j)|\neq 0$.
	\end{enumerate}
\end{Proposition}

\begin{proof}[Sketch of proof]
The proof is almost verbatim as \cite[Proposition 7.6]{XZ:forest}.  We give a sketch of the proof and refer the reader to \cite[Proposition 7.6]{XZ:forest} for details.

	Let $D=(D_1,\cdots,D_n)$ be a sequence of \emph{immersed} disks such that $\partial D_i=K_i$. We say that $D$ is \emph{generic}, if all the intersection points are locally diffeomorphic to one of the following models in $\bR^3$ at $(0,0,0)$: 
\begin{enumerate}
\item the intersection of the $xy$--plane and $\{(x,y,z)|x=0, z\ge 0\}$,
\item the intersection of the $xy$--plane and the $yz$--plane, 
\item the intersection of the $xy$--, $xz$--, and $yz$--planes. 
\end{enumerate}
If $D=(D_1,\cdots,D_n)$ a generic sequence of disks, let $\Sigma_1(D)$, $\Sigma_2(D)$, $\Sigma_3(D)$ be the set of intersection points described by (1), (2), (3) above respectively. Define the \emph{complexity} of $D$ to be the number of connected components of $\Sigma_2(D)$.

We say that the sequence $D=(D_1,\cdots,D_n)$ is \emph{admissible}, if 
\begin{enumerate}
	\item $D$ is generic,
	\item $D_i$ intersects $K_j$ at $|\lk(K_i,K_j)|$ points for all $i\neq j$,
	\item the interior of $D_i$ is disjoint from $K_i$ for all $i$,
	\item every element of $\Sigma_3(D)$ is contained in at least two different disks.
\end{enumerate}

Condition (4) above is equivalent to the condition that none of the disks in $D$ has triple self-intersections. 
	
Let $D_1,\cdots,D_n$ be the disks given by Lemma \ref{lem_Seifert_disks_of_Ki},  then after a generic perturbation, the sequence $(D_1,\cdots,D_n)$ is admissible, so admissible sequences exist.
	Take a sequence of admissible disks with minimal complexity. We will abuse notation and still denote this sequence by  $(D_1,\cdots,D_n)$.
	
	For each $i$, if $D_i$ is not embedded, then the admissibility condition implies that the self-intersection locus of $D_i$ is a disjoint union of finitely many circles. Resolving the singularity of $D_i$ along one of the circles yields a sequence of admissible disks with smaller complexity.  Therefore, by the assumption on the minimality of complexity, every disk $D_i$ is embedded. 
	
	If the intersection of $D_i$ and $D_j$ ($i\neq j$) contains circles, then resolving the singularity along one of the circles yields a sequence of disks with smaller complexity.  Since all the disks $D_i$ are  embedded, the resulting sequence is still  admissible, contradicting the assumption.  
	
	In conclusion, for an admissible sequence with minimal complexity, every disk $D_i$ is embedded, and the intersection of $D_i$ and $D_j$ ($i\neq j$) only contains arcs. Hence the statement is proved.
\end{proof}

%% file: cycles.tex

This section proves Proposition \ref{prop_cycle_length_neq_4}.  By Proposition \ref{prop_2_3_components}, we only need to prove the following statement.

\begin{Proposition}
\label{prop_cycle_length_greater_than_5}
	Suppose $n\ge 5$ and $L$ is an $n$--component link such that the linking graph of $L$ is a cycle, the $L$ does not have minimal $\II^\natural$.
\end{Proposition}

\begin{figure}
	\includegraphics[width = 0.4\textwidth]{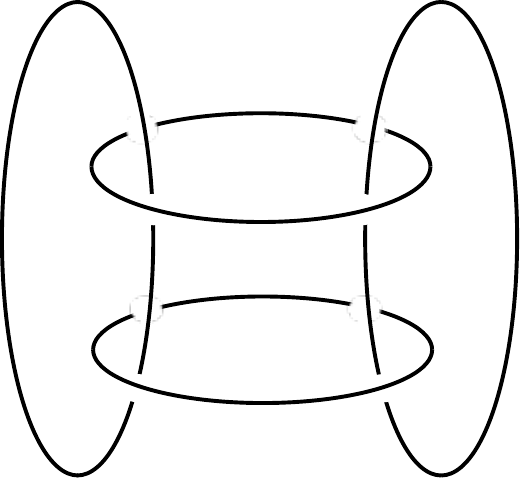}
	\caption{The link L8n8}
	\label{fig_L8n8}
\end{figure}

We will also prove the following result.
\begin{Proposition}
\label{prop_L8n8}
	Suppose $L$ is a link with 4 components such that the linking diagram of $L$ is a cycle and $L$ has minimal $\II^\natural$. Then $L$ is isotopic to L8n8, which is the link shown in Figure \ref{fig_L8n8}.
\end{Proposition}

It is not clear to the authors how to compute the instanton Floer homology for the link L8n8. We conjecture that L8n8 does not have minimal $\II^\natural$.

\begin{Conjecture}
\label{conj_l8n8}
The link L8n8 does not have minimal $\II^\natural$.
\end{Conjecture}

\begin{remark}
\label{rmk_after_conj}
If Conjecture \ref{conj_l8n8} is true, then by Propositions \ref{prop_2_3_components}, \ref{prop_cycle_length_greater_than_5}, \ref{prop_L8n8}, the linking graph of a link with minimal $\II^\natural$ cannot be a cycle. Hence by Proposition \ref{prop_sublink_minimal_I}, if a link has minimal $\II^\natural$, then its linking graph is a forest. By Proposition \ref{prop_tree_implies_Hopf_connected_sum}, this link is a forest of unknots. This would imply the main theorem by Proposition \ref{prop_abelian_rep_imply_minimal_I_Z}.

When $L$ is the link L8n8, Proposition \ref{prop_I>Alexander} does not yield any non-trivial lower bound for $\dim \II^\natural(L,p)$ from the Alexander polynomials.
 By \eqref{eqn_bound_I_from_KHI}, Conjecture \ref{conj_l8n8} would be true if 
 $$\dim_\bC\KHI(L;\bC)>64$$
  for the link L8n8. Kronheimer and Mrowka \cite[Conjecture 7.25]{KM:suture} conjectured that $\KHI(L)$ is isomorphic to the Heegaard Knot Floer homology defined by Ozsv\'ath-Szab\'o and Rasmussen \cite{OS:HFK, Ras:HFK}, which is more computable. The authors have computed the rank of $\HFK(L;\bZ/2)$ for the link L8n8 using grid homology and found that $\rank\HFK(L;\bZ/2) = 68 > 64$.
\end{remark}

The proofs of Proposition \ref{prop_cycle_length_greater_than_5} and Proposition  \ref{prop_L8n8} are divided into two parts. The first part shows that for a given $n\ge 4$, if $L$ is a link with $n$ components such that the linking graph is a cycle, then the topological properties from  Section \ref{sec_top_properties} imply that $L$ must be isotopic to one of the two special links.  The second part shows that when $n\neq 4$, both links have non-minimal $\II^\natural$. The first part is essentially  contained in \cite{XZ:forest}, and we give a sketch of the argument in Section \ref{subsec_isotopy_L_cycle}.  The second part is given in Section \ref{subsec_lower_bound_L_n_1-n_2-n}.

\subsection{Fixing the isotopy class of $L$}
\label{subsec_isotopy_L_cycle}

Suppose $L$ is a link with $n\ge 4$ components such that the linking graph is a cycle, write  $L=K_1\cup\cdots\cup K_n$, and let $L':=K_1\cup\cdots\cup K_{n-1}$. By Proposition \ref{prop_sublink_minimal_I} and Proposition \ref{prop_tree_implies_Hopf_connected_sum}, we may assume that the link $L'$ is given by Figure \ref{fig_standard_chain}. 

By Proposition \ref{prop_disks_in_minimal_position}, there exists a sequence of embedded disks $D_1,\cdots,D_n$, such that 
\begin{enumerate}
	\item $\partial D_i=K_i$,
	\item $D_i$ and $K_j$ intersect transversely for all $i\neq j$,
	\item $D_i$ and $D_j$ are disjoint if $|i-j|\neq 1$ or $n-1$, and intersect transversely at an arc otherwise.
\end{enumerate} 
Since $n\ge 4$ and the linking graph of $L$ is a cycle, we have $D_i\cap D_j\cap D_k = \emptyset$ for all distinct $i,j,k$.

After an isotopy on $D_1\cup\cdots\cup D_{n-1}$, we may assume that each component of $L'$ in Figure \ref{fig_standard_chain} is contained in a (flat) plane, and that the disk $D_i$ ($1\le i\le n-1$) is the disk bounded by $K_i$ in the corresponding plane. Since $D_n$ intersects each of $D_1$ and $D_{n-1}$ at an arc and is disjoint from $D_2,\cdots, D_{n-2}$, the component $K_n$ can be isotoped in $S^3-L'$ to a position described by Figure \ref{fig_cable}.

\begin{figure}
\includegraphics[width=0.8\textwidth]{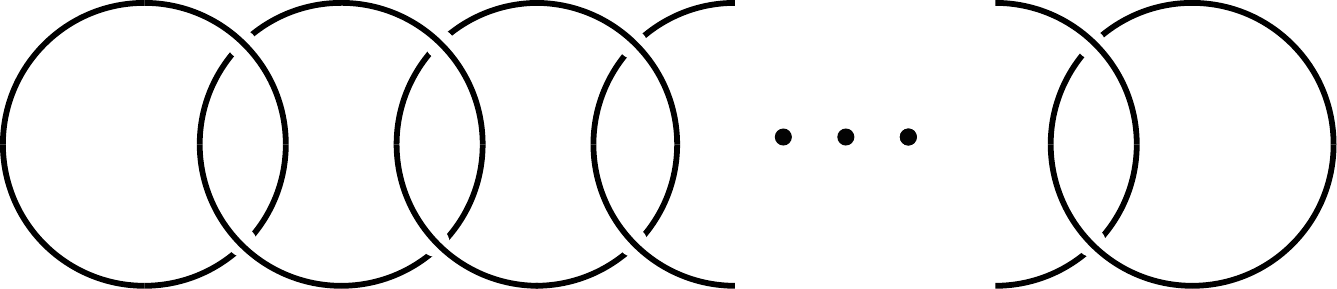}
\caption{The link $L'$}
\label{fig_standard_chain}

\vspace{\baselineskip}

\begin{overpic}[width=0.8\textwidth]{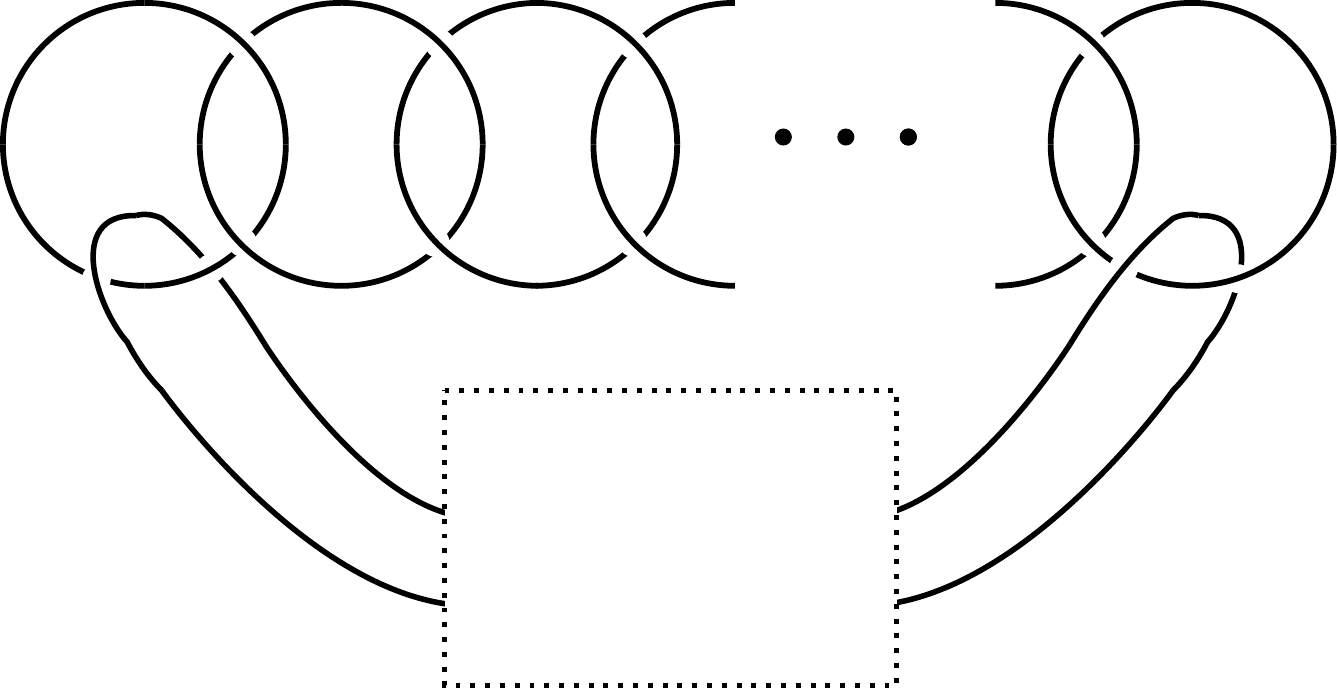}
\put(25,50){$K_n$}
\put(0,140){$L'$}
\end{overpic}
\caption{Isotopy class of $K_n$ in $S^3-L'$}
\label{fig_cable}
\end{figure}

On the other hand, let $S_1$ be the Seifert surface of $L'$ shown in Figure \ref{fig_S1}, and let $S_2$ be the Seifert surface of $L'$ shown in Figure \ref{fig_S2}.  
By the assumption on the linking graph of $L$, the algebraic intersection number of $K_n$ with either $S_1$ or $S_2$ is zero.
By Proposition \ref{prop_disjoint_from_Seifert_surface}, after an isotopy of $K_n$ in $S^3-L'$, the component $K_n$ is disjoint from either $S_1$ or $S_2$.

\begin{figure}
\includegraphics[width=0.8\textwidth]{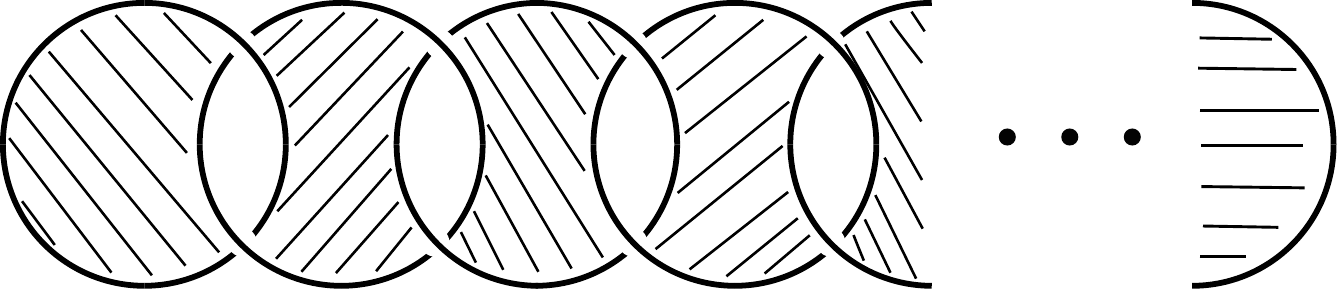}
\caption{Seifert surface $S_1$}
\label{fig_S1}

\vspace{\baselineskip}

\includegraphics[width=0.8\textwidth]{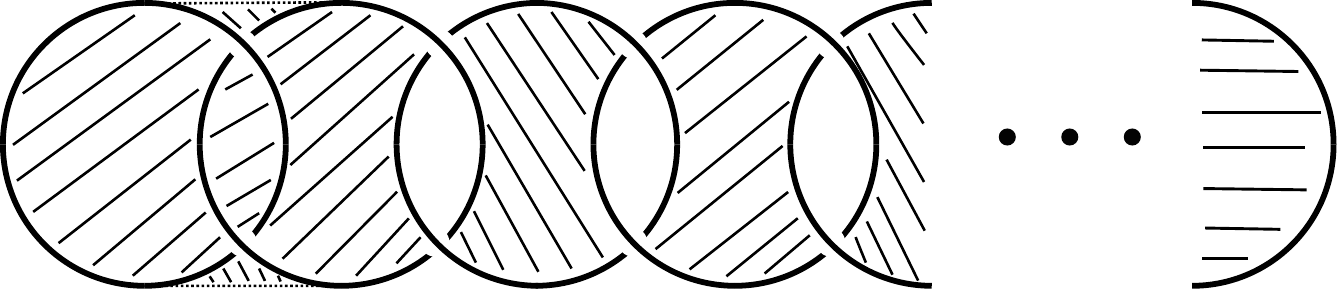}
\caption{Seifert surface $S_2$}
\label{fig_S2}
\end{figure}

Lemma \ref{lem_cycle_linking_graph_two_isotopy_class} below shows that the properties listed above almost uniquely specifies the isotopy class of $L$.  We need to introduce one more notation before stating the lemma.
For $u\ge 3, v\in\bZ$, we define the link $L_{u,v}$ as in \cite{XZ:forest}: 
if $v\ge 0$, then $L_{u,v}$ denotes the link given by Figure \ref{fig_Luv1} with $u$ components and $v$ crossings in the dotted rectangle; if $v<0$, then $L_{u,v}$ denotes the link given by Figure \ref{fig_Luv2} with $u$ components and $-v$ crossings in the dotted rectangle.
\begin{figure}  
  \includegraphics[width=0.7\linewidth]{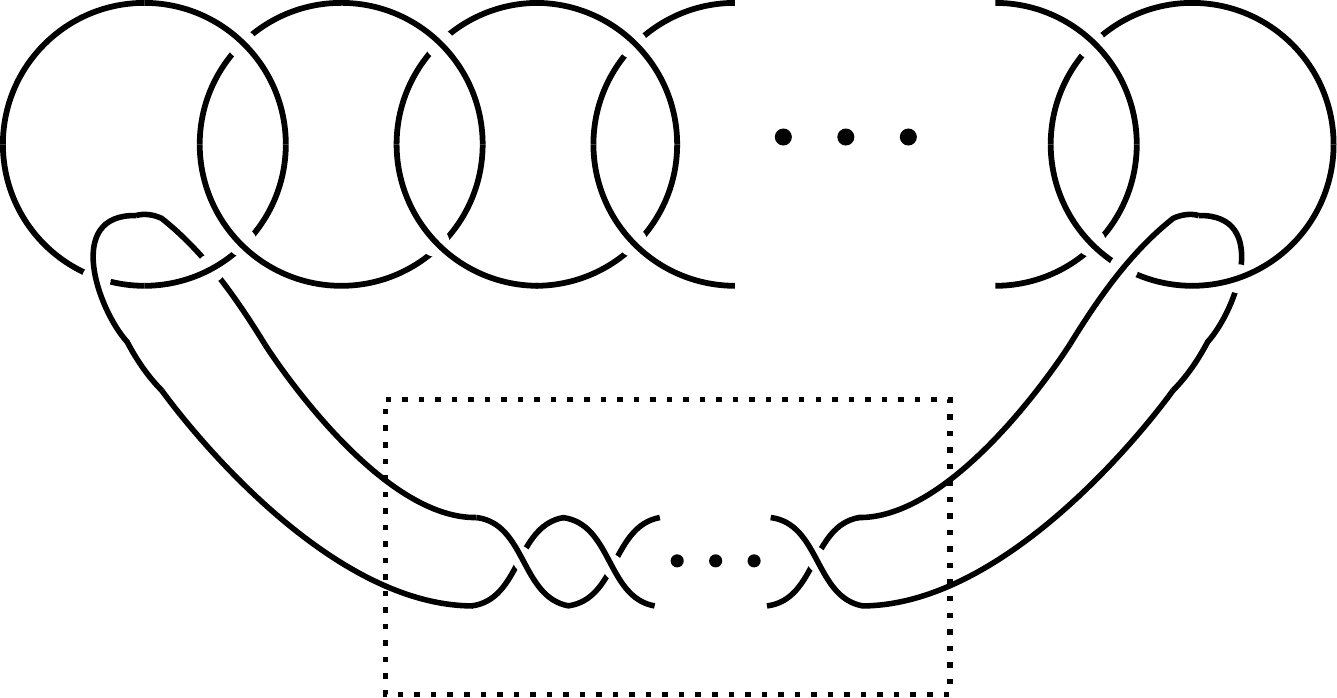}
  \caption{$L_{u,v}$ when $v\ge 0$.}
  \label{fig_Luv1}
  \vspace{\baselineskip}
  \includegraphics[width=0.7\linewidth]{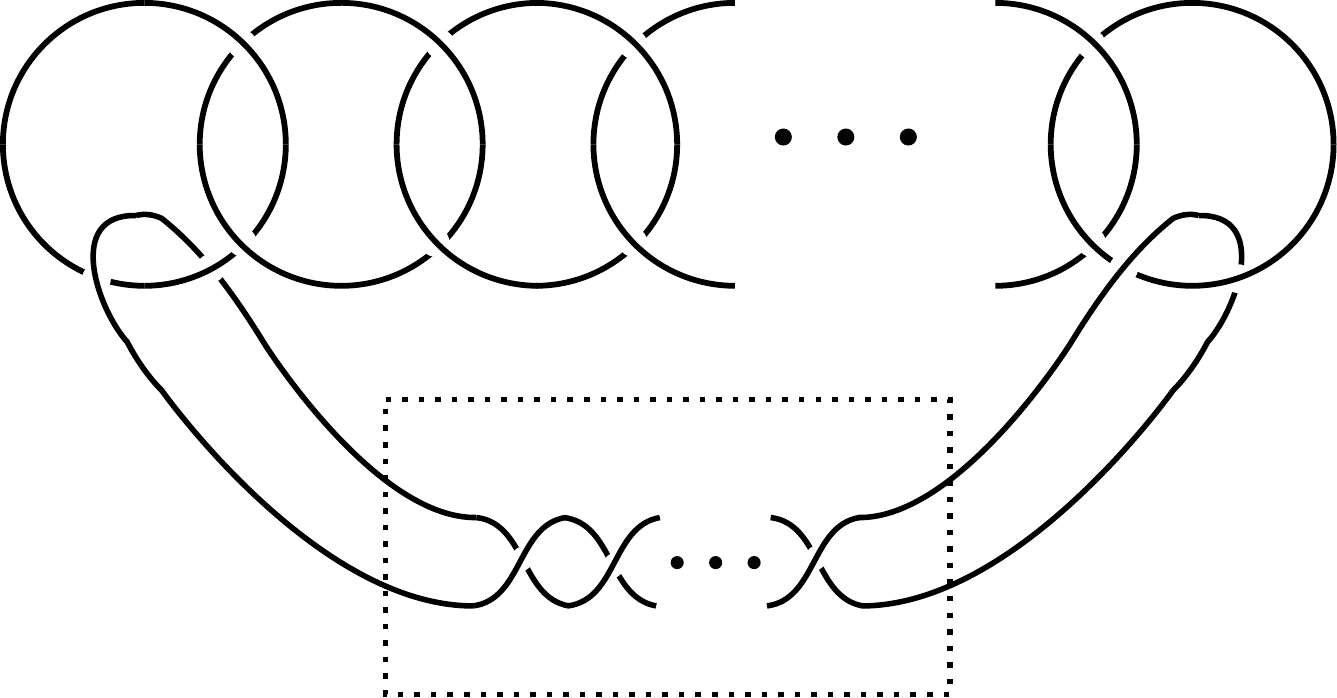}
  \caption{$L_{u,v}$ when $v<0$.}
  \label{fig_Luv2}
\end{figure}

\begin{Lemma}
\label{lem_cycle_linking_graph_two_isotopy_class}
Suppose $n\ge 3$ and  a link $L=K_1\cup \cdots\cup K_n$ satisfies the following properties:
\begin{enumerate}
	\item $L':=K_1\cup\cdots K_{n-1}$ is the link given by Figure \ref{fig_standard_chain},
	\item $K_n$ bounds a disk that is disjoint from $K_2\cup\cdots\cup K_{n-2}$, and intersects each of $K_1$ and $K_{n-1}$ transversely at one point,
	\item $K_n$ can be homotoped in $S^3-L'$ to a position described by Figure \ref{fig_cable}.
\end{enumerate}
Then we have
\begin{enumerate}
	\item[(i)] if $K_n$ can be isotoped in $S^3-L'$ to be disjoint from $S_1$, then the link $L$ is isotopic to the link $L_{n,1-n}$,
	\item[(ii)] if $K_n$ can be isotoped in $S^3-L'$ to be disjoint from $S_2$, then the link $L$ is isotopic to the link $L_{n,2-n}$.
\end{enumerate}
\end{Lemma}

\begin{remark}
	Although the statement of Lemma \ref{lem_cycle_linking_graph_two_isotopy_class} holds for all $n\ge 3$, we only need the case for $n\ge 5$ for the proof of Proposition \ref{prop_cycle_length_greater_than_5}, and the case $n=4$ for the proof of Proposition \ref{prop_L8n8}.
\end{remark}

\begin{proof}[Sketch of proof]
	The proof of the lemma was given in \cite[Sections 8,9,11]{XZ:forest} without explicitly stating the result. We give an outline of the proof and refer the reader to \cite{XZ:forest} for details. The argument does not use gauge theory. We will focus on Case (i) where $K_n$ is disjoint from $S_1$ up to an isotopy in $S^3-L'$. The other case is similar.
	
	Isotope $K_n$ in $S^3-L'$ such that $K_n$ is disjoint from $S_1$.
	Let $D$ be the disk given by Assumption (2), perturb $D$ such that $D$ intersects $S_1$ and $\partial S$ transversely.  Then $D\cap S_1$ is the disjoint union of an arc and finitely many circles. We denote the arc component by $\gamma$. One can shrink $K_n$ along $D$ into a small neighborhood of $\gamma$, therefore the isotopy class of $K_n$ in $S^3-L'$ is uniquely determined by the isotopy class of $\gamma$ in $S_1$ (relative to $\partial S_1$). 
	
	The homotopy class of $K_n$ defines a conjugacy class in $\pi_1(S^3-L')$, and Assumption (3) implies that the homotopy class of $K_n$ is conjugate to $g_1^{\epsilon_1}g_{n-1}^{\epsilon_2}$, where $\epsilon_i=\pm 1$ (see \cite[Section 8]{XZ:forest} for the definitions of $g_1,g_{n-1}\in\pi_1(S^3-L')$).   By solving an equation in $\pi_1(S^3-L')$, this implies that $\gamma$ is \emph{homotopic} to one of the standard arcs in $S_1$ relative to $\partial S_1=L'$ (see \cite[Lemma 11.2]{XZ:forest}), and hence it is \emph{isotopic} to one of these arcs (see \cite{feustel1966homotopic} or \cite[Proposition 9.1]{XZ:forest}).  Recall that the isotopy class of $L$ is determined by the arc $\gamma$.  It turns out that all these standard arcs give the same link $L$ up to isotopy (see \cite[Corollary 11.3]{XZ:forest}), and this link is isotopic to $L_{n,1-n}$ (see \cite[Lemma 11.7]{XZ:forest}).
	
	A similar argument shows that $L$ is isotopic to $L_{n,2-n}$ in Case (ii).
\end{proof}

\subsection{Lower bound for $\dim\II^\natural$}
\label{subsec_lower_bound_L_n_1-n_2-n}
\begin{Proposition}
\label{prop_lower_bound_L_n_1-n_2-n}
	Suppose $m\ge 4$, $n\ge 5$. Then for each $p\in L_{m,1-m}$ and $q\in L_{n,2-n}$, we have $$\dim_\bC\II^\natural(L_{m,1-m},p;\bC)>2^{m-1},$$
	 $$\dim_\bC\II^\natural(L_{n,2-n},q;\bC)>2^{n-1}.$$
\end{Proposition}

\begin{proof}
	It was proved in \cite[Section 10]{XZ:forest} that 
$\det L_{u,v}=2^{u-1}|u+2v|.$  Therefore, the desired result follows from Proposition \ref{prop_I>Alexander}  except for the link $L_{5,-3}$. The case for $L_{5,-3}$ is proved by Lemma \ref{lem_lower_bound_I_L5-3} in the appendix.
\end{proof}

\begin{proof}[Proof of Proposition \ref{prop_cycle_length_greater_than_5} and Proposition \ref{prop_L8n8}]
	Let $n\ge 4$, and suppose $L$ is an $n$--component link with minimal $\II^\natural$ such that the linking graph of $L$ is a cycle. By Section \ref{subsec_isotopy_L_cycle}, $L$ is isotopic to $L_{n,1-n}$ or $L_{n,2-n}$. By Proposition \ref{prop_lower_bound_L_n_1-n_2-n}, we must have $n=4$ and $L$ is isotopic to $L_{4,-2}$.  The link $L_{4,-2}$ is isotopic to L8n8.
\end{proof}

%% file: eliminate_G0.tex
Recall that $G_0$ is the graph in Figure \ref{fig_G0_graph}. This section proves Proposition \ref{prop_no_G0}. 

\begin{repProposition}{prop_no_G0}
If $L$ is a link with 6 components such that its linking graph is isomorphic to $G_0$, then $L$ does not have minimal $\II^\natural$.
\end{repProposition}

The strategy of the proof of Proposition \ref{prop_no_G0} is similar to that of Proposition \ref{prop_cycle_length_greater_than_5}.  For the rest of this section, we use $L$ to denote a hypothetical 6-component link that has minimal $\II^\natural$ and has linking graph isomorphic to $G_0$. We first use the results from Section \ref{sec_top_properties} to show that $L$ must be isotopic to a particular link, then use Proposition \ref{prop_I>Alexander} to show that this link does not have minimal $\II^\natural$, which yields a contradiction.

This section is organized as follows. In Section \ref{subsec_G0_topology}, we obtain several topological properties of $L$ using results from Section \ref{sec_top_properties}. In Section \ref{subsec_G0_isotopy_theta}, we show that these topological properties imply that $L$ is isotopic to the link given by Figure \ref{fig_link_G0}. We then finish the proof by invoking Lemma \ref{lem_lower_bound_I_G0} from the appendix, which states that the link in Figure \ref{fig_link_G0} does not have minimal $\II^\natural$. Section  \ref{subsec_G0_lemma_free_group} proves a technical lemma that is used in Section \ref{subsec_G0_isotopy_theta}.

\subsection{Topological properties of $L$}
\label{subsec_G0_topology}
Label the components of $L$ by $K_1,\cdots,K_6$ as in Figure \ref{fig_linking_graph_with_labels}. Let $L':=K_3\cup K_4\cup K_5\cup K_6$, let $L'':=K_1\cup K_2$. By Proposition \ref{prop_sublink_minimal_I} and Proposition \ref{prop_tree_implies_Hopf_connected_sum}, the link $L'$ is isotopic to the link given by Figure \ref{fig_tree_in_G0}.

\begin{figure}
\vspace{\baselineskip}
\begin{overpic}[width=0.4\textwidth]{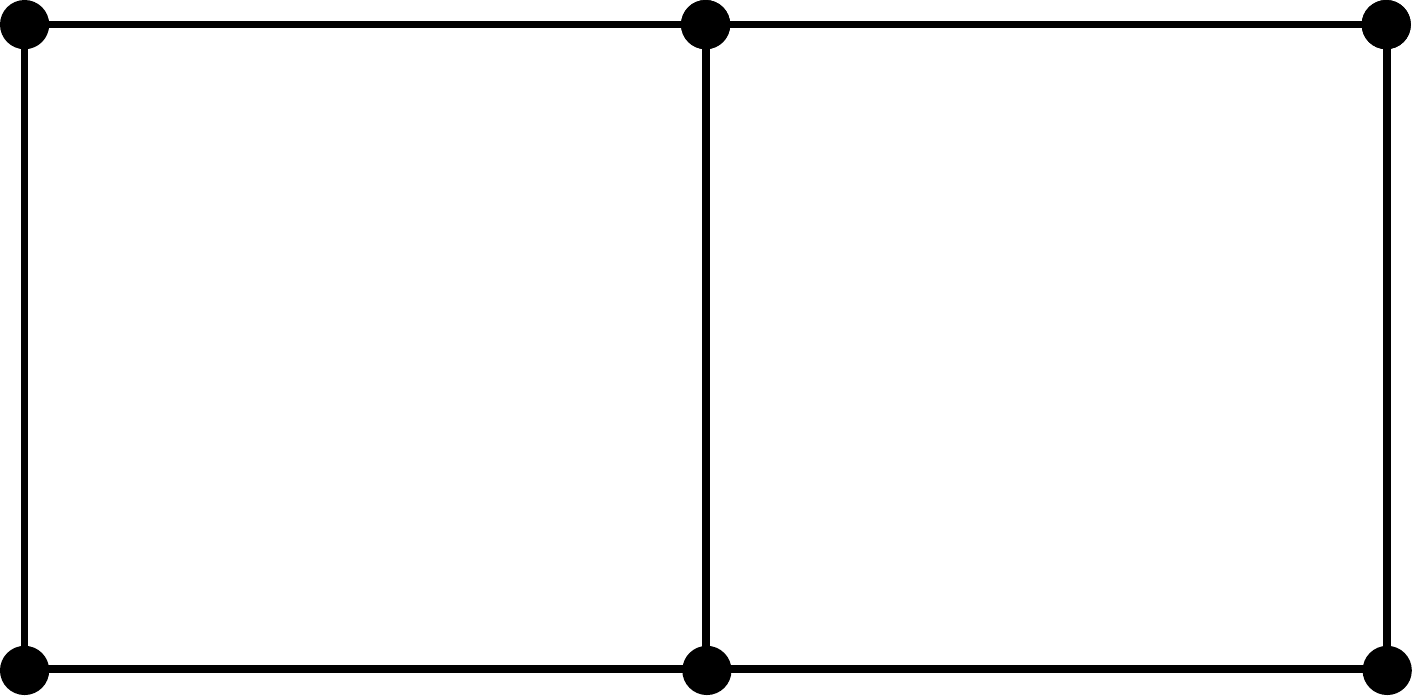}
\put(-5,-12){$K_1$}
\put(65,-12){$K_6$}
\put(135,-12){$K_2$}
\put(-5,75){$K_3$}
\put(65,75){$K_4$}
\put(135,75){$K_5$}
\end{overpic}
\vspace{\baselineskip}
\caption{The linking graph of $L$}
\label{fig_linking_graph_with_labels}
\end{figure}

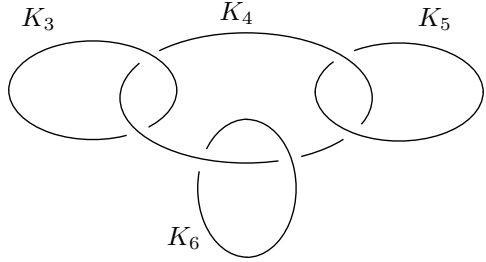
\begin{figure}
\vspace{\baselineskip}
\begin{overpic}[width=0.5\textwidth]{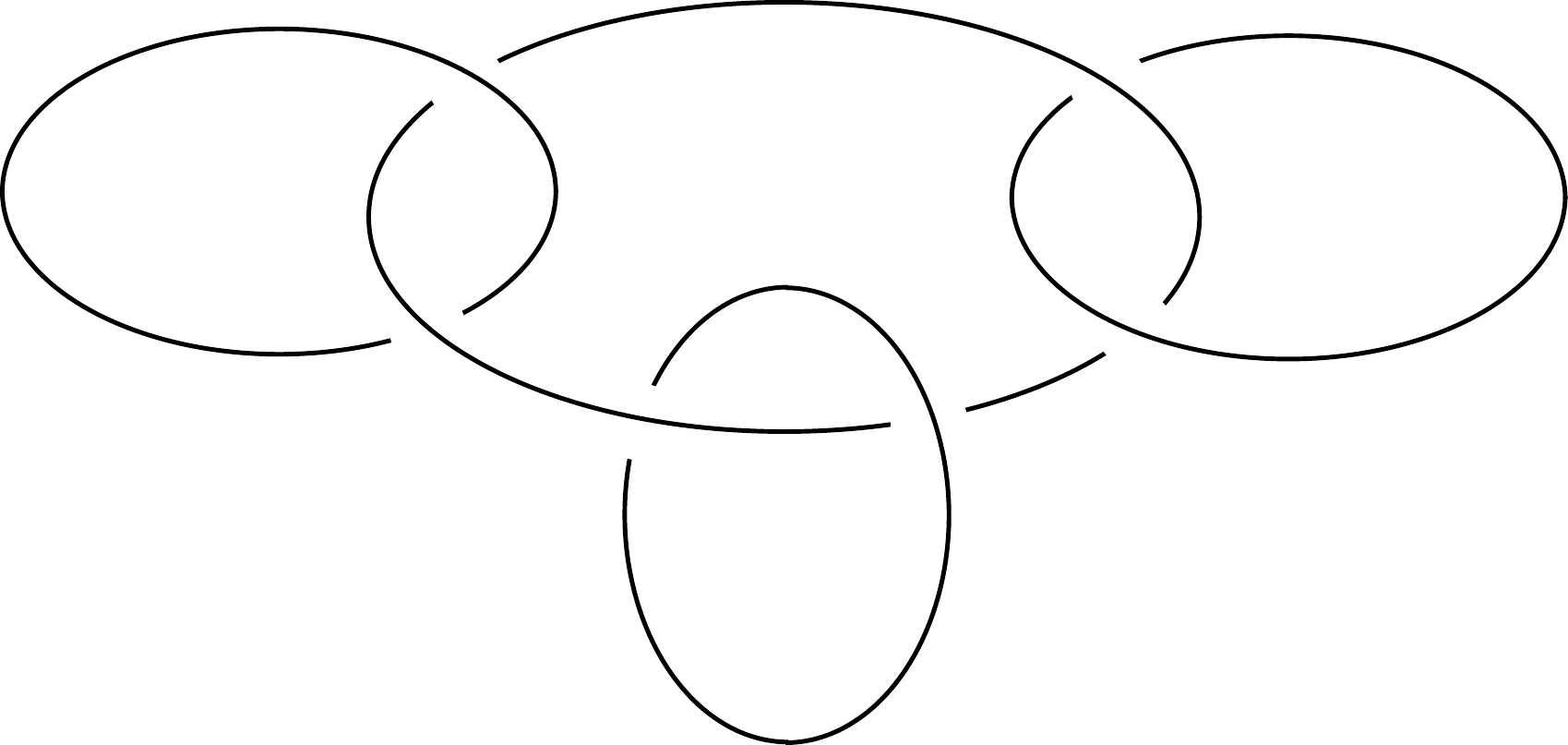}
\put(5,88){$K_3$}
\put(80,90){$K_4$}
\put(155,88){$K_5$}
\put(60,5){$K_6$}
\end{overpic}
\caption{The link $L'$}
\label{fig_tree_in_G0}
\end{figure}

Let $D_i$ ($i=1,\cdots,6$) be the disks given by Proposition \ref{prop_disks_in_minimal_position}.  Recall that we have
\begin{enumerate}
	\item $D_i$ is an embedded disk with $\partial D_i=K_i$ for all $i$,
	\item $D_i$ is transverse to $K_j$ for all $i\neq j$,
	\item $D_i$ is disjoint from $D_j$ if $\lk(K_i,K_j)=0$, and $D_i$ intersects $D_j$ transversely at an arc if $|\lk(D_i,D_j)|=1$.
\end{enumerate} 
After an isotopy, we may assume that the components of $L'\subset L$ is given by Figure \ref{fig_tree_in_G0},  every component of $L'$ is contained in a (flat) plane, and the disk $D_i$ is the disk bounded by $K_i$ in the respective plane for $i=3,\cdots,6$. Since $D_1$ is disjoint from $D_4,D_5$ and intersects each of $D_3$ and $D_6$ at an arc, we have
\begin{equation}
\label{eqn_describe_isotopy_class_of_K1}
	K_1 \text{ can be isotoped in $S^3-L'$ to a position described by Figure \ref{fig_tree_with_K1} }. 
\end{equation}
Similarly, we have
\begin{equation}
\label{eqn_describe_isotopy_class_of_K2}
	K_2 \text{ can be isotoped in $S^3-L'$ to a position  described by Figure \ref{fig_tree_with_K2} }. 
\end{equation}

\begin{figure}
\vspace{\baselineskip}
\begin{overpic}[width=0.5\textwidth]{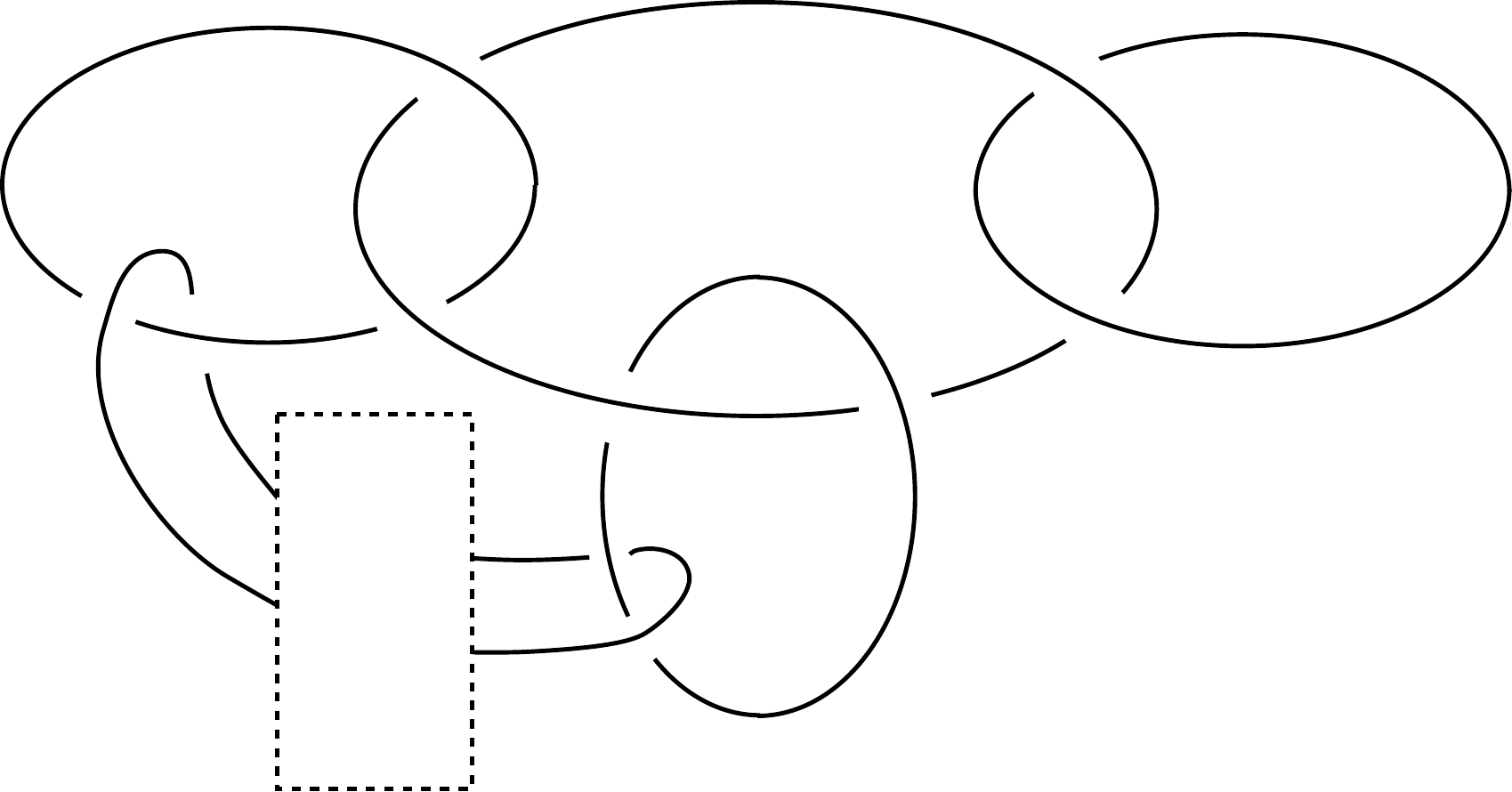}
\put(5,95){$K_3$}
\put(80,100){$K_4$}
\put(155,95){$K_5$}
\put(112,20){$K_6$}
\put(15,18){$K_1$}
\end{overpic}
\caption{The isotopy class of $K_1$ in $S^3-L'$}
\label{fig_tree_with_K1}

\vspace{\baselineskip}

\vspace{\baselineskip}
\begin{overpic}[width=0.5\textwidth]{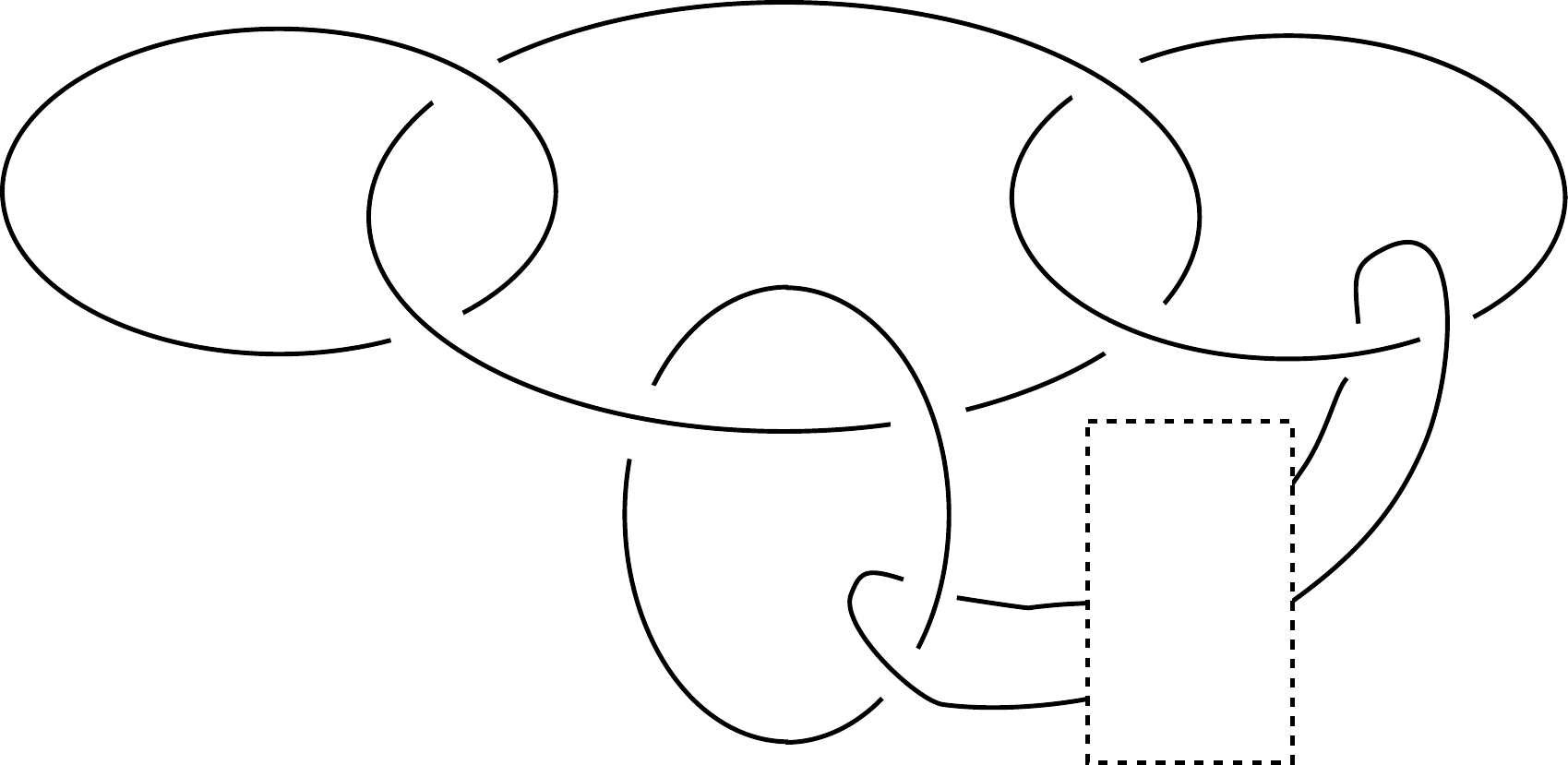}
\put(5,88){$K_3$}
\put(80,92){$K_4$}
\put(155,88){$K_5$}
\put(60,10){$K_6$}
\put(152,10){$K_2$}
\end{overpic}
\caption{The isotopy class of $K_2$ in $S^3-L'$}
\label{fig_tree_with_K2}
\end{figure}

Now consider the value of
$$
	\lk(K_3,K_4)\cdot \lk(K_4,K_6) \cdot \lk(K_6,K_1)\cdot \lk(K_1,K_3).
$$
The value of the product above is $\pm 1$ and it does not depend on the choice of the orientation of $L$. 
By Proposition \ref{prop_sublink_minimal_I} and Proposition \ref{prop_L8n8}, we have
\begin{equation}
\label{eqn_prod_linking_number_K1}
	\lk(K_3,K_4)\cdot \lk(K_4,K_6) \cdot \lk(K_6,K_1)\cdot \lk(K_1,K_3)=1.
\end{equation}
 Similarly, we have
\begin{equation}
\label{eqn_prod_linking_number_K2}
\lk(K_4,K_5)\cdot \lk(K_5,K_2) \cdot \lk(K_2,K_6)\cdot \lk(K_6,K_4)=1.
\end{equation}

Let $S$ be the Seifert surface of $L'$ shown in Figure \ref{fig_G0_seifert}. Fix an orientation for $S$ and take the induced orientation on $L'= \partial S$, and  choose an arbitrary orientation for $K_1$ and $K_2$.   Then by \eqref{eqn_prod_linking_number_K1} and \eqref{eqn_prod_linking_number_K2}, we have
$$
\lk(K_1,K_3) =- \lk(K_1,K_6) , \quad \lk(K_2,K_5) =-\lk(K_2,K_6).
$$
Therefore, the algebraic intersection numbers of $S$ with $K_1$ and $K_2$ are both zero. By Proposition \ref{prop_disjoint_from_Seifert_surface}, the link $L''=K_1\cup K_2$ can be isotoped in $S^3-L'$ to a position that is disjoint from $S$. From now on, we  assume without loss of generality that $L''\cap S=\emptyset$.

\begin{figure}
\includegraphics[width=0.5\textwidth]{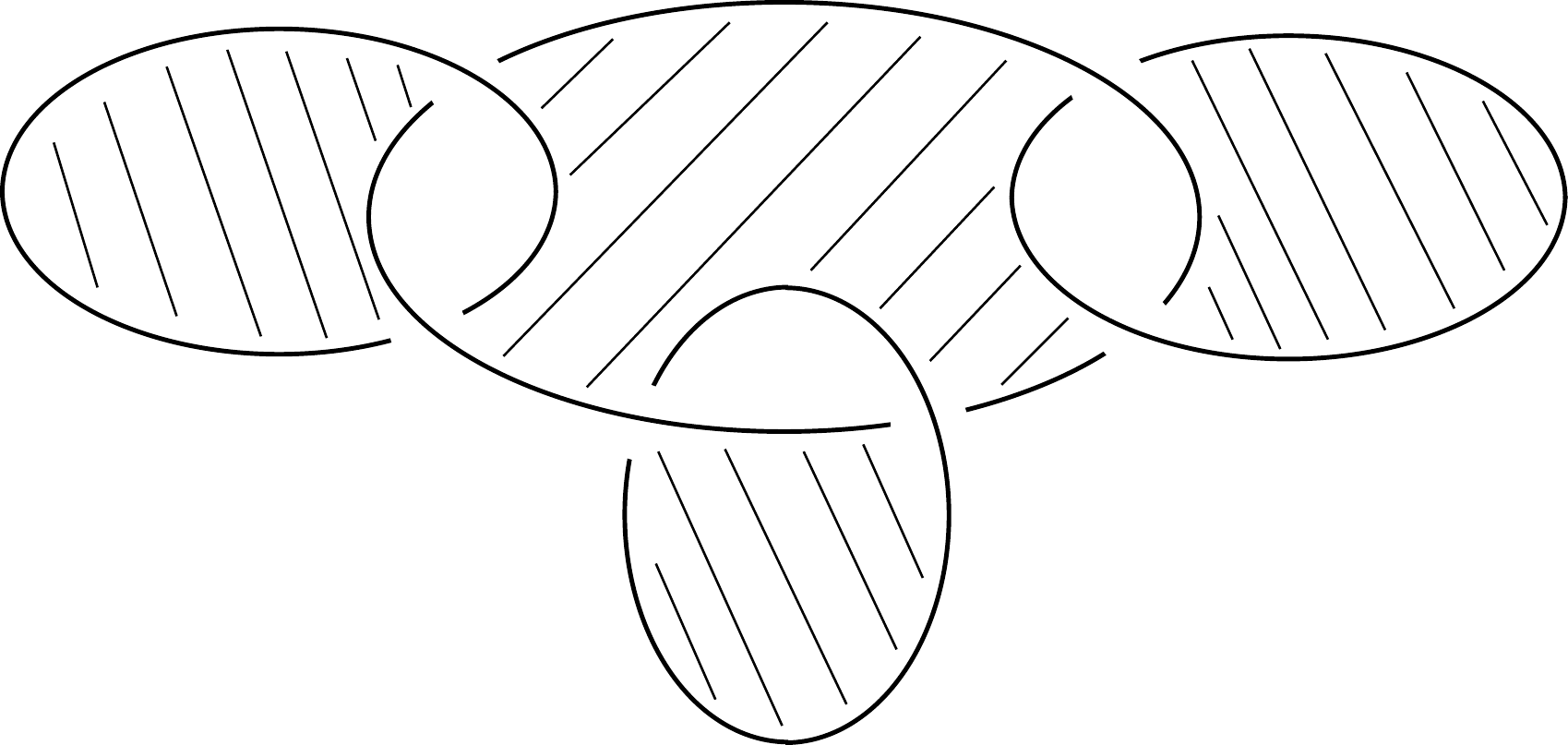}
\caption{The Seifert surface $S$ of $L'$.}
\label{fig_G0_seifert}
\end{figure}

We need the following lemma.
\begin{Lemma}
\label{lem_L(theta)}
Suppose $\theta= \theta_1\cup\cdots \cup \theta_n$ is a disjoint union of properly embedded arcs in $S$. Then there exists a unique link $L(\theta)$ up to isotopy in $S^3-L'$ such that the following properties hold:
\begin{enumerate}
	\item $L(\theta)$ has $n$ components $K(\theta_1),\cdots,K(\theta_n)$,
	\item every component $K(\theta_i)$ bounds an embedded disk $D(\theta_i)$,
	\item $D(\theta_i)$ intersects both $\partial S$ and $S$ transversely,
	\item $D(\theta_i)$ is disjoint from $D(\theta_j)$ when $i\neq j$,
	\item $D(\theta_i) \cap S = \theta_i$.
\end{enumerate}
\end{Lemma}

\begin{proof}
For the existence of $L(\theta)$, notice that 
for each $i$, the regular neighborhood $U_i$ of $\theta_i$ is diffeomorphic to $(-1,1)\times(-1,1)\times (-2,2)$,
and we may take a diffeomorphism $$\phi_i:U_i\to (-1,1)\times(-1,1)\times (-2,2)$$ such that 
$$\phi_i(\theta_i)= \{0\}\times\{0\}\times [-1,1],$$  
$$\phi_i(S\cap U_i)=\{0\}\times(-1,1)\times[-1,1].$$  
By shrinking the regular neighborhoods, we may also assume that $U_i\cap U_j=\emptyset$ for all $i\neq j$. 
Let $D$ be a disk in $(-1,1)\times\{0\}\times (-2,2)$ that contains $\{0\}\times\{0\}\times [-1,1]$.  One can construct the link $L(\theta)$ by taking $D(\theta_i):=\phi_i^{-1}(D)$ and $K(\theta_i):=\partial D(\theta_i)$.

The uniqueness of $L(\theta)$ follows from the observation that each $K(\theta_i)$ can be isotoped into a small neighborhood of $\theta_i$ through the disk $D(\theta_i)$.
\end{proof}

\begin{Definition}
Suppose $\theta$ is a disjoint union of properly embedded arcs on $S$. We use  $L(\theta)$ to denote the link in $S^3-L'$ given by Lemma \ref{lem_L(theta)}.
\end{Definition}

The following lemma is clear from the definition of $L(\theta)$.
\begin{Lemma}
\label{lem_L(theta)_under_isotopy}
	Suppose $\theta$ is a disjoint union of properly embedded arcs on $S$. Let $f:S\to S$ be a diffeomorphism. Suppose $f$ is isotopic to $\id_S$ via an isotopy on $S$ that does not necessarily fix $\partial S$. Then $L(\theta)$ is isotopic to $L(f(\theta))$ in $S^3-L'$. \qed
\end{Lemma}

Recall that $D_i$ is the embedded disk bounded by $K_i$ given by Proposition \ref{prop_disks_in_minimal_position}.  After a generic perturbation, the intersection of $D_1$ and $S$ is the disjoint union of an arc and finitely many circles. We denote the arc component of $D_1\cap S$ by $\theta_1$.  Similarly, take a generic perturbation on $D_2$ and let $\theta_2$ be the arc component of $D_2\cap S$. By shrinking $K_1$ to a small neighborhood of $\theta_1$ via $D_1$, and shrinking $K_2$ to a small neighborhood of $\theta_2$ via $D_2$, we see that $L''$ is isotopic to $L(\theta_1\cup\theta_2)$ in $S^3-L'$. 

\subsection{The isotopy class of $\theta_1\cup\theta_2$}
\label{subsec_G0_isotopy_theta}

We study the isotopy class of $\theta_1\cup\theta_2$ in $S$ (relative to $\partial S$) using \eqref{eqn_describe_isotopy_class_of_K1} and \eqref{eqn_describe_isotopy_class_of_K2}.
We will need the following two lemmas.

\begin{Lemma}
\label{lem_free_group_3}
Let $F_3$ be the free group generated by $a,b,c$. Suppose  $w\in F_3$.  If there exists $\epsilon_i\in\{1,-1\}$ such that
$a^{\epsilon_1}wb^{\epsilon_2} w^{-1}$ is conjugate to $a^{\epsilon_3}b^{\epsilon_4}$ in $F_3$, then $w$ is contained in the subgroup generated by $a,b$. 
\end{Lemma}

The proof of Lemma \ref{lem_free_group_3} will be given in Section \ref{subsec_G0_lemma_free_group}.

\begin{Lemma}
\label{lem_disjoint_homotopy_implies_isotopy}
Let $M$ be a surface with boundary. Let $\gamma_1$ and $\gamma_2$ be two properly embedded arcs on $M$ such that $\partial\gamma_1\cap \partial \gamma_2=\emptyset$. Then $\gamma_1$ can be isotoped (relative to $\partial M$) to a position disjoint from $\gamma_2$ if and only if it can be homotoped (relative to $\partial M$) to a position disjoint from $\gamma_2$. 
\end{Lemma}

\begin{proof}
The ``only if'' part is trivial. The ``if'' part is a corollary of the bigon criterion  (see, for example, \cite[Section 1.2.4 - 1.2.7]{farb2011primer}) as follows. Isotope $\gamma_1$ relative to $\partial M$ such that $\gamma_1$ intersects $\gamma_2$ transversely and $\gamma_1$ does not form a bigon with $\gamma_2$. By the bigon criterion,  $\gamma_1$ and $\gamma_2$ are in minimal position (i.e. they realize the minimal geometric intersection number up to \emph{homotopy} relative to $\partial M$). Therefore, $\gamma_1$ is disjoint from $\gamma_2$ after the isotopy.
\end{proof}

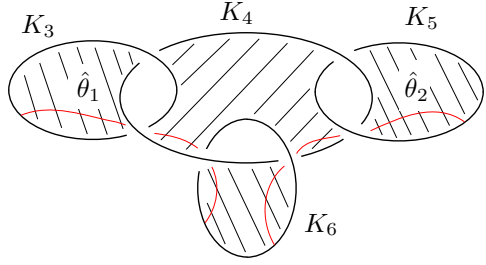
\begin{figure}
\vspace{\baselineskip}
\begin{overpic}[width=0.5\textwidth]{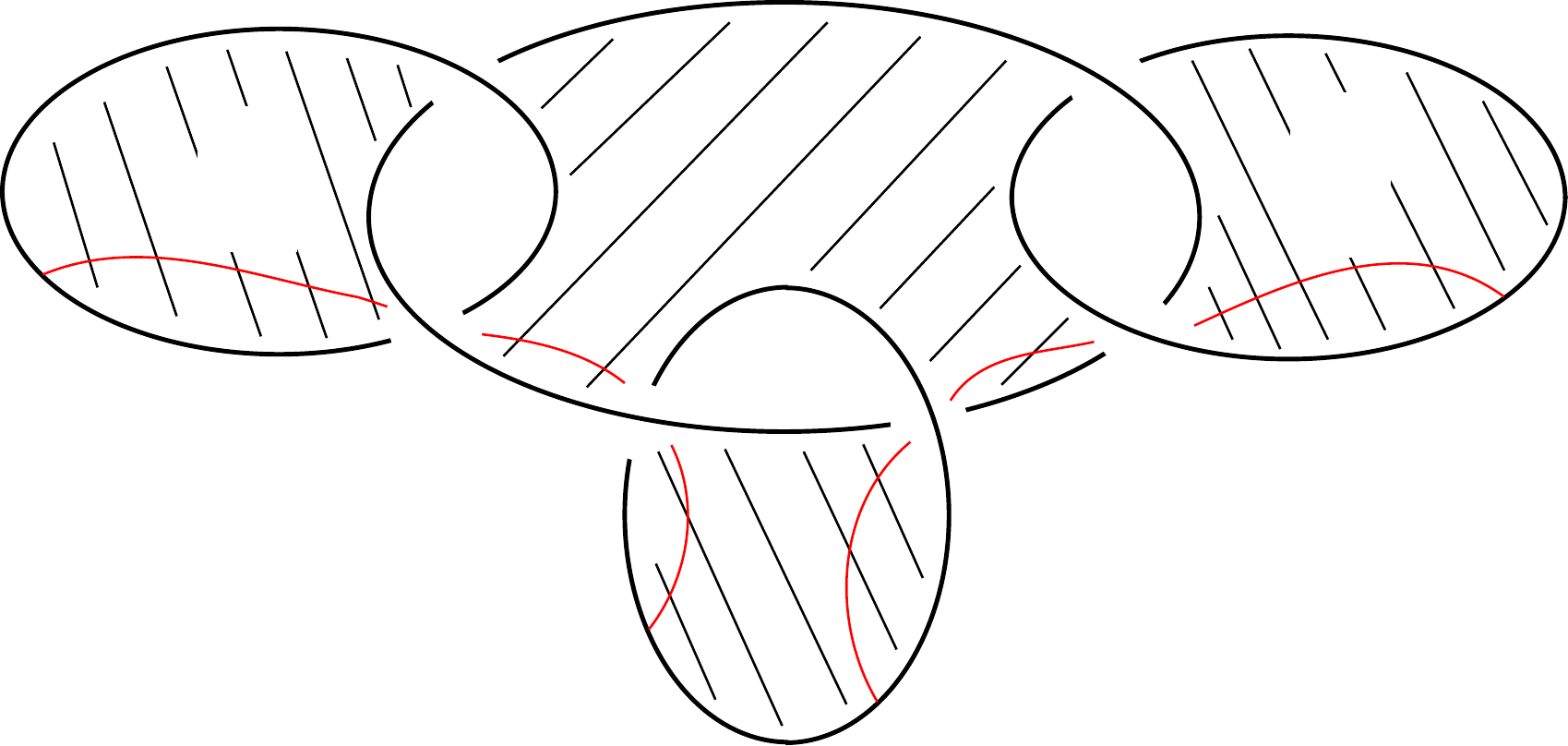}	
\put(26,61){$\hat\theta_1$}
\put(150,61){$\hat\theta_2$}
\put(5,85){$K_3$}
\put(80,90){$K_4$}
\put(150,88){$K_5$}
\put(112,10){$K_6$}
\end{overpic}
\caption{The arcs $\hat\theta_1$ and $\hat\theta_2$ on $S$.}
\label{fig_S_standard_arcs}
\end{figure}

By definition, $\theta_1$ is an arc on $S$ connecting $K_3$ and $K_6$, and $\theta_2$ is an arc on $S$ connecting $K_5$ and $K_6$.
Let $\hat \theta_1$ and $\hat \theta_2$ be the two arcs on $S$ given by Figure \ref{fig_S_standard_arcs}. By Lemma \ref{lem_L(theta)_under_isotopy}, we may assume without loss of generality that $\partial \theta_i = \partial \hat \theta_i$ for $i=1,2$.

\begin{figure}
\vspace{\baselineskip}
\begin{overpic}[width=0.8\textwidth]{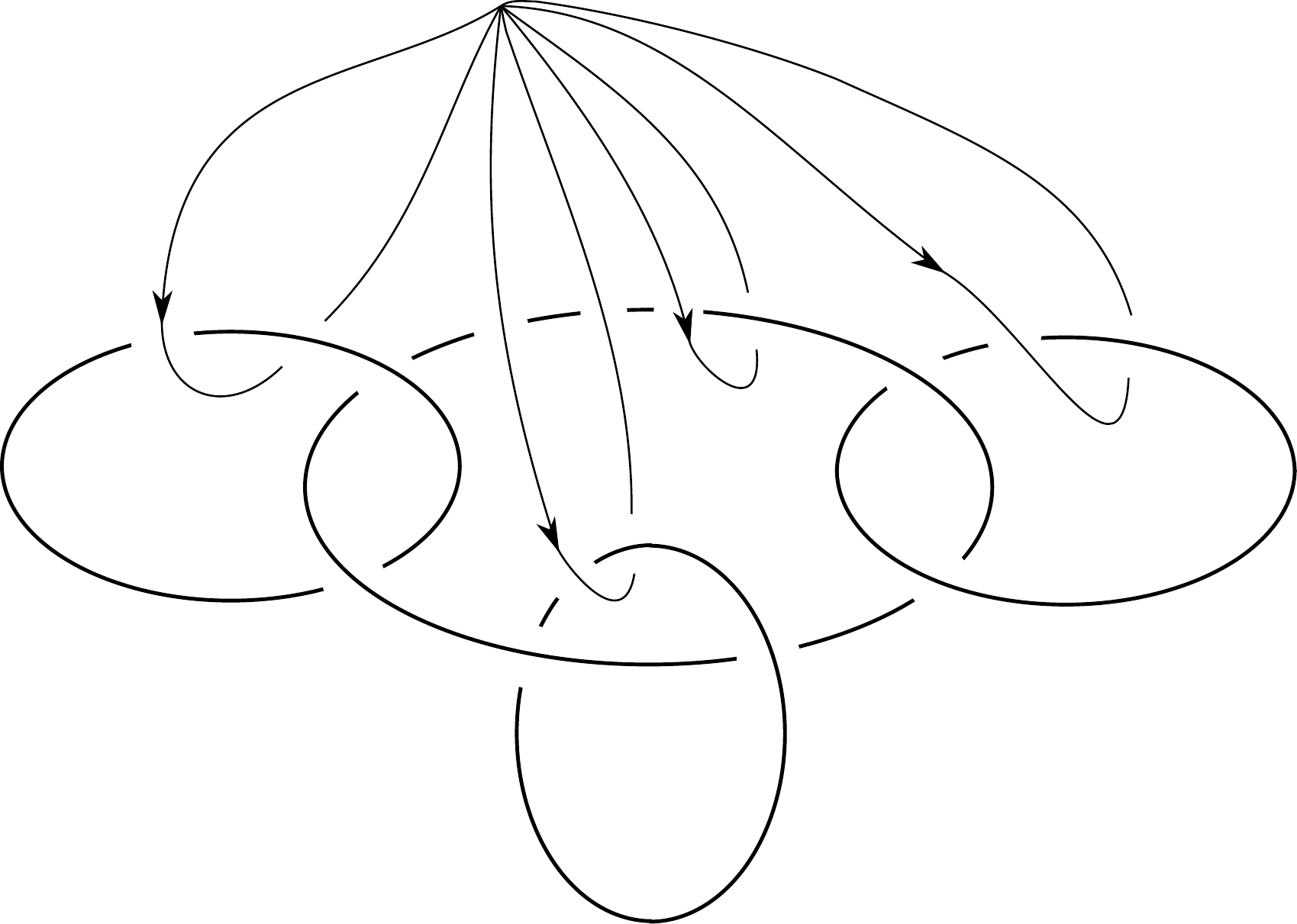}
\put(100,210){$p$}
\put(30,150){\Large $a$}
\put(122,100){\Large $b$}
\put(250,150){\Large $c$}
\put(165,150){\Large $d$}
\end{overpic}
\caption{Generators of $\pi_1(S^3-L')$.}
\label{fig_pi_1}
\end{figure}
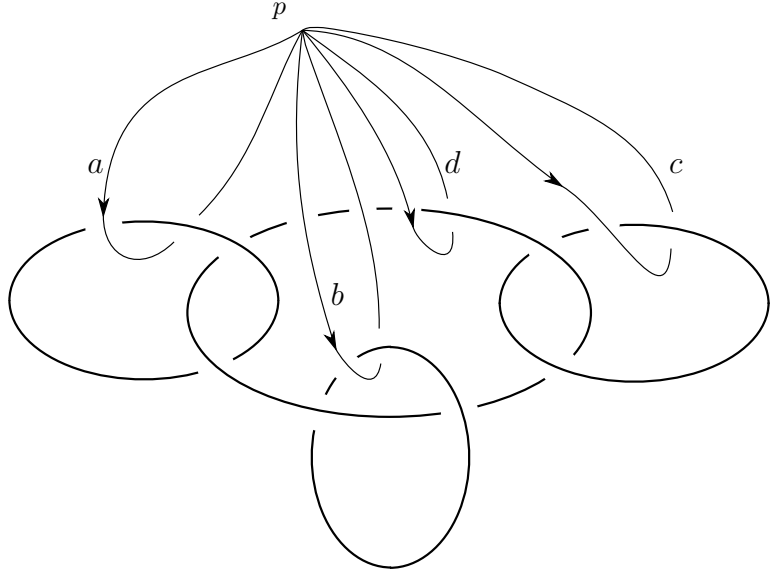

Let $p\in S^3-L'$ and define $a,b,c,d\in \pi_1(S^3-L',p)$ as in  Figure \ref{fig_pi_1}.  By the Wirtinger presentation, we have
$$
\pi_1(S^3-L', p) = \langle a,b,c,d | [a,d]=[b,d]=[c,d]=1\rangle.
$$

\begin{figure}
\begin{overpic}[width=0.5\textwidth]{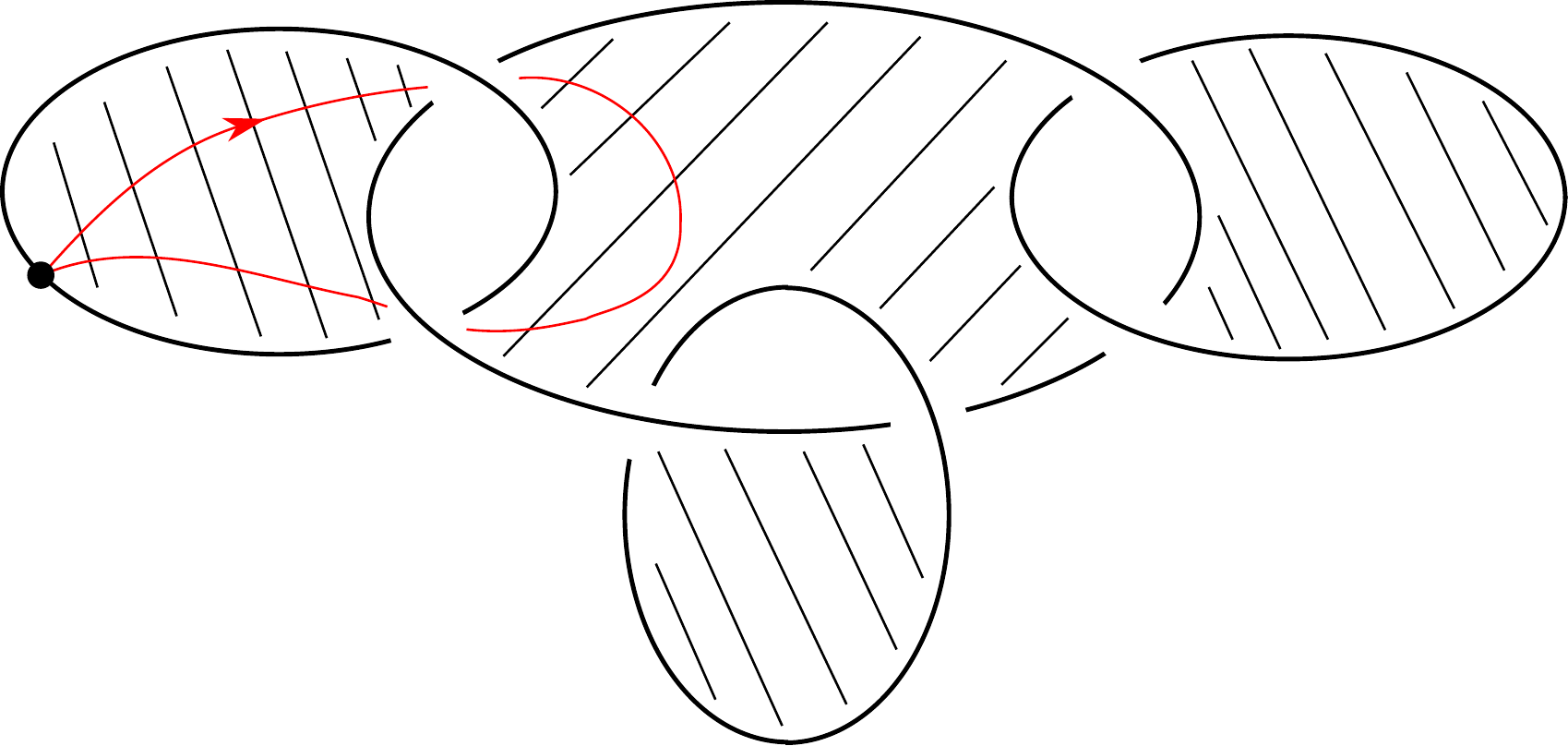}
\put(0,45){$q$}
\end{overpic}
\caption{The loop $\alpha$ on $S$.}
\label{fig_alpha}
\vspace{\baselineskip}

\begin{overpic}[width=0.5\textwidth]{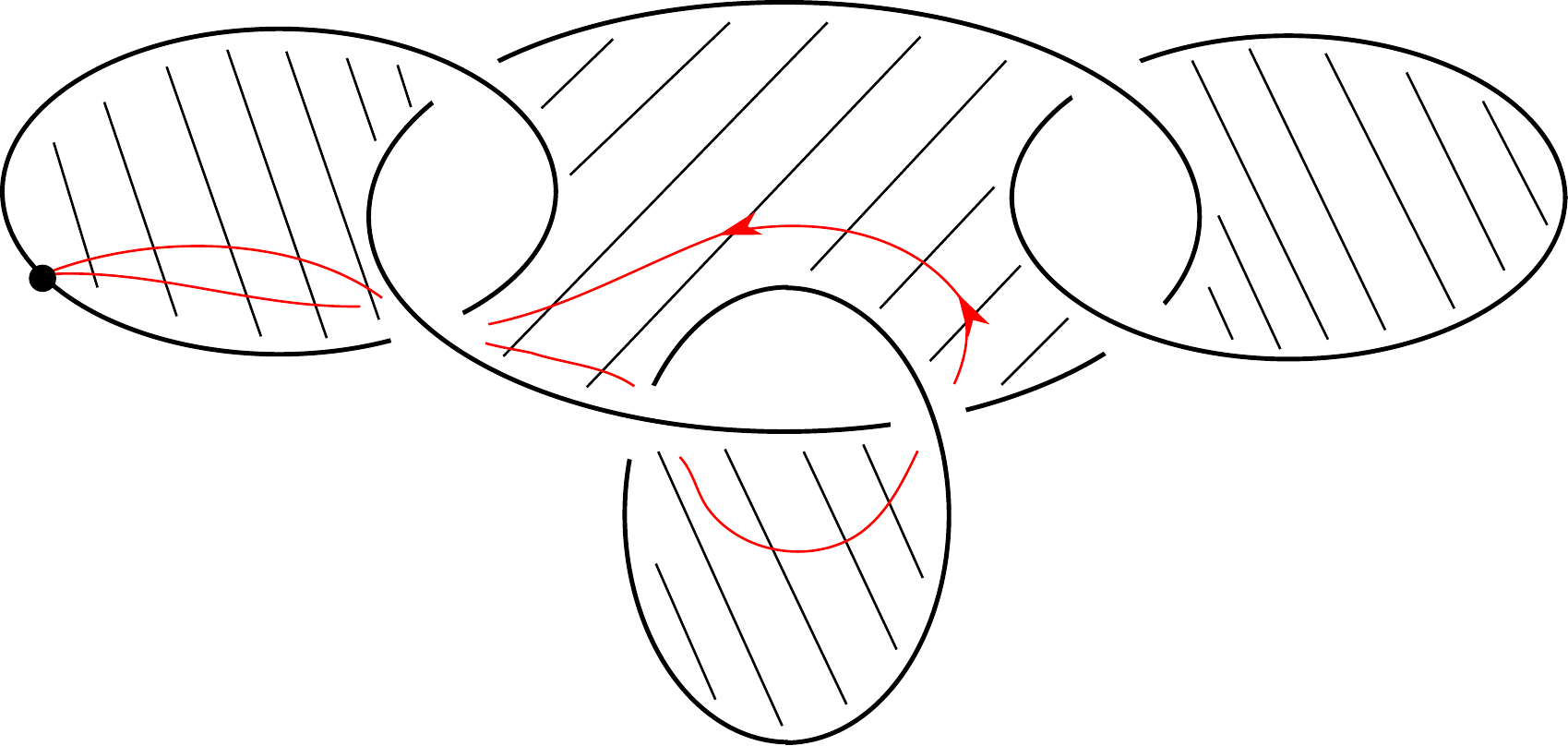}
\put(0,45){$q$}
\end{overpic}
\caption{The loop $\beta$ on $S$.}
\label{fig_beta}
\vspace{\baselineskip}

\begin{overpic}[width=0.5\textwidth]{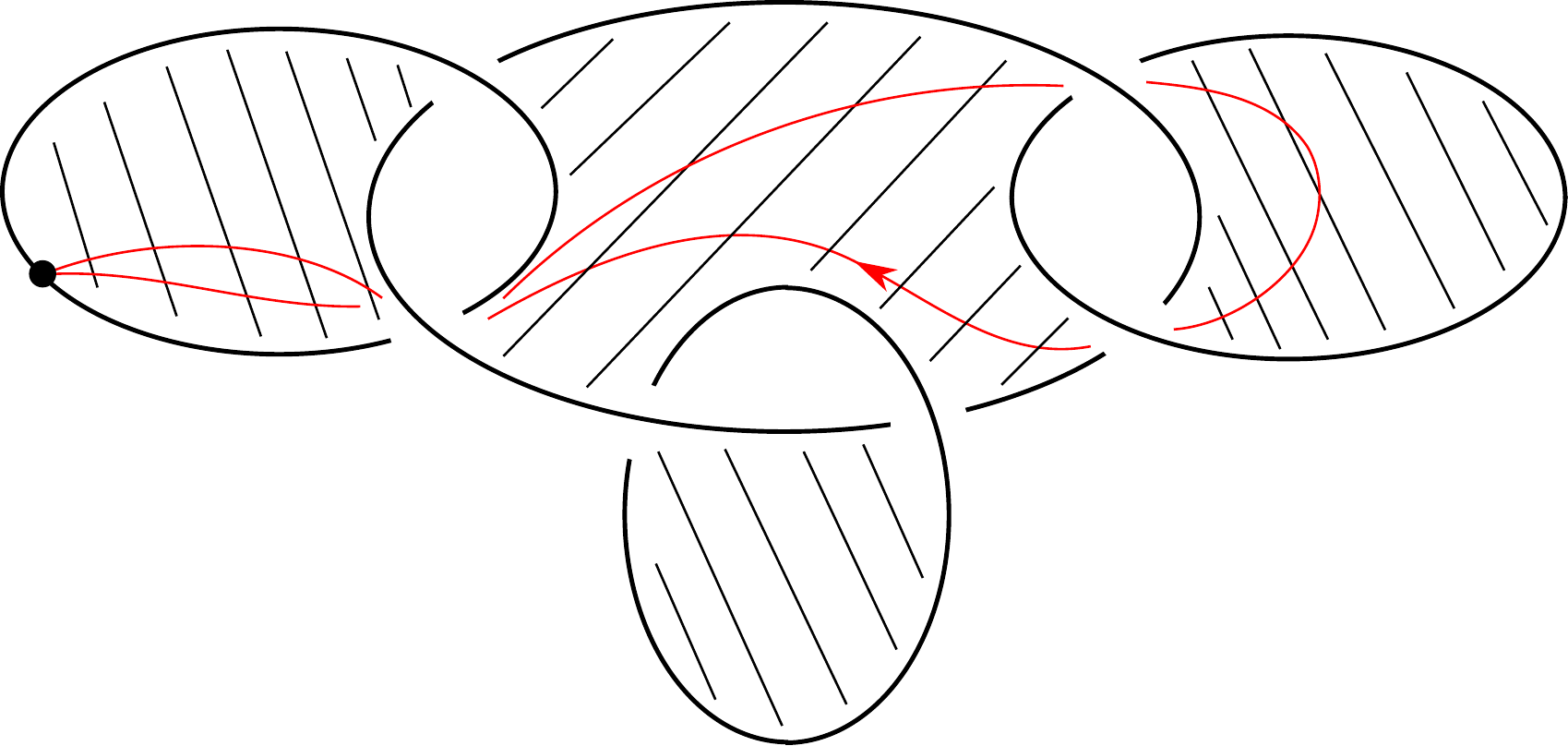}
\put(0,45){$q$}
\end{overpic}
\caption{The loop $\gamma$ on $S$.}
\label{fig_gamma}
\end{figure}

Let $q:= \partial \hat\theta_1 \cap K_3$. Define $\alpha,\beta,\gamma\in\pi_1(S,q)$ by Figures \ref{fig_alpha}, \ref{fig_beta}, \ref{fig_gamma} respectively. Then $\pi_1(S,q)$ is a free group generated by $\alpha,\beta,\gamma$. 

Let $q'\in S$ be a point in the interior of $S$ close to $q$. We also require that $q'$ is a point on the arc $\hat\theta_1$. The arc from $p$ to $q'$ as shown in Figure \ref{fig_arc_q_p} induces an isomorphism
\begin{equation}
\label{eqn_fund_gp_psi_1_q'_p}
	\psi_1 : \pi_1(S^3-L',p) \to \pi_1(S^3-L',q').
\end{equation}
Similarly, a short arc on $S$ from $q$ to $q'$ induces an isomorphism
\begin{equation}
\label{eqn_fund_gp_psi_2_q'_q}
	\psi_2 : \pi_1(S,q) \to \pi_1(S,q').
\end{equation}

Define a map $\phi$ from $\pi_1(S,q)$ to $\pi_1(S^3-L',p)$ by
$$
\phi := \psi_1^{-1} \circ \psi_2,
$$
then we have
$$\phi(\alpha) = a\cdot d, \quad \phi(\beta) = b\cdot d, \quad \phi(\gamma) = c\cdot d.$$

\begin{figure}
\begin{overpic}[width=0.5\textwidth]{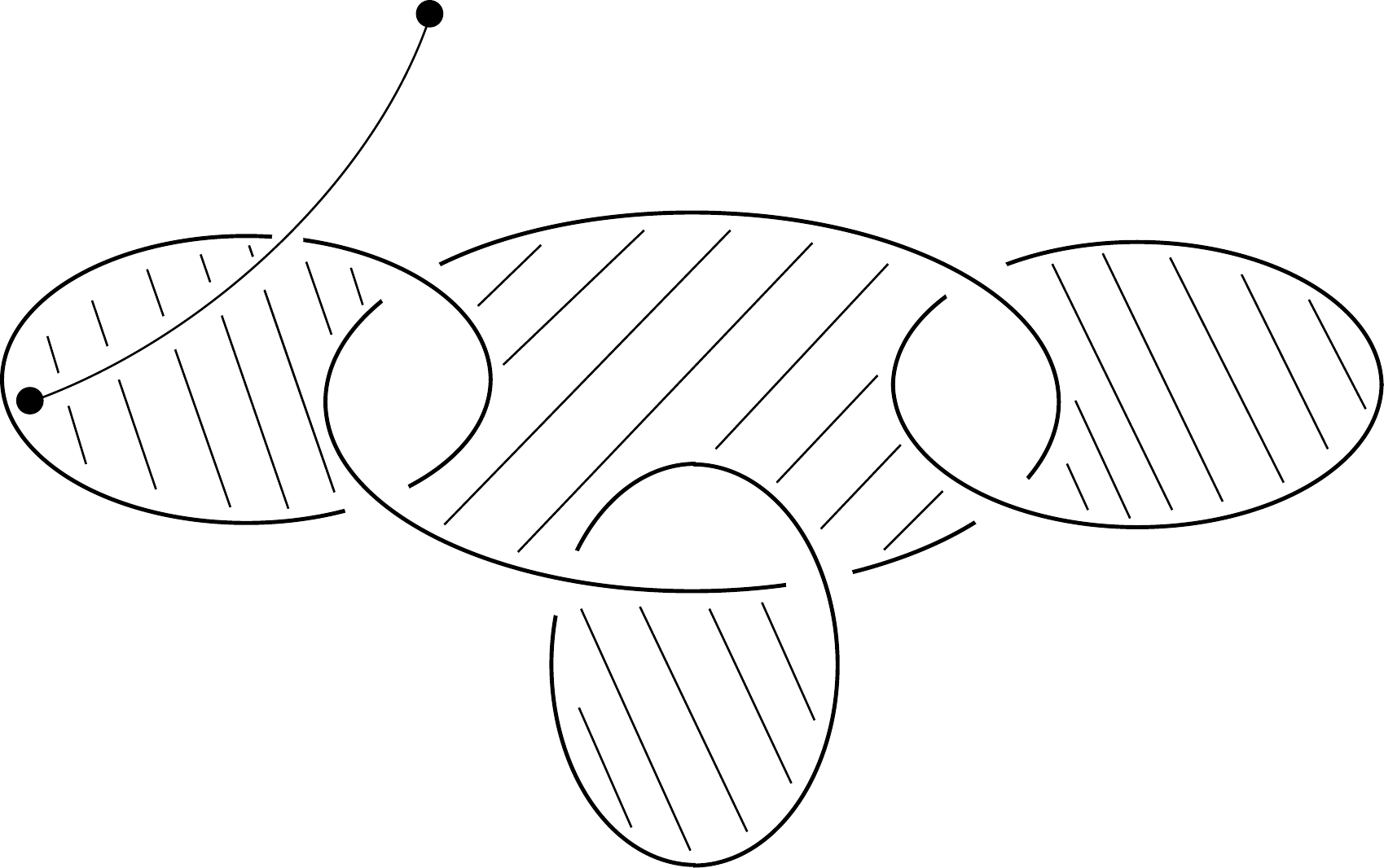}
\put(-10,50){$q'$}
\put(45,110){$p$}
\end{overpic}
\caption{The arc from $p$ to $q'\in S$.}
\label{fig_arc_q_p}
\end{figure}

Let $$\eta\in \pi_1(S,q)$$ be the element represented by $\theta_1\circ \hat{\theta}_1^{-1}$. Let $u$ be a loop in the interior of $S$ based at $q'$ representing $\psi_2(\eta)\in \pi_1(S,q')$, let $v_1$ be the arc from $q$ to $q'$ along $\hat \theta_1$, let $v_2$ be the arc from $q'$ to the other endpoint of $\hat\theta_1$ along $\hat \theta_1$. Then $\theta_1$ is homotopic to $v_1\circ u \circ v_2$ on $S$ (relative to $\partial S$).

The homotopy class of $K_1$ defines a conjugacy class in $\pi_1(S^3-L',p)$.
By the previous argument, the conjugacy class of $K_1$, which is isotopic to $L(\theta_1)$,  is given by
$a^{\epsilon_1} \phi(\eta) b^{\epsilon_2} \phi(\eta)^{-1}$ for some $\epsilon_i\in \{1,-1\}$.

On the other hand, by \eqref{eqn_describe_isotopy_class_of_K1}, the conjugacy class of $K_1$ in $\pi_1(S^3-L',p)$ is given by $a^{\epsilon_3} b^{\epsilon_4}$ for some $\epsilon_i\in \{1,-1\}$. Therefore, $
a^{\epsilon_1} \phi(\eta) b^{\epsilon_2} \phi(\eta)^{-1}
$ is conjugate to $a^{\epsilon_3} b^{\epsilon_4}$ in $\pi_1(S^3-L',p)$.

Let $F_3$ be the free group generated by $a,b,c$. Consider the homomorphism
$$
\psi: \pi_1(S^3-L') \to F_3
$$
defined by 
$$\psi(a)=a,\quad, \psi(b)=b,\quad \psi(c)=c, \quad \psi(d) =1.$$
Then $a^{\epsilon_1} \psi(\phi(\eta)) b^{\epsilon_2} \psi(\phi(\eta))^{-1}$ is conjugate to $a^{\epsilon_3} b^{\epsilon_4}$ in $F_3$. By Lemma \ref{lem_free_group_3}, $\psi(\phi(\eta))$ is contained in the subgroup of $F_3$ generated by $a,b$, so $\eta$ is contained in the subgroup of $\pi_1(S,q)$ generated by $\alpha, \beta$.  

Let $\delta$ be the arc on $S$ as shown in Figure \ref{fig_arc_disjoint_from_theta1}.  Since $\eta$ is contained in the subgroup of $\pi_1(S,q)$ generated by $\alpha, \beta$, the arc $\theta_1$ can be \emph{homotoped} on $S$ (relative to $\partial S$) so that it is disjoint from $\delta$. By Lemma \ref{lem_disjoint_homotopy_implies_isotopy}, the arc $\theta_1$ can be \emph{isotoped} on $S$ (relative to $\partial S$) so that it is disjoint from $\delta$.  

\begin{figure}
\includegraphics[width=0.5\textwidth]{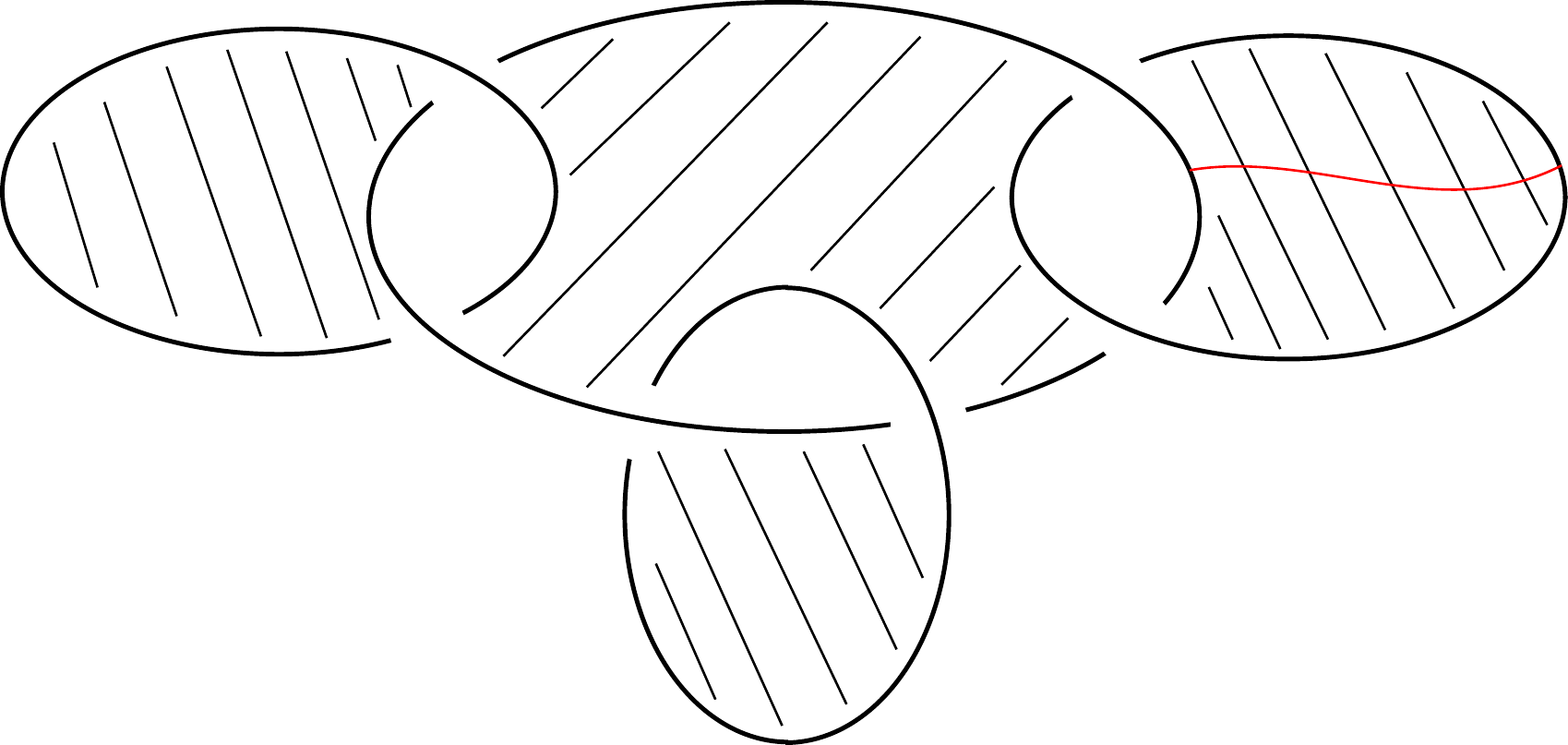}	
\caption{The arc $\delta$ on $S$.}
\label{fig_arc_disjoint_from_theta1}
\end{figure}

For $i=3,\cdots,6$, let $f_i:S\to S$ be the Dehn twist of $S$ along a simple closed curve parallel to $K_i\subset \partial S$. The orientations of the Dehn twists are not important here and we can choose them arbitrarily. 

Cutting $S$ open along $\delta$ yields genus-zero surface with three boundary components. Therefore by \cite[Lemma 9.2]{XZ:forest}, the arc $\theta_1$ is isotopic to $f_3^{u}f_6^{v} (\hat\theta_1)$ on $S$ (relative to $\partial S$) for some $u,v\in\bZ$.

Similarly, the arc $\theta_2$ is isotopic to $f_5^{k}f_6^{l}(\hat\theta_2)$ on $S$ (relative to $\partial S$) for some $k,l\in\bZ$.

Since $\theta_1$ and $\theta_2$ are disjoint, we must have $l=v$.  Therefore by Lemma \ref{lem_L(theta)_under_isotopy}, $L(\theta_1\cup\theta_2)$ is isotopic to $L(\hat\theta_1\cup\hat\theta_2)$ in $S^3-L'$.  As a consequence, the link $L$, which is isotopic to $L'\cup L(\theta_1\cup\theta_2)$, is given by Figure \ref{fig_link_G0}.

\begin{figure}
\includegraphics[width=0.5\textwidth]{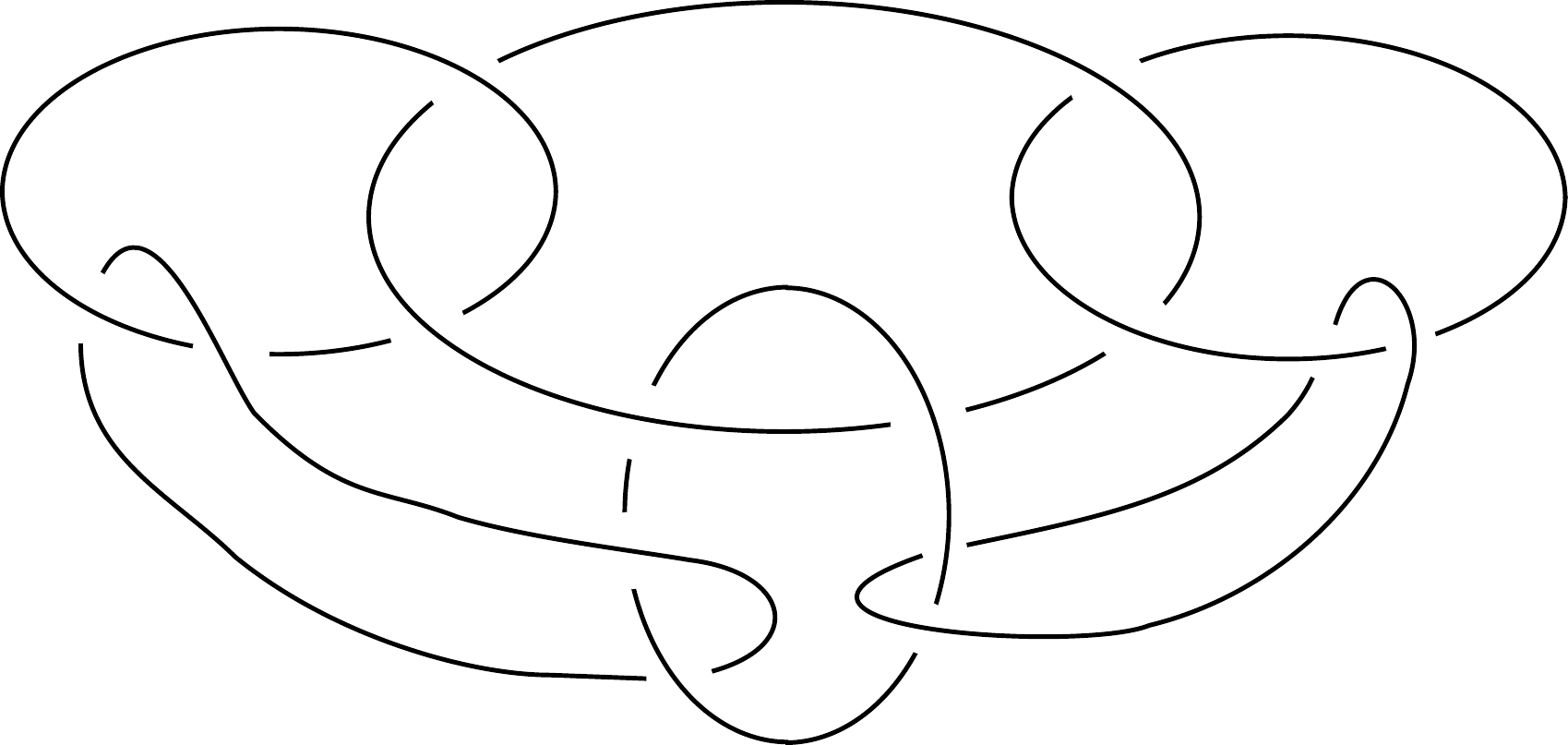}  
\caption{The isotopy class of $L$.}
\label{fig_link_G0}
\end{figure}

\begin{proof}[Proof of Proposition \ref{prop_no_G0}]
Recall that $L$ is a hypothetical 6-component link that has minimal $\II^\natural$ and has linking graph isomorphic to $G_0$. By the arguments above, $L$ has to be isotopic to the link shown in Figure \ref{fig_link_G0}.  By Lemma \ref{lem_lower_bound_I_G0}, this link does not have minimal $\II^\natural$, which yields a contradiction.
\end{proof}

\subsection{Proof of Lemma \ref{lem_free_group_3}}
\label{subsec_G0_lemma_free_group}

Let $F_3$ be the free group generated by $a,b,c$. 

\begin{Definition}
A \emph{word} is a sequence $(u_1,\cdots,u_N)$, where 
$$u_i\in \{a,b,c,a^{-1},b^{-1},c^{-1}\}\subset F_3.$$  The word $(u_1,\cdots,u_N)$ is called \emph{reduced}, if for all $i\in\{1,\cdots,N-1\}$, we have $u_i\neq u_{i+1}^{-1}$. 
\end{Definition}

Every word $(u_1,\cdots,u_N)$ represents an element $u_1 u_2 \cdots u_N$ in $F_3$, and every element in $F_3$ is represented by a unique reduced word.  For $w\in F_3$, we use $\len(w)$ to denote the length of the reduced word representing $w$.

\begin{Lemma}
\label{lem_centralizer_a}
The centralizer of $a$ in $F_3$ is generated by $a$.
\end{Lemma}
\begin{proof}
	Suppose $x\in F_3$ satisfies $[x,a]=1$. Then we can write $x$ as $x=x'\cdot x''$, where
	\begin{enumerate}
\item $x''$ is a power of $a$,
\item $\len(x) = \len(x') + \len(x'')$,
\item The reduced word representing $x'$ does not end with $a$ or $a^{-1}$.
\end{enumerate}
By the assumptions, we have $[x',a]=[x,a]=1$, and hence 
\begin{equation}
\label{eqn_[x',a]=1}
	x'\cdot a\cdot (x')^{-1} =a.
\end{equation} 
Let $(u_1,\cdots,u_N)$ be the reduced word representing $x'$. Then 
$$(u_1,\cdots,u_N,a,u_N^{-1},\cdots,u_1^{-1})$$ is a reduced word.  By \eqref{eqn_[x',a]=1}, we must have $N=0$, and hence the desired result is proved.
\end{proof}

\begin{Corollary}
\label{cor_centralizer_of_ab}
	Suppose $\epsilon_1,\epsilon_2\in\{1,-1\}$ and $x=a^{\epsilon_1}b^{\epsilon_2}$ or $b^{\epsilon_1}a^{\epsilon_2}$. Then the centralizer of $x$ in $F_3$ is generated by $x$. 
\end{Corollary}
\begin{proof}
	$F_3$ is a free group generated by $x,b,c$, thus the desired result follows from Lemma \ref{lem_centralizer_a}.
\end{proof}

\begin{proof}[Proof of Lemma \ref{lem_free_group_3}]
Let $w\in F_3$, and suppose there exist $\epsilon_i\in\{1,-1\}$ for $i=1,\cdots,4$ such that
$a^{\epsilon_1}wb^{\epsilon_2} w^{-1}$ is conjugate to $a^{\epsilon_3}b^{\epsilon_4}$ in $F_3$.
By comparing the images in the abelianization of $F_3$, we have $\epsilon_1 = \epsilon_3$ and $\epsilon_2 = \epsilon_4$. By changing $a,b$ to $a^{\pm 1},b^{\pm1}$, we may assume without loss of generality that $\epsilon_1 = \epsilon_3=-1$ and $\epsilon_2 = \epsilon_4=1$.
By the assumption, there exists $r\in F_3$, such that
\begin{equation}
\label{eqn_awbw}
a^{-1}wbw^{-1} = ra^{-1}br^{-1}.
\end{equation}
We show that both $w$ and $r$ must be included in the subgroup generated by $a,b$.
By \eqref{eqn_awbw}, we have
$$
wbw^{-1}rb^{-1} = ara^{-1},
$$
which implies
$$
(r^{-1}w) b (w^{-1} r) b^{-1}  = r^{-1} a r a^{-1},
$$
therefore
$$
a r^{-1} a^{-1} r = b (r^{-1}w) b^{-1} (w^{-1} r).
$$

Let $x := r^{-1}, y := r^{-1}w$, we study the equation 
\begin{equation}
\label{eqn_axby}
	axa^{-1} x^{-1} = byb^{-1} y^{-1}.
\end{equation}
Since $w=x^{-1} y, r=x^{-1}$, the desired result would follow if both $x$ and $y$ are in the subgroup of $F_3$ generated by $a,b$.

Write $x = x' \cdot x''$, such that
\begin{enumerate}
\item $x''$ is a power of $a$,
\item $\len(x) = \len(x') + \len(x'')$,
\item the reduced word representing $x'$ does not end with $a$ or $a^{-1}$.
\end{enumerate}
Similarly, write $y = y' \cdot y''$, such that
\begin{enumerate}
\item $y''$ is a power of $b$,
\item $\len(y) = \len(y') + \len(y'')$,
\item the reduced word representing $y'$ does not end with $b$ or $b^{-1}$.
\end{enumerate}
Then we have
$$
axa^{-1}x^{-1} = a x' a^{-1} (x')^{-1}, \quad b y b^{-1} y^{-1} = b y' b^{-1} (y')^{-1}.
$$
By \eqref{eqn_axby},
\begin{equation}
\label{eqn}
a x' a^{-1} x'^{-1} = b y' b^{-1} y'^{-1},
\end{equation}
and the definitions of $x'$ and $y'$ imply
$$\len(a x' a^{-1} x'^{-1}) - 2\len(x')= 0 \text{ or }2,$$
$$\len(b y' b^{-1} y'^{-1}) - 2\len(y') = 0\text{ or }2.$$  
Without loss of generality, we assume $\len(x')\ge \len(y')$, then 
$$\len(x')- \len(y')=0 \text{ or } 1.$$ We discuss three cases:\\

Case 1: $\len(y')=0$.  In this case, we have $y'=1$ and $[x',a]=1$, and the desired result follows from Lemma \ref{lem_centralizer_a}.\\

Case 2: $\len(x') = \len(y')$.  In this case, \eqref{eqn} implies
$$
ax'a^{-1} = by'b^{-1},\quad (x')^{-1} = (y')^{-1}.
$$
Therefore $x'=y'$, and they are both commutative with $b^{-1} a$. By Corollary \ref{cor_centralizer_of_ab}, we have $x'=y' = (b^{-1} a)^k$ for some integer $k$, and the desired result follows.\\

Case 3: $\len(x') - 1 = \len(y')\ge 1$. Then \eqref{eqn} implies
$$
ax'a^{-1} = by', \quad (x')^{-1} = b^{-1} (y')^{-1},
$$
thus
$$
x' = y' b, \quad ay'ba^{-1}= by',
$$
therefore $[ay', ba^{-1}]=1.$ By Corollary \ref{cor_centralizer_of_ab}, we have
$$
y'=a^{-1}(ba^{-1})^k, \quad x' = y'b =(a^{-1}b)^{k+1}
$$
for some integer $k$, and the desired result follows.
\end{proof}

%% file: combinatorics.tex
This section proves Proposition \ref{prop_graph}.    Recall that the \emph{restriction} of a graph $G$ is defined by Definition \ref{def_restriction_graph}, and $G_0$ is the graph given by Figure \ref{fig_G0_graph}.
We reformulate the statement of Proposition \ref{prop_graph} as follows.
\begin{repProposition}{prop_graph} Suppose $G$ is a finite simple graph with vertex set $V$ that satisfies the following conditions:
\begin{enumerate}
\item $G$ is not a tree,
\item no restriction of $G$ is a cycle with order $n$, where $n\ge 3$ and $n \neq 4$,
\item no restriction of $G$ is isomorphic to $G_0$.
\end{enumerate}
Then there exists a non-constant 
map $\varphi: V\to \{\bfi,\bfj,\bfk\}$, such that 
for every $v\in V$, the image $\varphi(v)$ is commutative to 
$$
\prod_{\{w|w \text{ is adjacent to }v\}} \varphi(w). 
$$
\end{repProposition}

\begin{proof}
We use induction on the number of vertices. The statement is obvious if $G$ is not connected.  We will assume that $G$ is connected from now.

By Condition (2), the graph $G$ does not contain any cycle of odd length, therefore it is a bipartite graph.  Write $V=V_1\sqcup V_2$ such that there is no edge within $V_1$ or within $V_2$.

Let 
$$E:=\{(x,y)\in V\times V| x \text{ and } y \text{ are adjacent in } G\}.$$ 

If $V'$ is a subset of $V$, we use 
$N(V')$ to denote the set of $x\in V$ such that there exists a $y\in V'$ with $(x,y)\in E$. For $v\in V$, we also use $N(v)$ to denote $N(\{v\})$.

Take $x_0,y_0\in V_1$ to be a pair of distinct vertices such that 
$\# \, N(x_0)\cap N(y_0)$ is maximal among all pairs in $V_1$. Let 
\begin{align*}
A &:= N(x_0)\cap N(y_0),\\
B &:= \{v\in V_1:\, A\subseteq N(v)\},\\
C &:= N(B)-A.
\end{align*}
Then $B\subseteq V_1$ and $A,C\subseteq V_2$. Every vertex in $B$ is adjacent to all vertices in $A$. Since $\# \, N(x_0)\cap N(y_0)$ is maximal, every vertex in $C$ is adjacent to a \emph{unique} vertex in $B$.

By the assumptions, $G$ is connected and is not a tree, therefore $G$ contains at least one cycle. By Condition (2), there exists a restriction of $G$ which is isomorphic to a cycle of order 4. Therefore, both $A$ and $B$ contain at least two elements.

We prove the following lemma.

\begin{Lemma}
\label{lem_graph_A_in_N(x)}
	Suppose there exists $x\in B$ such that $A \subsetneqq N(x)$, 
then the restriction of $G$ to $V-\{x\}$ is disconnected.
\end{Lemma} 

\begin{proof}  Assume the contrary. So for each $y\in N(x)-A$ and $z\in A$, there exists a path from $y$ to $z$ without passing through $x$.  Choose a shortest path $(y,v_1,\cdots, v_k, z)$ among all possible choices of $y$ and $z$. Since $V$ is bipartite, $k$ is odd.  

Since the path $(y,v_1,\cdots, v_k, z)$ is shortest, the restriction of $G$ to $\{y,v_1,\cdots,v_k, z\}$ does not contain any additional edge other than the edges on the path, and 
\begin{equation}
\label{eqn_v_i_disjoint_A}
	\{v_1,\cdots, v_k\}\cap A = \emptyset.
\end{equation} 

Let $G'$ be the restriction of $G$ to $\{x,y,v_1,\cdots,v_k,z\}$. Notice that $(x,y),(x,z)\in E$.
By Condition (2), no restriction of $G'$ is a cycle of order $\neq 4$. Since $G'$ is a bipartite graph, we must have 
\begin{equation}
\label{eqn_x_adjacent_to_vi}
	(x,v_i)\in E \quad \text{for all even } i.
\end{equation}

If $k\ge 3$, then by \eqref{eqn_v_i_disjoint_A} and  \eqref{eqn_x_adjacent_to_vi}, we have $v_2\in N(x)-A$, so the path $(v_2,\cdots,v_k,z)$ is a shorter path connecting $N(x)-A$ to $A$ without passing through $x$, contradicting the assumptions.  Therefore $k=1$. 

Since $x$ is the unique element in $B$ adjacent to $y$ and $v_1\neq x$, we have $v_1\notin B$. Hence there exists $z_1\in A$ such that $(z_1,v_1)\notin E$. Choose $x_1\in B$ such that $x_1\neq x$, then the restriction of $G$ to $\{x,x_1,v_1,z,z_1,y\}$ is isomorphic to $G_0$, which contradicts Condition (3).
\end{proof}

Back to the proof of Proposition \ref{prop_graph}. We discuss three cases:\\

{\bf Case 1:} There exists a vertex in $G$ with degree no greater than 1. 

In this case, the desired result follows immediately from the induction hypothesis.\\

{\bf Case 2:}
There exists $x\in B$ such that $A \subsetneqq N(x)$.  

By Lemma \ref{lem_graph_A_in_N(x)}, the graph $G$ becomes disconnected after removing the vertex $x$. Therefore, there exists finitely many graphs $G^{(1)},\cdots, G^{(k)}$ with $k\ge 2$, such that the following statements hold:
\begin{enumerate}
	\item  each $G^{(i)}$ has at least two vertices and satisfies Conditions (2) and (3);
	\item for each $i$, there is a vertex $v^{(i)}$ of $G^{(i)}$, such that $G$ is isomorphic to the quotient graph of $\sqcup G^{(i)}$ after identifying all $v^{(i)}$'s to one vertex.
\end{enumerate}
Recall that $V$ denotes the vertex set of $G$. 
Let $V^{(i)}$ be the vertex set of $G^{(i)}$.  We view $V^{(i)}$ as a subset of $V$ using the isomorphisms in Statement (2) above. 

Since $G$ is not a tree, there exists $i\in\{1,\cdots,k\}$ such that $G^{(i)}$ is not a tree. Let $\varphi_{i}: V^{(i)}\to \{\bfi,\bfj,\bfk\}$ be the map given by the induction hypothesis on $G^{(i)}$. Extending $\varphi_{i}$ to $V-V^{(i)}$ by the constant value $\varphi(v^{(i)})$ yields the desired map $\varphi$. \\

{\bf Case 3:}
$A = N(x)$ for all $x\in B$, and all vertices of $G$ have degrees at least 2.

In this case, we have $C=\emptyset$.
Recall that $A$ and $B$ both have at least two elements.
Let $x_1,x_2$ be two distinct elements of $B$, and let $G'$ be the restriction of $G$ to $V-\{x_1,x_2\}$. We discuss two sub-cases depending on whether $G'$ is a tree:

{\bf Case 3.1:}  If $G'$ is not a tree, let $\varphi':V-\{x_1,x_2\}\to\{\bfi,\bfj,\bfk\}$ be the map obtained from the induction hypothesis on $G'$. Extending $\varphi'$ to $V$ by taking $\varphi(x_1)=\varphi(x_2)$ and letting them be an element commutative to 
$$
\prod_{v\in A} \varphi'(v)
$$
yields the desired map $\varphi$.

{\bf Case 3.2:} If $G'$ is a tree, then $\#B = 2$ or $3$.  If $\#B = 3$, then the assumption that all vertices of $G$ have degrees at least $2$ implies that $B=V_1$ and $A=V_2$. If $\#B=2$ and $G'$ is a tree, then Condition (2) and the assumption that all vertices of $G$ have degrees at least $2$ imply that every vertex in $V-A-B$ is adjacent to at least two vertics in $A$. It then follows from Condition (3) and the assumption that $G'$ is a tree that if $V-A-B$ is non-empty, then it contains exactly one element that is adjacent to all the elements of $A$. By the definition of $B$, this implies $V-A-B=\emptyset$.

Therefore, we conclude that $V$ is a complete bipartite graph with $\# V_1\ge 2$ and $\# V_2\ge 2$. In this case, one can construct a non-constant map $\varphi:V\to\{\bfi,\bfj,\bfk\}$ directly by requiring
$$
\prod_{v\in V_1} \varphi(v) = \pm 1, \quad \prod_{v\in V_2} \varphi(v) =\pm 1. 
\phantom\qedhere\makeatletter\displaymath@qed
$$
\end{proof}

%% file: appendix.tex
The computations in this section rely on the Mathemtica package ``KnotThoery'' from The Knot Atlas \cite{knotatlas}.  In particular, we use a program written by Jana Archibald that computes the multivariable Alexander polynomial.

\begin{Lemma}
\label{lem_lower_bound_I_L5-3}
	Let $L_{5,-3}$ be the link defined in Section \ref{subsec_isotopy_L_cycle}, and suppose $p\in L_{5,-3}$. We have 
	$$
	\dim_\bC \II^\natural(L_{5,-3},p;\bC)> 16.
	$$
\end{Lemma}

\begin{figure}
	\includegraphics[width=0.9\textwidth]{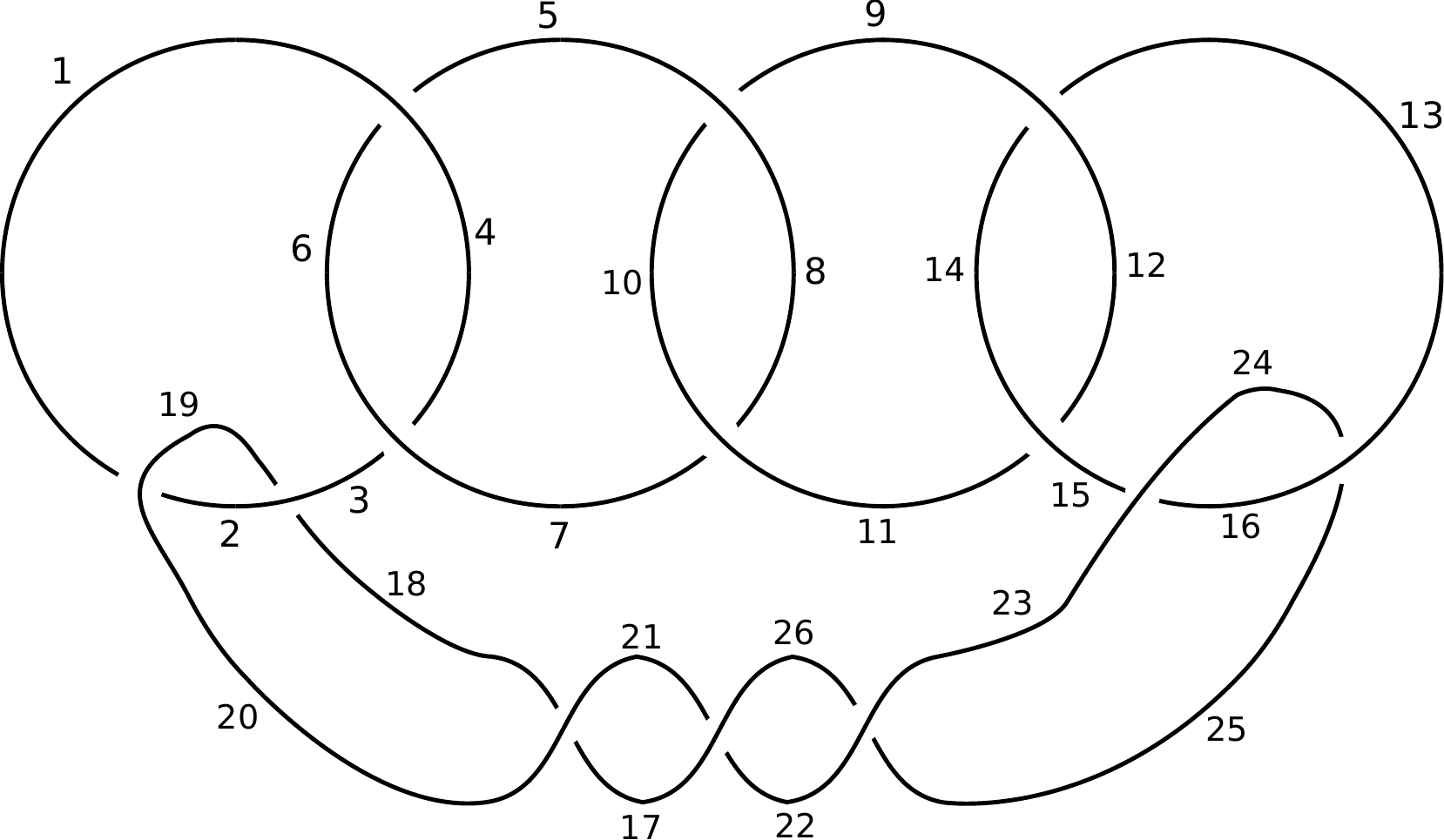}	
	\caption{A planar diagram of $L_{5,-3}$.}
	\label{fig_L5-3}
\end{figure}

\begin{proof}
	We apply Proposition \ref{prop_I>Alexander} to the link $L_{5,-3}$. For the link $L_{5,-3}$, the right-hand side of \eqref{eqn_I>Alexander} is equal to $300$. This is computed using the planar diagram in Figure \ref{fig_L5-3} and the following Mathematica code:\footnote{The code was executed in Mathematica 12.1.1.0. For the set up of the KnotTheory package, see \url{http://katlas.org/wiki/Setup}. The package we used was retrieved on April 2, 2021.}
\begin{lstlisting}
	<< KnotTheory`;
	L = PD[X[5,1,6,4],X[1,20,2,19],X[18,3,19,2],
		X[3,7,4,6],X[9,5,10,8],X[7,11,8,10],
		X[13,9,14,12],X[11,15,12,14],X[15,23,16,24],
		X[24,16,25,13],X[25,23,26,22],X[21,17,22,26],
		X[17,21,18,20]];
	mva = MultivariableAlexander[L][x]/.{x[1]->x1,
		x[2]->x2,x[3]->x3,x[4]->x4,x[5]->x5};
	f = mva*Sqrt[x1]*Sqrt[x2]*Sqrt[x3]*Sqrt[x4]*
		 Sqrt[x5]; (*remove the denominator from mva*)
	Total[Abs/@Flatten[CoefficientList[Expand
		[f*(x1-1)*(x2-1)*(x3-1)*(x4-1)*(x5-1)],
		{x1,x2,x3,x4,x5}]]]
	 \end{lstlisting}	 
Therefore, we have
	$$\dim  \II^\natural (L,p;\bC) \ge \ceil*{\frac{300}{16}}=19, $$
	and the result is proved. 
\end{proof}

\begin{Lemma}
\label{lem_lower_bound_I_G0}
	Let $L$ be the link given by Figure \ref{fig_link_G0}, suppose $p\in L$. Then 
	$$
	\dim_\bC \II^\natural (L,p;\bC) > 32.
	$$ 
\end{Lemma}

\begin{figure}
	\includegraphics[width=0.9\textwidth]{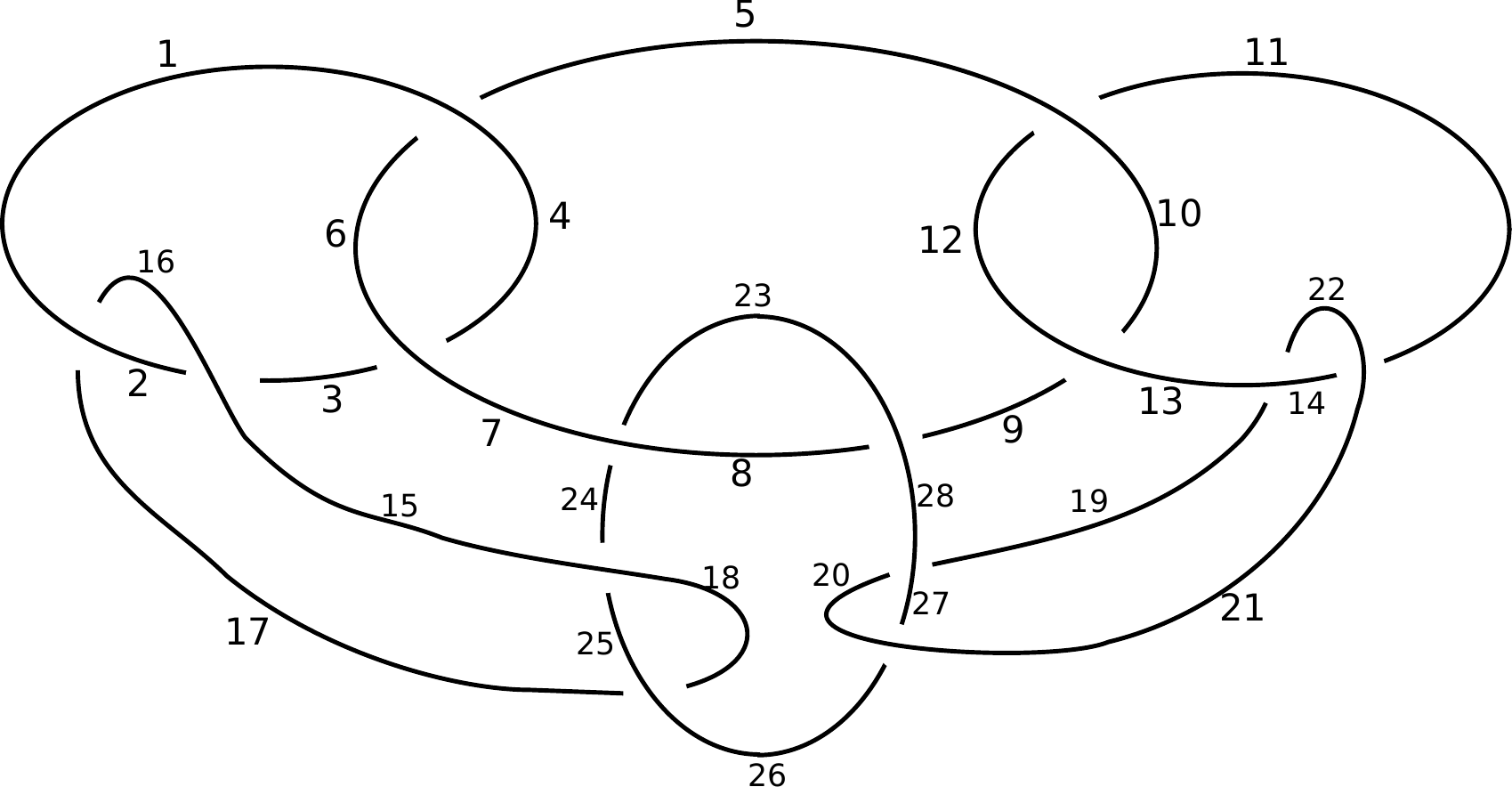}	
	\caption{A planar diagram of $L$.}
	\label{fig_link_G0_PD}
	\end{figure}

\begin{proof}
	We apply Proposition \ref{prop_I>Alexander} to the link $L$. For the link $L$, the right-hand side of \eqref{eqn_I>Alexander} is equal to 1216. This is computed using the planar diagram in Figure \ref{fig_link_G0_PD} and the following Mathematica code:
\begin{lstlisting}
	<< KnotTheory`;
	L = PD[X[5,1,6,4],X[3,7,4,6],X[16,1,17,2],
		X[2,15,3,16],X[23,7,24,8],X[8,28,9,23],
		X[9,13,10,12],X[11,5,12,10],X[22,13,19,14],
		X[14,21,11,22],X[19,28,20,27],X[26,21,27,20],
		X[24,15,25,18],X[17,26,18,25]];
	mva = MultivariableAlexander[L][x]/.{x[1]->x1,
		x[2]->x2,x[3]->x3,x[4]->x4,x[5]->x5,x[6]->x6};
	f = mva*Sqrt[x1]*Sqrt[x2]*Sqrt[x3]*Sqrt[x4]*
		 x5*x6; (*remove the denominator from mva*)
	Total[Abs/@Flatten[CoefficientList[Expand
		[f*(x1-1)*(x2-1)*(x3-1)*(x4-1)*(x5-1)*(x6-1)],
		{x1,x2,x3,x4,x5,x6}]]]
	 \end{lstlisting}	 
Therefore, we have
	$$\dim  \II^\natural (L,p;\bC) \ge \frac{1216}{32}=38, $$
	and the result is proved. 
\end{proof}